\documentclass[11pt,a4paper,reqno]{amsart}

\usepackage[utf8]{inputenc}
\usepackage[english]{babel}

\usepackage[a4paper, top=3cm, bottom=3cm, left=2cm, right=2cm]{geometry}

\usepackage{bbm,amsmath,amsthm,epsfig,latexsym,marvosym}
\usepackage{amsfonts}
\usepackage{bigints}
\usepackage{amssymb}
\usepackage{subfig} 

\allowdisplaybreaks

\usepackage{esint,xfrac}
\usepackage[colorlinks=true,linkcolor=blue,citecolor=red]{hyperref}

\def\A{\mathbb A}

\def\R{\mathbb R}
\def\N{\mathbb N}
\def\C{\mathbb{C}}
\def\BD{\mathrm{BD}}
\def\BDD{\mathrm{BD}_{\mathrm{dev}}}

\def\BV{\mathrm{BV}}

\def\cal{\mathcal}
\def\Ao{{\cal A}}

\def\E{{\cal E}}
\def\F{{\cal F}}

\def\H{{\cal H}}

\def\L{{\cal L}}

\def\II{{\mathrm{Id}}}
\def\HH{{\mathbf H}}
\def\AA{{\mathbb A}}

\def\a{\alpha}
\def\b{\beta}

\def\e{\varepsilon}

\def\ss{\mathrm{s}}
\def\bb{\mathrm{b}}
\def\Aa{\mathrm{A}}

\def\div{\mathrm{div}}
\def\tr{\mathrm{tr}}

\def\Om{\Omega}

{\left\lbrace\begin{array}{@{}l@{}}}%
{\end{array}\right.}

\def\trace{{\rm tr}\,}
\def\wt{{\rightharpoonup^*}\,}
\def\Id{{\rm Id}\,}
\def\loc{{\rm loc}}

\def\d{\, \mathrm{d}}

\def\curl{{\rm curl}}

\def\00{{\bf 0}}
\def\dive{{\rm div}}

\newcommand{\restr}{%
  \,\raisebox{-.127ex}{\reflectbox{\rotatebox[origin=br]{-90}{$\lnot$}}}\,%
}

\DeclareMathOperator*{\spt}{spt}
\DeclareMathOperator*{\wlim}{\textrm{w*-lim}}
\DeclareMathOperator*{\diam}{diam}

\newtheorem{theorem}{Theorem}[section]
\newtheorem{corollary}[theorem]{Corollary}
\newtheorem{proposition}[theorem]{Proposition}
\newtheorem{lemma}[theorem]{Lemma}
\newtheorem{conjecture}[theorem]{Conjecture}

\theoremstyle{definition}
\newtheorem{remark}[theorem]{Remark}

\numberwithin{equation}{section}
\numberwithin{figure}{section}

\usepackage{color}

\title[Rigidity and functional properties of $\BDD(\Omega)$]{Rigidity and functional properties of $\BDD(\Omega)$ }

\author{Marco Caroccia}
\address{Politecnico di Milano, Dipartimento di matematica}

\email{marco.caroccia@polimi.it}

\author{Nicolas Van Goethem}
\address{Centro de Matemática e Aplicações Fundamentais, Universidade de Lisboa}

\email{vangoeth@fc.ul.pt}

\makeindex

\begin{document}
\maketitle

\begin{abstract}
We provide a structural analysis of the space of functions of bounded deviatoric deformation, $\BDD$, which arises in models of plasticity and fluid mechanics. The main result is the identification of the annihilator and a rigidity theorem for $\BDD$ maps with constant polar vector in the wave cone characterizing the structure of singularities for such maps. This result, together with an explicit kernel projection operator, enables an iterative blow-up procedure for relaxation and homogenization problems, allowing for integrands with explicit dependence on $u$ as well as $\E_d u$. Our approach overcomes several difficulties compared to the $\BD$ case, in particular due to the lack of invariance of $\E_d$ under orthogonalization of the polar directions. Applications to integral representation and material science are discussed.
\end{abstract}
 
\tableofcontents
\section{Introduction}
Functional spaces have long occupied a foundational role in modern analysis, furnishing a natural and robust framework for the formulation and resolution of both abstract and applied problems across Mathematics and theoretical Physics. Their power stems from the capacity to encode, within a unified formalism, properties of functions—such as regularity, integrability, and decay—often in a manner that reflects the intrinsic geometry and symmetries of the underlying structures or physical systems under consideration
\\

The importance of choosing an appropriate functional setting becomes evident when dealing with problems characterized by singularities, multi-scale phenomena, or by the presence of differential constraints, as in continuum mechanics, materials science, or image processing. In these contexts, rigidity properties play a central role: they capture the extent to which the imposition of certain differential constraints limits the possible structure of admissible functions, especially at singular points or along lower-dimensional sets where concentrations may occur.\\

Rigidity for functions satisfying some differential properties in the sense of measures is encompassed in the general study of the  fine properties of these functions, and serves as a powerful tool to tackle several problems in Mathematical physics. In the field of Calculus of Variations, rigidity serves in particular to determine blow-ups at some singular points (Cantor points for instance), that are a crucial tool for homogenization and relaxation purposes (see \cite{CVG2025}). In particular, in phase transition problems, image segmentation, material design, or variational frameworks in continuum mechanics involving multi-scale phenomena, field concentration often occurs on lower-dimensional sets, such as cracks and dislocations.. Typical examples of functions with differential constraints are the $\mathrm{BV}$ functions (see \cite{ambrosio2000functions}), whose gradient are defined in the sense of measures, or functions of bounded deformations, denoted as $\BD$ (see \cite{ambrosio1997fine,ebo99b,TS2}), whose symmetric gradient is constrained to be a finite Radon  measure. This latter space is typically of use in linearized elasticity models of fracture mechanics or in infinitesimal elasto-plasticity.  However, in the field of Plasticity or Fluid mechanics (where the shear deformation is the kinematical variable used to define the viscous stresses), shear deformations are favored and the corresponding functional space is the space of bounded deviatoric deformations named as $\BDD$. Note that such a prominent phenomena as crystal plasticity is also  mainly governed by shear efforts due to the formation and motion of dislocations in shear planes, as plasticity is not affected by traction or compression efforts in the material. \\

Despite their significance, the fine properties and structural results for the space $\BDD$ remain much less explored compared to the classical settings of $\BV$ or $\BD$. The aim of the present work is to address this gap by providing a thorough analysis of $\BDD$. The main novelty of this manuscript is the characterization of $\BDD$ maps with constant polar vector lying in the \textit{wave cone}, which is useful in the theory of relaxation and homogenization via iterated blow-ups, as in \cite{caroccia2019integral}, \cite{CVG2025}, \cite{de2016structure}. A notable advantage of this approach to homogenization, based on iterated blow-ups through rigidity and kernel projection, is that it does not restrict the energy to depend solely on $\E_d u$, but also accommodates integrands with explicit dependence on $u$ itself, up to an $L^{\infty}$ bound. 

\subsection{The space of bounded deviatoric deformations}
For a general $(n\times n)$-matrix-valued tensor $T$, two classical orthogonal decompositions apply: 
\[
T=\underset{\mathrm{sym} T }{ \underbrace{\frac{T+T^{t}}{2} }}+\underset{\mathrm{skew} T }{ \underbrace{\frac{T-T^{t}}{2} }}, \ \ \ \ \ \ \ T=\underset{\mathrm{dev}T}{ \underbrace{T - \frac{\tr(T)}{n}\II }}+\frac{\tr T}{n}\II.
\]
The first decomposition is into the symmetric and skew (or anti-symmetric) part. The second decomposition is into the shear part and the volumetric part, with $\tr T$ denoting the trace of the matrix $T$ and $\II$ the $n\times n$ identity matrix. Combining both decompositions, we get the so-called Cartan decomposition of the Lie algebra $\mathfrak{gl}(n)$, i.e.,  $\mathfrak{gl}(n)=\left(\mathfrak{sl}(n)\cap\mathrm{Sym}(n)\right)\oplus\mathfrak{so}(n)\oplus\R\II$ (see e.g., \cite{BNPS2015}), namely 
\[
T=\mathrm{sym}\ T+\mathrm{skew}\ T=\mathrm{dev\ sym}\ T+\mathrm{skew}\ T+\frac{\tr T}{n}\II.
\]
Take now $T=Du$ for $u$ a vector-valued function of bounded variation; one obtains
\[
Du=\E u+\mathrm{skew}\ Du=\mathrm{dev}\ \E u+\mathrm{anti}\ \curl u+\frac{\dive u}{n}\II,
\]
where $\E u:=\mathrm{sym}\ Du=\left(\frac{Du+Du^t}{2}\right)$, $\mathrm{anti}\ \curl u:=\mathrm{skew}\ Du$, thus showing a decomposition of the deformation gradient $Du$ into an (infinitesimal) rotation term, $\mathrm{anti}\ \curl u$, a volumetric term, $\frac{\dive u}{n}\II$, and, eventually, a pure shear term, $\mathrm{dev}\ \E u$. 

It turns out that several issues arising from continuum mechanics, such as fluid mechanics or linearized elasticity (but also in general relativity), involve the deviatoric operator
\[
\E_d u:=\mathrm{dev}\ \E u =\E u- \frac{\dive(u)}{n}\II.
\]
$\E_d u$ is defined, in principle, on functions $u\in C^1(\Omega;\R^n)$; but, as for the larger spaces $\BV$ and $\BD$, it becomes relevant to introduce the space of functions of \textit{bounded deviatoric deformations}:
\[
\BDD(\Omega):=\left\{\left. u\in L^1(\Omega;\R^n) \ \right| \ \E_d u\in \mathcal{M}(\Omega;\mathbb{M}_{\mathrm{sym}_0}^{n\times n})\right\}\subset \BD(\Omega).
\]
where $\mathcal{M}(\Omega;\mathbb{M}_{\mathrm{sym}_0}^{n\times n})$ is the space of finite \textit{trace-free, symmetric} $\mathbb{M}^{n\times n}$-valued Radon measures.  \smallskip

\subsection{Energy minimization problems in Materials science}\label{sbsct:Energy}
        As far as variational problems are concerned, in elasticity or elasto-plasticity, but also soft matter and material engineering, it is classical to consider the stored elastic energy
\[
\int_\Om  W(\E u)dx
\]
and to consider the additive decomposition $W(\E u)= W_{\rm{shear}}(\rm{dev}\E u)+ W_{\rm{bulk}}(\dive u)$ with some growth properties on each kinematical variable. Simply putting $ W_{\rm{bulk}}(\dive u)=0$ is not a reasonable choice, at least for any natural material, since it implies the unphysical assumption of a material with vanishing bulk (compression) modulus. However, it might be of interest and studied as a limit case for the fabrication of some metamaterials, within the challenging field of material synthesis. These modern, artificial solids or liquids, like colloidal crystals, polymer foams, or unscreened metals, exhibit plenty of surprising properties and therefore open the way to many modern industrial applications \cite{VanishingBulk2,VanishingBulk4,VanishingBulk1}. Many of these properties are related to materials with negative Poisson ratios \cite{NegativePoisson2,NegativePoisson}, like auxetic materials, which, in contrast to classical materials, exhibit a reverse deformation mechanism \cite{NegativePoisson}. The huge diversity of mechanical properties of modern and natural materials can also be viewed in plots of the bulk modulus $B$ versus shear modulus $G$. For instance, materials with a small Poisson’s ratio are more easily compressed than sheared (small $B/G$), whereas those with a high Poisson’s ratio resist compression in favour of shear (large $B/G$). As a limit case, when $B/G\ll 1$ (hence with Poisson's ratio at its lower bound, $\nu\to -1$), materials are extremely compressible, examples being re-entrant foams and molecular structures \cite{NegativePoisson2}.  

As a matter of fact, one may be interested in studying homogenization problems like $W_\varepsilon(\E u)= g(\varepsilon)W_0(\E_d u) + {b(\varepsilon) W_{\rm{vol}}(\dive u)}$ with $b/g(\varepsilon)\to0$ as $\varepsilon\to0$. At the limit, no bound on $\dive u$ exists, and for linear type $W_0$ as $C^{-1}|\xi|\leq W_0(\xi)\leq C|\xi|$ the relevant space is thus $\BDD$. In mathematical physics, the problem of minimizing
\[
\int_\Om  \left(W_{\rm{shear}}(\E_d u)-f\cdot u\right)dx
\]
is relevant in general relativity (specifically for the study of black holes), for the so-called “momentum constraint” issue (see \cite{Dain} and the references therein). 


\smallskip

In terms of functional spaces, one has $\BDD\supset \mathrm{BD}\supset \mathrm{BV}$, and, as we said, the mechanics of the phenomena under analysis do not always allow one to consider the full deformation gradient $Du$ as a kinematical variable on its own in any mathematical model, i.e., one often has no control on all components of the full gradient. Now, suppose for instance we have a functional $\mathcal F:\mathrm{S}(\Omega)\times \mathrm{Bor}(\Omega)\to [0,+\infty]$  (here $ \mathrm{Bor}(\Omega)$ is the family of Borel sets of $\Omega$) defined on the functional space $\mathrm{S}(\Omega)=\mathrm{BV}(\Omega), \mathrm{BD}(\Omega)$, or $\BDD(\Omega)$, lower semicontinuous on $\mathrm{S}$, and satisfying the bound
\begin{align}\label{Fbound}
0\leq \mathcal F(u)\leq C |\mu|(\Omega),
\end{align}
for some $C>0$, and where $\mu=Du, \E u$, or $\E_d u$, respectively, with $|\mu|$ the total variation measure of $\mu$. Then, according to the physical problem under study, the analysis of $\mathcal F$ and in particular the question of the existence of an integral representation for $\mathcal F$ arises naturally in each $\mathrm S$. This issue is prominent if, for instance, $\F$ arises from a relaxation process, since then it will be lower semicontinuous, even if its integral expression is not known, or even if it would exist. Also, in homogenization processes, one might wish to pass to the limit with respect to a small scale parameter present in the deviatoric part of the energy $W_{\rm{bulk}}(\E_d u)$ and eventually determine the integral form of the limit. For $\BV$, the pioneering work can be found in \cite{bouchitte1998global}, while in $\BD$ it was recently achieved by the authors in \cite{caroccia2019integral}. It thus became natural to raise the question of an integral representation for $\mathcal F$ as defined in $\BDD$.

\smallskip

The strategy developed in \cite{caroccia2019integral} goes through an iterative blow-up procedure that is based on rigidity properties of $\BD$ maps with constant polar vectors and a specific projection operator onto $\mathrm{Ker}(\E)$ appearing in the Korn-Poincaré inequality. Actually, as shown in \cite{CVG2025}, where the authors refined a double blow-up technique for this purpose, these two ingredients are enough to tackle homogenization and integral representation problems.\\

\smallskip
We here provide a rigidity property for maps with constant polar in $\BDD$ and a specific projection operator $\mathcal{R}:L^1\rightarrow \mathrm{Ker}(\E_d)$. With these two results at hand, we will establish the two main ingredients needed to solve integral representation and homogenization problems in $\BDD$.

\subsection{Rigidity result}
Given a map $u\in \BDD(K)$, suppose that
\begin{equation}\label{eqn:introrigidity}
\E_d u = M \mu, \ \  \mu\in \mathcal{M}(K;\R^+)  
\end{equation}
for $\mu$ a real-valued non-negative Radon measure and $M\in \mathbb{M}_{\mathrm{sym}_0}^{n\times n}$ a constant matrix. Then what can be said about $u$? Actually, for the purposes of the present analysis it is not necessary for $M$ to be a generic matrix, but just a matrix in the \textit{wave cone} $\Lambda_{\Ao}$: a specific subspace of $\mathbb{M}_{\mathrm{sym}_0}^{n\times n}$ depending on the differential operator $\Ao$ annihilating $\E_d$ (such that $\Ao \E_d =0$). Indeed, with classical tools (as done in \cite{de2019fine}), it is easy to show that if $M\notin \Lambda_{\Ao}$, then $u$ must be actually $C^{\infty}$ and thus the relevant case from the technical point of view is the case $M\in \Lambda_{\Ao}$. Not only this: given $u\in \BDD(\Omega)$, we can consider the Radon-Nikodým decomposition
\[
\E_d u= \frac{\d \E_d u}{\d \L^n}\L^n + \E_d^s u.
\]
Blowing up $u$ at some absolutely continuous point $x$, we have convergence to an affine map given by $y\mapsto e_d (u) (x) y + u(x)$ (here $e_d (u):=e(u)-\frac{1}{n}\II$, $e(u)=\frac{\nabla u + \nabla u^t}{2}$, where $\nabla u(x)$ is the \textit{approximate differential of $u$ at $x$} that exists $\L^n$-a.e.). But for integral representation and homogenization, a characterization of the blow-ups on $ \spt(\E_d^s u)$ is also required. A celebrated result (valid for any elliptic differential operator) due to De Philippis and Rindler \cite{de2016structure} implies that
\[
 \frac{\d \E_d^s u}{\d |\E_d^s u|}\in \Lambda_{\Ao} \ \ \text{for $|\E_d^s u|$-a.e. $x\in \Omega$ }.
\]
If we then consider a blow-up at $x\in \spt(\E_d^s u)$ on a specific convex set $K$: 
\begin{equation}\label{eqn:blwupSequences}
u_{K,\e,x}(y)=\frac{u(x+\e y) - \mathcal R_K[u](y)}{\frac{|\E_d u|(K_{\e}(x) )}{|K|\e^{n-1}}},
\end{equation}
we see that $u_{K,\e,x}\rightarrow v$ where 
\[
\E_d v =  \frac{\d \E_d^s u}{\d |\E_d^s u|}(x) |\E_d v|.
\]
Thus, in order to characterize the blow-ups, we need to study the structure of solutions to \eqref{eqn:introrigidity} in the case when $M\in \Lambda_{\Ao}$. \\

Note that the annihilator of a differential operator is not unique, since if $\Ao$ annihilates $\E_d$ then $\partial_i \Ao$ still annihilates $\E_d$. But it is true that $\Lambda_{\Ao}\subset \Lambda_{\partial_i \Ao}$. In particular, to constrain as much as possible the polar of the singular part, one wants to find the annihilator $\Ao$ of $\E_d$ of lowest order. We are able to compute such an annihilator, and with this differential operator at hand we can focus our rigidity result only on $M\in \Lambda_{\Ao}$. 
\begin{remark}
A major challenge in the present setting, compared to the classical $\BD$ case, arises from the fact that the deviatoric operator $\E_d$ does not behave well under changes of variables that orthogonalize the vectors $a$ and $b$. Such changes are routinely used in the $\BD$ case to simplify the analysis, but are not available here. This fundamental obstacle is one of the main sources of technical complexity in our proof.
\end{remark}
We recall the notation:
\[
a\odot b:=\frac{a\otimes b+b\otimes a}{2}.
\]
\begin{theorem}\label{MainTheoremINTRO}
Let $n\geq 3$. There exists an annihilator $\Ao$ for $\E_d$ of order $4$ for which it holds 
\[
\Lambda_{\Ao}=\left\{\left. a\odot b - \frac{(a \cdot b)}{n}\II\  \right| \ a,b \in \R^n  \right\}.
\]
Moreover, for any $u\in \BDD(K)$ satisfying
\begin{equation}\label{Eqn:MainThmRigidityRelationINTRO}
\E_d u = \left(a \odot b - \frac{(a \cdot b)}{n}\II \right)\mu 
\end{equation}
for some $a,b\in \R^n$, $\mu\in \mathcal{M}(K;\R^+)$, one of the following two cases holds:
\begin{itemize}
    \item [1)] If $a$ and $b$ are not parallel, then there exist two functions $\psi_1,\psi_2\in C^{\infty}(\R)$ and $v\in \langle a, b\rangle^{\perp} $ such that
\begin{equation}\label{eqn:rigidityMainINTRO}
u(x)=\psi_1(x\cdot a)b+\psi_2(x\cdot b)a+Q(x)+L(x)
\end{equation}
for some $L\in \mathrm{Ker}(\E_d)$ and for some homogeneous third-order degree polynomial $Q$ solving 
\begin{equation}\label{eqn:PolynomialRemainderMainThmINTRO}
\E_d Q= \left(a \odot b - \frac{(a \cdot b)}{n}\II \right) \left((v\cdot x) +\eta(a\cdot x)(b\cdot x)-\vartheta \sum_{j=3}^n  (x\cdot w_j)^2 \right)
\end{equation}
where  $\eta,\vartheta\in \R$, $v\in \langle a,b\rangle^{\perp}$ and $\{w_3,\ldots,w_{n} \}$ is an orthonormal basis of $\langle a,b\rangle^{\perp}$;
\item[2)] If $b=\lambda a$ then there exist functions $F \in \BV_{loc}(\R)$, $\{P_j\}_{j=2}^n\subset W^{1,1}_{loc}(\R)$ with $P_j'\in \BV_{loc}$, $G\in W^{1,1}_{loc}(\R)$ with $G'\in \BV_{loc}(\R)$ such that
\begin{equation}\label{eqn:ShapeUForBetaEqualZeroINTRO}
\begin{split}    
u=& F(a \cdot x) a + \left(\sum_{j=2}^n P_j'(a\cdot x)(w_j\cdot x) + \frac{(w_j\cdot x)^2}{2}G'(a\cdot x)\right)a\\
&-\sum_{j=2}^n ((w_j\cdot x) G(a\cdot x)+P_j(a\cdot x ))w_j+\varrho Q(x)+L(x)
\end{split}
\end{equation}
for $\{w_2,\ldots,w_n\}$ an orthonormal basis of $a^{\perp}$, for some $L\in \mathrm{Ker}(\E_d)$, $\varrho \in \R$ and for some homogeneous third-order degree polynomial $Q$ solving 
\begin{equation}\label{eqn:PolynomialRemainderMainThmBETAZEROINTRO}
\E_d Q= \lambda \left(a \odot a - \frac{|a|^2}{n}\II \right) \left( (w_2\cdot x)^2-(w_3\cdot x)^2\right).
\end{equation}
Moreover, if $n\geq 4$ then $\varrho=0$.
\end{itemize}
\end{theorem}

The operator $\Ao$ is explicitly computed in Proposition \ref{prop:Annihilator}.  \\

In both cases, 1) and 2), there is a one-dimensional part with $\BV$ regularity, a part orthogonal to $a,b$ with $W^{1,1}$ regularity, and a polynomial part. The main difference between the $\BD$ and the $\BDD$ case is the polynomial part, which is non-trivial for $\BDD$ maps and arises from the fact that \eqref{eqn:PolynomialRemainderMainThmINTRO}, \eqref{eqn:PolynomialRemainderMainThmBETAZEROINTRO} have non-trivial solutions (explicitly computable).
\begin{remark}
    The constraint $n\geq 3$ is actually quite important, since for $n=2$ the operator $\E_d$ lacks a fundamental property called $\C$-ellipticity, required for several structural properties (as, for instance, the existence of traces, as shown in \cite{breit2017traces}). The difference between $n\geq 4$ and $n=3$, as made explicit in the condition on $\varrho$ in part 2) of Theorem \ref{MainTheoremINTRO}, is actually quite common throughout the proofs. Heuristically speaking, the number of differential equations satisfied by a $u$ solving \eqref{Eqn:MainThmRigidityRelationINTRO} depends on the dimension: for $n=3$ there are simply fewer conditions on $u$, making the proof of the rigidity structure more challenging. For $n=3$, when $a$ is parallel to $b$, we lose an additional equation in \eqref{Eqn:MainThmRigidityRelationINTRO}, resulting in the presence of the polynomial part. These considerations seem to strongly indicate that, for $n=2$, there might be too few equations to constrain the solution of \eqref{eqn:introrigidity} for $M\in \Lambda_{\Ao}$ to have such one-dimensional $\BV$ parts. 
\end{remark}

\subsection{Kernel projection}
While the rigidity result presented above constitutes the main analytical cornerstone of our approach, the development of a comprehensive blow-up strategy in the space $\BDD$ requires a further structural ingredient. In particular, although the following result has a somewhat less pronounced impact compared to the rigidity theorem, it remains an essential tool in iterative blow-up procedure as clarified in Subsection \ref{subsct:theRole}. 
\begin{theorem}\label{MainTheoremINTRO2}
Let $n\geq 3$ and $K$ be a center symmetric convex body. Upon defining the integral operators $\mathrm{s}_{K}, \mathrm{A}_{K}, \gamma_{K},   \mathrm{b}_{K}$ and $\tau_K$ (see Section \ref{sct:KernelProjection} for detail \eqref{eqn:s},\eqref{eqn:A},\eqref{eqn:gamma}, \eqref{eqn:d}), let us define the map
$ \mathcal{R}_{K}:\BDD(\Omega)\rightarrow \mathrm{Ker}(\E_d)$ as
     \begin{equation} 
    \begin{split}
    \mathcal{R}_{K}[u](y):=&\left(\mathrm{A}_{K}[u]+ \gamma_{K}[u] \Id\right) (y-\mathrm{bar}(K))\\ &+(\mathrm{s}_{K}[u] \cdot (y-\mathrm{bar}(K)))(y-\mathrm{bar}(K)) -\mathrm{s}_{K}[u]\frac{|y-\mathrm{bar}(K)|^2}{2}+ \mathrm{b}_{K}[u]
    \end{split}
   \end{equation}
is a linear, bounded operator satisfying $\mathcal{R}_K(L)=L$ for all $L\in \mathrm{Ker}(\E_d)$. As a consequence,
\begin{equation}\label{eqn:PoincaréInequality}
\|u-\mathcal{R}_K[u]\|_{L^1(\Omega)}\leq c |\E_d u|(\Omega) \ \ \ \text{for all $u\in \BDD(\Omega)$}.
\end{equation}
\end{theorem}
Note that the Poincaré inequality is actually well-known for a whole series of operators (see \cite[Theorem 3.7]{diening2024sharp}), holding for any linear, bounded kernel projection operator. Here we simply restrict ourselves to computing a specific kernel operator.  

\subsection{The role of Theorems \ref{MainTheoremINTRO} and \ref{MainTheoremINTRO2} in homogenization} \label{subsct:theRole}
As can be seen by looking at the iterative blow-up procedures \cite{caroccia2019integral}, \cite{CVG2025}, \cite{de2016structure}, the most important part relies on having a one-dimensional structure on the $\BV$ part and a specific structure for $\mathcal{R}_K$ (for computational reasons). The other terms in the rigidity theorem can be easily handled due to their better regularity (namely $W^{1,1}$ and polynomial). Let us sketch the $\BD$ case along the milestones in \cite{bouchitte1998global,caroccia2019integral}. \\

As seen in Subsection \ref{sbsct:Energy} Given a local, lower-semicontinuous energy $\F:\BD(\Omega)\times \mathrm{Bor}(\Omega)\rightarrow \R$ we seek an explicit characterization of
\begin{equation}\label{eqn:LIMITENERGY}
  \frac{\d \F(u;\cdot)}{|\E u|}(x_0)  =\lim_{\e\rightarrow 0}\frac{ \F(u;B_\e(x_0))}{ |\E u|(B_\e(x_0))} 
\end{equation}
at $|\E u|$-a.e. $x_0\in \Omega$, that would yield an integral representation of $\F$ (in terms of $|\E u|$, which has a well-known structure). This is done by understanding the blow-ups of $u$ at $x_0$.\\

We have that an annihilator of $\E$ is called the \textit{Saint-Venant operator} (or $\curl \curl^t$, \cite{MSVG2015}), and we denote it by $\mathcal{SV}$. First, by \cite{de2016structure}:
\[
\Lambda_{\mathcal{SV}}:=\left\{a \odot b \  | \ a, b \in \R^n \right\}, \ \ 
\]
Second, if
\[
\E v= (a\odot b) \mu \ \ \text{on $K$}
\]
then the shape of $v$ must be \cite{de2019fine}:
\begin{equation}\label{eqn:RigStrBD}
v(y)=b \psi_1(y\cdot a)+a \psi_2(y\cdot b)+L(y),
\end{equation}
for some $\psi_1,\psi_2\in \BV_{loc}(\R)$, $L\in \mathrm{Ker}(\E)$. Finally, a kernel projection operator fixing $\mathrm{Ker}(\E)$ is
\begin{equation}\label{kernel}
\mathcal{R}'_K[u](x)=\mathrm{A}_K[u] x + \frac{1}{P(K)}\int_{\partial K} u(y)\d \H^{n-1}(y)
\end{equation}
with $\mathrm{A}_K[u]$ defined in \eqref{eqn:A}.
The iterated blow-ups strategy for relaxation can now be described as a repeated application of a rigidity structure theorem as \ref{MainTheoremINTRO} and the knowledge of the operator $\mathcal{R}'_K$ as follows.  
\begin{itemize}
     \item[1)] If $x_0$ is a point of approximate differentiability or a jump point, then blow-ups are given by the approximate differential $e(u)(x)y$ or by a jump function;
     \item[2)] If $x_0$ is neither a jump point nor an approximate differentiability point then, by \cite{de2016structure}, we still have that $\frac{\d \E^s u}{\d |\E^s u|}$ must belong to $\Lambda_{\mathcal{SV}}$;
    \begin{itemize}
        \item[2.1)] We consider a blow-up $v$, $L^1$-limit of $u_{K,\e,x}=\frac{u(x+\e y) - \mathcal{R}'_K[u](y)}{\frac{|\E d|(K_\e (x))}{|K_\e (x)|}}$ (which exists by means of a Poincaré inequality similar to \eqref{eqn:PoincaréInequality} for $\BD$), given $\mathcal{R}'_K$ as in \eqref{kernel}. This blow-up will have the property of satisfying
     \[
     \E v=  (a\odot b)  |\E  v| \ \ \text{on $K$};
     \]
     \item[2.2)] The rigidity \eqref{eqn:RigStrBD} will now yield information on the shape of $v$. In particular (in the $\BD$ case), we have 
     \[
     v(x)=b \psi_1(x\cdot a)+a \psi_2(x\cdot b)+L(x);
     \]  
     \item[2.3)] By selecting a specific point $y$ on the domain, we perform a second blow-up on $v$, which linearizes one direction, resulting in one blow-up of the form (for some $\psi\in \BV_{loc}$)
     \[
     w(z)= \kappa  a  (b\cdot z) + \psi(z\cdot a)b;
     \]
     \item[2.4)] By employing that $\mathcal{R}'_K(w)=0$ and the knowledge of $\mathcal{R}'_K$, the constant $\kappa$ is computed, together with some useful properties of $\psi$;
     \item[2.5)] By a general principle (blow-ups of blow-ups are blow-ups, see also \cite[Proposition 3.6]{CVG2025}), we conclude that $w$ must be a blow-up of $u$ at $x$, namely that along a specific sequence $\{\e_i\}_{i\in \N}$, $u_{K,\e_i,x}\rightarrow w$.
    \end{itemize}
 \end{itemize}
At $|\E u|$-a.e. point we then have either convergence to an affine or a jump function, or convergence to a partially linear map. This information is now all that is required to identify the limit in \eqref{eqn:LIMITENERGY}.\\

Therefore, rigidity theorem \ref{MainTheoremINTRO} and kernel projection \ref{MainTheoremINTRO2} provide the main steps for running the same strategy also on local, lower-semicontinuous energies on $\BDD$. \\

The main advantage of this approach is that it allows the energy $\F$ to depend also on $u$ and not only on $\E_d u$.  

\subsection{Fine properties} 
Once the annihilator is computed, as a corollary we can derive a structure theorem for $\E_d u$ in the spirit of the one holding for $\BV$. In particular, we can prove that $|\E_d u|\ll\H^{n-1}$, and obtain the specific structure of the jump part (see also \cite{breit2017traces}): 
\[
\E_d u =  e_d (u) \L^n + \left([u] \odot \nu_u - \frac{([u] \cdot \nu_u)}{n}\II \right) \H^{n-1}\restr_{J_u}+ \left(a(x) \odot b(x) - \frac{(a(x) \cdot b(x))}{n}\II \right) |\E_d^c u|
\]
where $|\E_d^c u|$ is the Cantor part. Let us spend a few words about a major difference that we actually encounter when looking at the decomposition of $\E_d u$ compared to the decomposition of $\E u$ or $\mathcal{D} u$. Calling $S_u$ the set of points where $u$ is not \textit{approximately continuous} (i.e., $x$ is not a Lebesgue point for $u$), it is known that $\BD$ and $\BV$-maps charge this set almost all on $J_u$, namely $|\E u|(S_u \setminus J_u)=0$ for $\BD$, and $\H^{n-1}(S_u\setminus J_u)=0$ for $\BV$. It is actually an open problem whether the same property holds for $\BDD$ maps: $|\E_d u|(S_u\setminus J_u)=0$? The slicing technique developed in \cite{arroyo2020slicing}, as a generalization of \cite{ambrosio1997fine}, seems not to work for $\E_d u$ due to a specific missing property of the operator: it is not true in general that—not even in some special directions—the one-dimensional slices of a general $\BDD$ function are $\BV$ functions. \\

As a corollary of Theorem \ref{MainTheoremINTRO2}, by applying an argument similar to the one developed in \cite{ambrosio1997fine,kohn1980new}, we can rewrite the integral operators defining $\mathrm{s}_K, \mathrm{A}_K, \mathrm{\gamma}_K$ and $\mathrm{b}_K$ (in \eqref{eqn:s}-\eqref{eqn:d}) as nonlocal interaction integrals against $\E_d u$. This allows us to control the infinitesimal behavior of these  quantities, providing quasi-continuity $|\E_d u|$-a.e. for $\BDD$ functions, a weaker notion than approximate continuity. \\

We therefore report this very natural conjecture about the fine properties of $\BDD$ functions.
\begin{conjecture}
For all $u\in \BDD(\Omega)$ it holds $|\E_d u|(S_u\setminus J_u)=0$.
\end{conjecture}

The conjecture might be true only for $n\geq 3$, since in $n=2$ the operator fails to be $\mathbb{C}$-elliptic, although it is not clear how important this property is for the $|\E_d u|$-a.e. approximate continuity.

\subsection{Strategy of the proof for the rigidity}
The most important result, where the highest non-triviality lies, is the proof of the rigidity part in Theorem \ref{MainTheoremINTRO}. The proof is quite computational, so we spend a few words to explain the underlying strategy, which is fully developed in Section \ref{sct:ProofofRigidity}. \\

We first consider a general function $u\in C^{\infty}(K;\R^n)$ satisfying
\begin{equation}\label{eqn:rigidregintro}
\E_d u =\left(a  \odot b  - \frac{(a  \cdot b )}{n}\II \right)g
\end{equation}
for some $g\in C^{\infty}(\R^n;\R)$. We exploit this structure of $\E_d u$ and apply Schwarz's Theorem to the differential relation
\[
\nabla \left(\frac{\nabla u - \nabla u^t}{2}\right)_{ij}= \partial_i ((\E_d u)e_j) - \partial_j ((\E_d u)e_i) +\partial_i \left(\frac{\dive(u)}{n}\right) e_j - \partial_j \left(\frac{\dive(u)}{n}\right) e_i 
\]
(which is a variant of the differential relation exploited for the rigidity of $\BD$ maps with constant polar vector). This leads to some general considerations and to a set of PDEs (listed in Lemma \ref{lem:tecnico}) that $u$ must solve. In particular, we give these equations in terms of $u$, $g$, and $f:=\frac{\dive(u)-(a\cdot b)g}{n}$, which quantifies how much $u$ fails to satisfy $\BD$ rigidity, since 
\[
\E u= (a  \odot b) g + f\II
\]
($f=0$ implies that $u$ has the structure of \eqref{eqn:RigStrBD}). To perform this computation, we use specific coordinates. Having established this set of PDEs in Lemma \ref{lem:tecnico}, for $u\in C^{\infty}$ solving \eqref{eqn:rigidregintro}, we proceed to treat separately the cases $a\not \parallel b$ and $a\parallel b$. \\

The first case, treated in Subsection \ref{sbsct:RigidityNonParallel}, which is the most technical one, is handled by showing that the set of PDEs in Lemma \ref{lem:tecnico}—after suitable manipulation—actually leads to wave equations for $\partial_a g$ and $\partial_b g$. The D'Alembert formula then implies the one-dimensionality of $\partial_a g$ and $\partial_b g$ up to a polynomial remainder. This characterizes $g$ as being the sum of two waves plus a polynomial part, and this gives the precise structure \eqref{eqn:rigidityMainINTRO} for $u\in C^{\infty}$. This is done in Subsection \ref{sbsbsct:RigidityRegFunctionsNONPAR}. Now, for a generic $u\in \BDD(\Omega)$ solving \eqref{Eqn:MainThmRigidityRelationINTRO}, we consider a mollification $u_\e:=u \star \varrho_\e \in C^{\infty}$. Since mollification preserves the structure of \eqref{Eqn:MainThmRigidityRelationINTRO}, $u_\e$ will solve \eqref{eqn:rigidregintro} (with $g_\e=\mu\star \varrho_\e$). Thus the mollified $u_\e$ must have the claimed structure — being $C^{\infty}$ — and the main point now consists in showing that such a structure is stable when taking the limit as $\e\to 0$. The most difficult part here is handling the polynomial part $Q_{\e}$. Indeed, the coefficient of $Q_{\e}$ will involve the quantity $\tau_\e=\partial_{12}\left(\frac{\dive(u_\e)-g_\e}{n}\right)$, and such a quantity might in principle have no limit as $\e \to 0$, since $\dive(u_\e)$ might not converge for $\BDD$ maps ($\E_d u$ does not provide any control on $\dive(u)$). However, the constant polar structure of $\E_d u$ suggests that a $u$ solving \eqref{eqn:rigidityMainINTRO} is somewhat more than just $\BDD$. Thus, in Subsection \ref{sbsbsct:FeaturesPolSolNONPAR}, with an explicit computation of (part of) the polynomial remainder $Q$, we gain control on $\tau_\e$ and are able to pass to the limit in Subsection \ref{sbsbsct:ApproximationNONPAR}. \\

The second case, when $a\parallel b$, is treated in \ref{sbsct:RigidityParallel}, where we again argue first for $u\in C^{\infty}$, and we see that the equations given by Lemma \ref{lem:tecnico} are easily integrable, yielding a $g$ which has a one-dimensional part plus a purely polynomial part plus a mixed term, which is a polynomial in $x_j$ for $j\geq 2$ with coefficients depending on $x_1$. Again, once $g$ is identified, the shape of $u$ follows. For a generic function, we again need to pass to the limit, and we need to handle the floating constants that might diverge. We obtain the required control by testing \eqref{eqn:rigidityMainINTRO} against specific test functions. In this part, we find it convenient to treat separately the case $n=3$ and $n\geq 4$ due to a slight difference between the two cases.\\

In the end, being able to pass to the limit somehow amounts to showing that $u$ is actually more regular than $\BDD$  (at least $\BD$). This fact should not be a surprise, since $\E_d u$ has constant polar vector, and a posteriori, as Theorem \ref{MainTheoremINTRO} clarifies, such a $u$ is actually in $\BV$. 

\subsection{Organization of the paper}\label{sbsct:organization}
In Section \ref{sct:Prel}, we introduce the main ingredients required to fully treat the topics contained herein. Section \ref{sct:KernelAnnihilator} is devoted to computing $\mathrm{Ker}(\E_d)$ and the annihilator $\A$. Section \ref{sct:Fine} exploits the annihilator to derive some fine properties of $\E_d u$, in the spirit of \cite{de2019fine}, and lays the basis for the general analysis in Section \ref{sct:ProofofRigidity}, where the proof of the rigidity Theorem \ref{MainTheoremINTRO} is completed. Section \ref{sct:KernelProjection} instead provides the explicit computations leading to the identification of the kernel operator $\mathcal{R}_K$ in Theorem \ref{MainTheoremINTRO2}. Finally, in Section \ref{sct:Vanishing} and the appendix \ref{sct:appendix}, we provide some well-known computations based on the non-local approach from Kohn \cite{kohn1980new}, giving infinitesimal information on the quantities defining $\mathcal{R}_K$ and also implying \textit{quasi-continuity} $|\E u|$-a.e. on $\Omega$.

\subsection{Acknowledgments}MC acknowledges the financial support of PRIN 2022R537CS "Nodal optimization, nonlinear elliptic equations, nonlocal geometric problems, with a focus on regularity," funded by the European Union under Next Generation EU. NVG was supported by the FCT project UID/04561/2025. The authors thank Adolfo Arroyo Rabasa and Franz Gmeindeder for numerous stimulating discussions and valuable feedback on the topic over the past years.

\section{Preliminaries and main result}\label{sct:Prel} 

\subsection{General notations}\label{sbsct:General}
The letter $n$ will always denote the ambient Euclidean space dimension. We will denote by $B_r(x)$ the ball of radius $r$ and centered at $x$. Whenever $x=0$ we just write $B_r$, as well as in the case $r=1$ when we simply write $B(x)$. More in general, given a convex body $K$ we denote by $K_{r}(x):=x+rK$. We denote by $\mathbb{M}^{m\times n}$ the set of $n\times n$ matrices. The notation $e_i$ stands for the $i$-th vector of the canonical basis of $\R^n$, $\Id$ denotes the $n\times n$ identity matrix. With $\mathbb{M}_{sym}^{n\times n}$,  $\mathbb{M}_{sym_0}^{n\times n}$$\mathbb{M}_{skew}^{n\times n}$ we denote the subsets of $\mathbb{M}^{n\times n}$ made respectively by all symmetric matrices, all trace free symmetric matrices and skew-symmetric matrices. The space $\text{Lin}(X;Y)$ denotes the family of all linear maps between the two  vector spaces $X$ and $Y$. Given $v,w\in \R^n$ we will often consider the rank-one matrix $v\otimes w$ acting as $(v\otimes w)z=v (w\cdot z)$ for all $z\in \R^n$ and the matrix 
	\[
	v\odot w:=\frac{v\otimes w+w\otimes v}{2}.
	\]
The notation $\L^n$, $\H^{n-1}$ stand for the $n$-dimensional Lebesgue measure and the $(n-1)$-dimensional Hausdorff measure on $\R^n$ while $\mathcal{M}(\Omega;V)$ is the space of all finite $V$-valued Radon measures on $\Omega$ and all $V$. \\

For $u:\R^n\rightarrow \R^m$ we specify that 
    \begin{equation}
    \partial_j u:=    \left(\begin{array}{c}
    \partial_j u_1\\
    \vdots\\
    \partial_j u_m
    \end{array}\right) \in \R^m,
 \ \ \ \nabla u:=(\partial_1 u,\ldots,
    \partial_n u)\in \mathbb{M}^{m\times n}.
    \end{equation}
For the matrix-valued function $F:\R^n\rightarrow \mathbb{M}^{n\times n}$ we define its gradient as
    \[
    (\nabla F)_{lmj}=(\partial_j F)_{lm}=\partial _j F_{lm},
    \]
while its divergence is define as the row-divergence operator, namely the vector field
\[    
\div F\cdot e_i=\sum_{j=1}^n \partial_j F_{ij}.
\]

\subsection{Maps of bounded deformation}\label{sbsct:MapsBD}
For $u\in C^{\infty}(\R^n;\R^n)$ the \textit{symmetric gradient} is defined as
    \[
   \E u= \frac{\nabla u+\nabla^t u}{2}
    \]
The adjoint operator on $F\in  C^{\infty}(\R^n;\mathbb{M}_{sym}^{n\times n})$, reads
  \[
  \E^*F=\div F.
  \]
We also define, for $\xi\in \R^n$, the symbols $\mathbb{E}[\xi]: \R^n \rightarrow \mathbb{M}_{sym}^{n\times n}$ as 
\begin{align}\label{eqn:SymbolE}
    	\mathbb{E} [\xi] w= \xi\odot w   \ \ \ \text{for all $\xi,w\in \R^n$}.
\end{align}
The symmetric gradient measure $\E u \in \mathcal{M}(\Omega;\mathbb{M}_{sym}^{n\times n})$ is defined as
    \[
    \int_\Om\varphi(x) \cdot \d\E u(x) :=\int_\Om \dive\varphi(x)\cdot u(x) \d x \  \ \forall\  \varphi\in C^{\infty}_c(\Om;\mathbb{M}_{sym}^{n\times n}).
    \]
and the space of function with bounded deformation is
    \[
  \BD(\Omega):=\left\{\left. u\in L^1(\Omega;\R^n) \ \right| \ \E u\in \mathcal{M}(\Omega;\mathbb{M}_{sym}^{n\times n})\right\}\subset \BV(\Omega).
    \]
Several properties are already well-studied for this operator. In particular see \cite{ambrosio1997fine} or \cite{de2019fine} for a more recent approach, closer to the one in this paper. It is well-known that the Kernel of $\E$ is made by anti-symmetric affine transformation:
\begin{align}\label{eqn:KerE}
    \mathrm{Ker}(\E):=\{A x+b \ | \ A\in \mathbb{M}_{skew}^{n\times n}, b\in \R^n\}.
\end{align}
Moreover, for $u\in \BD(\Omega)$ we have that $u$ is $\L^n$-a.e. \textit{approximately differentible} and the gradient decomposition is in force
\[
\E u=e(u)\L^n + ([u]\odot \nu) \H^{n-1}\llcorner_{J_u} + \E^c u
\]
where $e(u)=\frac{\nabla u+\nabla u^t}{2}$, $J_u$ is the $\textit{jump set}$, $[u]$ is the \textit{jump of $u$} and $\E^c u$ is the \textit{Cantor part}. Recent development on this topic \cite{de2016structure} allows also to say that the Cantor part has a very rigid polar vector field
\[
\frac{\d \E^c u}{\d |\E^c u|}(x)=a(x)\odot b(x)
\]
for some $a,b:\Omega\rightarrow \R^n$ $|\E^c u|$ measurable Borel vector-fields. The \textit{Saint-Venant} compatibility condition are also a well-established fact: setting
\begin{equation}\label{eqn:SVcondition}
( \mathcal{SV}(M) )_{jk} := \sum_{i=1}^d \partial_{i k}(M)_{i j} +\partial_{i j}(M)_{i k} - \partial_{j k}(M)_{i i}-\partial_{i i}(M)_{j k}  \ \ \ \text{for all $j,k=1,\ldots,d$}
\end{equation}
for $M\in C^{\infty}(\Omega;\mathbb{M}_{sym}^{n\times n})$ then it holds
\[
\mathcal{SV} (\E u)=0 \ \ \ \text{for all $u\in C^{\infty}(\Omega;\R^n)$}.
\]
Note that its symbol is given by
\[
\mathbb{SV}[\xi]M:=    [( M\xi)\otimes \xi +\xi \otimes  (M\xi)] -  \trace(M)(\xi \otimes \xi) -|\xi|^2  M
\]




\subsection{Maps of bounded deviatoric deformation}\label{sbsct:Maps}
Now we define the set of functions with bounded $\E_d$-variation. For $u\in C^{\infty}(\R^n;\R^n)$ we consider the differential operator
    \[
   \E_d u=\E u -\frac{ \tr e(u)}{n}\II= \E u-\frac{\dive(u)}{n}\II
    \]
The adjoint operator on $F\in  C^{\infty}(\R^n;\mathbb{M}_{sym_0}^{n\times n})$, reads
  \[
  \E^*_dF= \dive (F)
  \]
We also define, for $\xi\in \R^n$, the symbols $\mathbb{E}_{d}[\xi]: \R^n \rightarrow \mathbb{M}_{sym_0}^{n\times n}$ as 
\begin{align}\label{eqn:SymbolEd}
    	\mathbb{E}_d [\xi] w= \xi\odot w - \frac{(\xi\cdot w)}{n}\II \ \ \ \text{for all $\xi,w\in \R^n$}.
\end{align}
We will often use the notation introduced in \cite{breit2017traces}
\[
v\otimes_{\E_d} w:= \mathbb{E}_d [v] w=\mathbb{E}_d [w] v.
\]
We can define then the measure $\E_d \in \mathcal{M}(\Omega;\mathbb{M}_{sym_0}^{n\times n})$ as
    \[
    \int_\Om\varphi(x) \cdot \d\E_d u(x) :=\int_\Om \dive\varphi(x)\cdot u(x) \d x \  \ \forall\  \varphi\in C^{\infty}_c(\Om;\mathbb{M}_{sym_0}^{n\times n}).
    \]
Given this notation we consider the space 
    \[
  \BDD(\Omega):=\left\{\left. u\in L^1(\Omega;\R^n) \ \right| \ \E_d u\in \mathcal{M}(\Omega;\mathbb{M}_{sym_0}^{n\times n})\right\}\subset \BD(\Omega).
    \]
Let us report the following results, proved in \cite[Theorem, 4.20]{breit2017traces} for general differential operators, and clarifying the relation between $\E_d$ and $\mathbb{E}_d$. 
\begin{proposition}
Let $n\geq 3$. Let $u\in \BDD(\Omega)$. Then for any $\H^{n-1}$-rectifiable set $R\subset \Omega$ there exists the trace $u\Big{|}_{R}$. Moreover, for any $u\in \BDD(\Omega), F\in C^{\infty}(\R^n;\mathbb{M}_{sym_0}^{n\times n}) $ it holds
\begin{align} 
 \int_{\Omega}  F\cdot \d \E_{d} u (x)= - \int_{\Omega} u\cdot \E^*_{d} F(x)\d x +\int_{\partial \Omega} (\mathbb{E}_{d}[\nu_{\Omega}](u) \cdot F \d \H^{n-1}(x)  \nonumber\\
 = - \int_{\Omega} u\cdot \E^*_{d} F(x)\d x +\int_{\partial \Omega}  (\nu_{\Omega}\otimes_{\E_d}u ) \cdot F \d \H^{n-1}(x) \label{eqn:GGformula}
\end{align}
\end{proposition}
\begin{remark}
In \cite{breit2017traces} the result is proven for any $\C$-elliptic operator $\A : C^{\infty}(\Omega;\R^n)\rightarrow C^{\infty}(\Omega:V)$ for some vector space $V$. The notion of $\mathbb{C}$-ellipticity can be stated as the \textit{injectivity of the symbol $\AA[\xi]:\mathbb{C}^n\rightarrow V+iV$ as a linear map from $\C^n$ into $V+iV$, for all $\xi \neq 0$}. Actually, $\mathbb{C}$-ellipticity is a very important property in order to ensure structural properties to the operator and its functional spaces. In particular for instance, the existence of the trace cannot be guaranteed for non $\mathbb{C}$-elliptic operator. We refer the interested reader to  \cite{breit2017traces} and the literature therein for more details on $\mathbb{C}$-ellipticity and functional spaces. 
\end{remark}
\begin{remark}\label{rmk:EdNotCell}
 It is a simple computation to show that $\mathbb{E}_d[\xi]:\mathbb{C}^n\rightarrow \mathbb{M}_{sym_0}^{n\times n}+i\mathbb{M}_{sym_0}^{n\times n}$ is $\mathbb{C}$-elliptic for all $n\geq 3$ while it is not $\mathbb{C}$-elliptic for $n=2$. This is probably one of the main reason for the failing of a lot of properties in $n=2$,  as for instance the existence of trace (see \cite{breit2017traces}). Also our rigidity Theorems \ref{thm:MainRigidityNonParallel} and \ref{thm:MainRigidityParallel} are proven for $n=2$. It is not clear though wether $n=2$ still allows for a rigidity structure. Our computation - and the full proofs - seem to strongly suggest that probably there are not enough equation in dimension $2$, to derive a strong rigidity structure on a $u$ with constant polar vector in the wave cone. 
\end{remark} 
 
A simple algebraic computations yields the following Leibniz rule, together with a useful property of the adjoint operator:
    \begin{align}
        \E_{d} (\varphi  u)&=  \varphi \E_{d} u + \mathbb{E}_{d}[\nabla \varphi] u & \varphi\in C^{\infty}(\R^n), u\in C^{\infty}(\R^n;\R^n) \label{eqn:LeibnizRule}\\
        (\mathbb{E}_{d}[\xi] z)\cdot M&= z\cdot\mathbb{E}^*_{d}[\xi] M  &  z\in  \R^n, M\in \mathbb{M}_{sym_0}^{n\times n}, \xi\in \R^n.\label{eqn:AdjointProp}
    \end{align}
 
\subsection{Poincaré-Sobolev and compactness}\label{sbsct:Poincare}
Finally we underline that as a consequence of several general results in literature on $\mathbb{C}$-elliptic operator we have also the following Poincaré-Sobolev inequality. In the following, $\Pi^{U}:L^1(U;\R^m)\rightarrow \mathrm{Ker}(\E_d)$ stands for a bounded linear projection operator onto the kernel of $\E_d$. We refer to \cite[Proposition 2.5]{gmeineder2017critical}, \cite[Proposition 2.5, Remark 2.6]{CVG2025}

\begin{proposition}[Poincaré-Sobolev inequality]\label{prop:sobBalls}
Let $n\geq 3$ and $K$ be a center-symmetric convex set. Then there exists a constant $c$ depending on $n$ and $K$ only such that
\begin{align}\label{SGNforA}
      \|u-\Pi^{K_r(x)}u\|_{L^\frac{n}{n-1}(K_r(x);\R^m)}\leq c |\E_d u|(\overline{K_r(x)}) \ \ \ 
\end{align}
for all $x\in \R^d$, $r>0$ and $u\in \BDD(\R^n)$. 
\end{proposition}
The space $\BDD(\Omega)$, endowed with the norm $\|u\|_{\BDD}:=|\E_d u|(\Omega)+\|u\|_{L^1}$, is a Banach space. The Poincaré-Sobolev inequality in \ref{prop:sobBalls} provide a standard argument, by following for instance the ideas in \cite{kohn1980new} (combined with the extension argument in \cite{gmeineder2019} to prove the following compactness Theorem.
\begin{theorem}[Compactness Theorem]\label{thm:cmp}
Let $\Omega\subset \R^n$ be an open bounded set with Lipschitz boundary and $n\geq 3$ be a first order linear operator. Let $\{u_k\}_{k\in \N}\subset \BDD(\Omega)$. Suppose that
\[
\sup_{k\in \N}\{\|u_k\|_{\BDD}\}<+\infty.
\]
Then there exists $u\in \BDD(\Omega)$ and a subsequence $h(k)$ such that $u_{h(k)}\rightarrow u$ in $L^1$ and $\E_d u_{h(k)} \wt \E_d u$.
\end{theorem}

The notation $\E_d u_{h(k)} \wt \E_d u$ stands for the standard \textit{weak$^*$ convergence} of Radon measures (see \cite{evans2018measure} or \cite{Maggi}).\\

As a consequence of \cite[Theorem 3.7]{diening2024sharp} or \cite[Proposition 2.8]{CVG2025} we have the following Poincaré inequality.
\begin{proposition}[Poincaré inquality]\label{prop:poincare}
Let $n\geq 3$, $K\subset \R^n$ be a fixed convex set. Let $\mathcal{R}: L^{1}(K;\R^n)\rightarrow \mathrm{Ker}(\E_d)$ be a linear, bounded operator such that $\mathcal{R}(L)=L$ for all $L\in  \mathrm{Ker}(\E_d)$. Then there exists a uniform constant $c=c(\mathcal{R},K,n)$ depending on $n,\mathcal{R}$ and $K$ such that
    \begin{align}\label{Poincare}
    \|u-\mathcal{R}[u]\|_{L^{1}(K;\R^n)}\leq  c\diam(K)|\E_d u|(K).
    \end{align}
\end{proposition}
\section{Kernel and annihilator}\label{sct:KernelAnnihilator}
In  this section we present the Kernel and the Annihilator of $\E_d$.

\subsection{Kernel}\label{sbsct:Kernel}
It is a well-known fact that the Kernel of $\E_d$ is made by Killing vector fields
\[
\mathrm{Ker}(\E_d):= \left\{L(y)=(A+ \gamma \mathrm{Id})y + s\frac{|y|^2}{2}-(s \cdot y) y +b\ \left| A\in \mathbb{M}^{d\times d}_{skew}, \ s,b\in \R^n, \ \gamma \in \R \right. \right\}.
\]
Since the proof of the Kernel structure and the ingredients required, are quite enlightening in order to deepen the approach used to prove rigidity we here present the proof of this result.\\

Set, for $u\in C^{\infty}(A;\R^n)$ the quantity $Wu:=\frac{\nabla u - \nabla u^t}{2}$ and observe that
\begin{equation}\label{eqn:Fond}
\nabla (Wu)_{ij}= \partial_i ((\E_d u)e_j) - \partial_j ((\E_d u)e_i) +\partial_i \left(\frac{\dive(u)}{n}\right) e_j - \partial_j \left(\frac{\dive(u)}{n}\right) e_i.
\end{equation}

\begin{proposition}[Kernel structure]\label{prop:ker}
Let $n\geq 3$, $u\in C^{\infty}(\Omega;\R^n)$ with $\Omega\subseteq \R^n$ a connected set. Suppose that $\E_d u =0$. Then there exists $s, b\in \R^n$, $A\in \mathbb{M}^{n\times n}_{skew}$, $\gamma\in \R$ such that
    \[
    u(y)= (A+\gamma \II)y + s\frac{|y|^2}{2}-(s\cdot y) y +b.
    \]
\end{proposition} 
\begin{proof}
Because of \eqref{eqn:Fond} and $\E_d u =0$ we have for all $i\neq j$
\[
\nabla (Wu)_{ij}= \partial_i \left(\frac{\dive(u)}{n}\right) e_j - \partial_j \left(\frac{\dive(u)}{n}\right) e_i.
\]
In particular, by taking the curl we get
\begin{equation}\label{one}
\partial_{i i} \dive(u)+ \partial_{jj} \dive(u)=0 \ \ \ \text{for all $j\neq i$}
\end{equation}
and
\begin{equation}\label{two}
\partial_{k i} \dive(u)=0 \ \ \ \text{for all $k\neq i$}.
\end{equation}
Since we have $n\geq 3$ we have at least another index $m\neq i\neq j$ for which 
\[
\partial_{mm} \dive(u) + \partial_{ii}\dive(u)=0, \ \ \ \partial_{mm} \dive(u) + \partial_{jj}\dive(u)=0,\ \ \ \partial_{ii} \dive(u) + \partial_{jj}\dive(u)=0.
\]
But then 
\[
\partial_{ii}\dive(u)=-\partial_{mm} \dive(u)=\partial_{jj}\dive(u)
\]
from which it follows $\partial_{ii} \dive(u)=\partial_{jj}\dive(u)=\partial_{mm}\dive(u)=0$. By combining this with \eqref{two} we have
\[
\nabla (\partial_i \dive(u))=0
\]
this being true for every fixed $i=1,\ldots,n$. Then, since $\Omega$ is connected, for some $s \in \R^n$
\[
\partial_i \frac{\dive(u)}{n}= s_i, \ \ \ \nabla (\dive(u))=n s 
\]
which again implies that $\dive(u)= n(s\cdot y + \gamma)$ for $s \in \R^n$, $\gamma\in \R$. This implies
\[
0=\E u - \frac{\dive(u)}{n}\II \ \Rightarrow \ \E u =  (s\cdot y + \gamma)\II.
\]
Observe that, setting $p(y):= (s\cdot y) y - s\frac{|y|^2}{2}  + \gamma y $ we have 
\[
 \E p (y)= (s\cdot y + \gamma) \II 
\]
and thus $\E (u-p)=0$. By now invoking the characterization of the kernel of $\E$ we conclude that, for some $A\in \mathbb{M}^{n\times n}_{skew}$, 
\[
u(y)-p(y)= Ay+b \ \Rightarrow \ u(y)=Ay+p(y)+b
\]
as claimed.  
\end{proof}

\subsection{Annihilator}\label{sbsct:Annihilator}
Given a function $F\in C^{\infty}(\Omega;\mathbb{M}_{sym_0}^{n\times n})$ we seek for an operator $\mathcal{A}$ such that $\mathcal{A} F=0$ whenever $F=\E_d u$ for some potential $u$. In particular the existence of such object for $\mathbb{C}$-elliptic operators is always guaranteed by a result of Van Shaftingen \cite[Remark  4.1, Lemma 4.4]{van2013limiting} (see also \cite[Proposition 17]{arroyo2023elementary} for an extension of the Van Shaftingen's construction). The result in the papers are abstract and not constructive. \\

The annihilator, together with the powerful result in \cite{de2016structure}, will allows us to determine the structure of the singular measure $\frac{\d \E^s_d u}{\d |\E^s_d u|}$. Indeed, setting
\begin{equation}
\Lambda_{\mathcal{A}}:=\bigcup_{|\xi|=1 } \mathrm{Ker}(\A[\xi])
\end{equation}
then (cf. with  \cite[Theorem 1.1]{de2016structure})  
\begin{equation}\label{eqn:waveconeStr}
\frac{\d \E_d u}{\d |\E_d u|}(x)\in \Lambda_{\mathcal{A}} \ \ \ \text{for $|\E_d^s u|$-a.e. $x\in \Omega$}
\end{equation}

Since the annihilitor is not unique (think about $\mathrm{curl}$ and $\nabla (\mathrm{curl})$ both annihiling $\nabla u$) we need to seek for the operator with the lowest possible order so to have the Kernel of its symbol (and thus its wave cone) the smallest possible in order to find the sharpest constraint on the polar vector of the singular part. 

\begin{proposition}[Annihilator]\label{prop:Annihilator}
Define for $F\in C^{\infty}(\Omega;\mathbb{M}_{sym_0}^{n\times n})$ the fourth-order operator $\mathcal{A}:  C^{\infty}(\Omega;\mathbb{M}_{sym_0}^{n\times n})\rightarrow  C^{\infty}(\Omega;\mathbb{M}_{sym_0}^{n\times n})$
\begin{align}
        (\mathcal{A} (F))_{j k}:=&\sum_{i,\ell=1}^n\partial_{ii j\ell} F_{k\ell}+\partial_{ii k\ell} F_{\ell j}-\sum_{i,\ell=1}^n\partial_{ii \ell\ell} F_{jk}\nonumber\\
        &-\frac{n-2}{n-1}\sum_{i,\ell=1}^n\partial_{i\ell j k} F_{i\ell}-\frac{\delta_{jk}}{n-1}\sum_{i,\ell,m =1}^n\partial_{ii\ell m} F_{\ell m}.\label{eqn:OperatorA}
\end{align}
Then, if $u\in C^{\infty}(\R^n;\R^n)$ it holds
\[
\mathcal{A}(\E_d u)=0.
\]
\end{proposition}
\begin{proof}
    It is verifiable via a direct computation or by arguing on the symbol as explained in Remark \ref{rmk:OnTheAnnihilitor}. We report, for the sake of completeness, the main step of the computations where we make repeatedly use of Schwarz's Theorem.\\
    \begin{align*}
        \sum_{i,\ell=1}^n \partial_{ii j\ell}  (\E_d u)_{k\ell}+ \partial_{ii k\ell}(\E_d u)_{\ell j}=&\frac{1}{2}\sum_{\substack{i,\ell=1  }}^n \partial_{ii j\ell} \left(\partial_{\ell }u_k +\partial_k u_{\ell}\right)   - \frac{1}{n}\sum_{\substack{i=1}}^n \partial_{ii jk}  \div(u) \\ 
        &+ \frac{1}{2}  \sum_{\substack{i,\ell=1    }}^n \partial_{ii k\ell} \left(\partial_{\ell }u_j +\partial_j u_{\ell}\right) - \frac{1}{n}\sum_{\substack{i=1 }}^n \partial_{ii k j} \div(u)  \\
        =&\frac{1}{2}\sum_{\substack{i,\ell=1  }}^n  \partial_{ii \ell \ell }(\partial_j u_k+\partial_k u_j) +\frac{n-2}{n}\sum_{\substack{i=1  }}^n \partial_{ii j k} \dive(u)\, ,  \\
        -\sum_{i,\ell=1}^n\partial_{ii \ell\ell}  (\E_d u)_{jk}=&-\frac{1}{2}\sum_{i,\ell=1}^n\partial_{ii \ell\ell}  \left( \partial_{j}u_k +\partial_k u_{j}\right)  +\frac{\delta_{jk}}{n} \sum_{i,\ell=1}^n\partial_{ii}\partial_{\ell\ell}  \dive(u)\,  ,\\
        -\frac{n-2}{n-1}\sum_{i,\ell=1}^n\partial_{i\ell j k} (\E_d u)_{i\ell}=& -\frac{n-2}{2(n-1)}\sum_{\substack{i,\ell=1 } }^n\partial_{i\ell j k} (\partial_i u_\ell+\partial_\ell u_i) + \frac{n-2}{n(n-1)}\sum_{\substack{i=1 } }^n\partial_{ii j k} \dive(u)\\
        =& -\frac{n-2 }{ n }\sum_{\substack{i = 1 } }^n  \partial_{i i j k } \dive(u) \, ,\\
        -\frac{\delta_{jk}}{n-1}\sum_{i,\ell,m =1}^n\partial_{ii\ell m} (\E_d u)_{\ell m}=& -\frac{\delta_{jk}}{2(n-1)}\sum_{\substack{i,\ell,m =1}}^n\partial_{ii\ell m} (\partial_\ell u_m +\partial_m u_\ell ) +\frac{\delta_{jk}}{n(n-1)}\sum_{\substack{i,\ell=1}}^n\partial_{ii\ell \ell} \dive(u) \\
        =&-\frac{\delta_{jk}}{n} \sum_{\substack{i,\ell  =1}}^n  \partial_{ii\ell\ell  }   \dive(u)\, .
    \end{align*}
Now by simply adding up the above relations we obtain $\mathcal{A}(\E_d u)=0$.
\end{proof}
\begin{remark}\label{rmk:OnTheAnnihilitor}
    Notice that the symbol of $\mathcal{A}$ is given by $\mathbb{A}[\xi]:\mathbb{M}^{n\times n}_{sym_0}\rightarrow \mathbb{M}^{n\times n}_{sym_0}$
    \begin{equation}\label{eqn:SymbolA}
    \mathbb{A}[\xi]M:=|\xi|^2(M\xi \otimes \xi + \xi \otimes M\xi) -|\xi|^4M -\frac{(\xi^t M \xi)}{n-1}\left[(n-2) \xi \otimes \xi +|\xi|^2\II \right].
    \end{equation}
Notice also that \eqref{eqn:OperatorA} is a very natural choice since, starting from 
\[
\mathbb{E}_d[\xi] u = u\odot \xi - \frac{(u\cdot \xi)}{n}\II
\]
we have
\[
\xi^t (\mathbb{E}_d[\xi] u)\xi= |\xi|^2(u\cdot \xi) \frac{(n-1)}{n}\ \ \ \Rightarrow \ \ \frac{(u\cdot \xi)}{n}=\frac{\xi^t (\mathbb{E}_d[\xi] u)\xi}{|\xi|^2 (n-1)}.
\]
Thus
\begin{align*}
(\mathbb{E}_d[\xi] u)\xi &= \frac{1}{2}\left( (u\cdot \xi)\xi + u |\xi|^2\right) - \frac{\xi^t (\mathbb{E}_d[\xi] u)\xi}{|\xi|^2 (n-1)}\xi \\
&=\frac{(n-2)}{2(n-1)} \frac{  \xi^t (\mathbb{E}_d[\xi] u)\xi}{ |\xi|^2 } \xi + \frac{u |\xi|^2}{2}  
\end{align*}
yielding also 
\[
u=\frac{2}{|\xi|^2}\left[(\mathbb{E}_d[\xi] u)\xi-\frac{(n-2)}{2(n-1)} \frac{  \xi^t (\mathbb{E}_d[\xi] u)\xi}{ |\xi|^2 } \xi\right]
\]
and
\begin{align*}
\mathbb{E}_d[\xi] u = \frac{2}{|\xi|^2}\left[((\mathbb{E}_d[\xi] u)\xi)\odot \xi -\frac{(n-2)}{2(n-1)} \frac{  \xi^t (\mathbb{E}_d[\xi] u)\xi}{ |\xi|^2 } \xi\odot \xi\right]  -  \frac{\xi^t (\mathbb{E}_d[\xi] u)\xi}{|\xi|^2 (n-1)}\II.
\end{align*}
Multiplying by $|\xi|^4$ we get
\begin{align*}
0&= 2|\xi|^2 ((\mathbb{E}_d[\xi] u)\xi)\odot \xi -\frac{(n-2)}{(n-1)} ( \xi^t (\mathbb{E}_d[\xi] u)\xi) \xi\odot \xi   -  |\xi|^2 \frac{\xi^t (\mathbb{E}_d[\xi] u)\xi}{  (n-1)}\II - |\xi|^4 (\mathbb{E}_d[\xi] u) \\
&=\mathbb{A}[\xi](\mathbb{E}_d[\xi]u).
\end{align*}
So somehow the fourth order is the minimum required in order to find a linear function $\mathbb{A}[\xi]$ for which $\mathbb{E}_d[\xi] u\in \mathrm{Ker}(\mathbb{A}[\xi])$.

\end{remark}
\begin{remark}
    Note that $\mathcal{A}(F)$ in \eqref{eqn:OperatorA}, for $F\in C^{\infty}(\Omega;\mathbb{M}_{\mathrm{sym}_0}^{n\times n})$ can be expressed also as
    \[
    \mathcal{A}(F)_{jk} =\Delta\left(\partial_{  j}\dive(F)_k +\partial_{  k } \dive(F)_j\right)-\Delta^2 F_{jk}  -\frac{n-2}{n-1}\partial_{j k}(\dive(\dive(F)) -\frac{\delta_{jk}}{n-1}\Delta (\dive(\dive(F)).
    \]
\end{remark}
\begin{remark}
The computation in Remark \ref{rmk:OnTheAnnihilitor} are consistent with the Saint-Venant condition \eqref{eqn:SVcondition} annhilating $\E u$. Indeed
\[
\mathbb{E}[\xi] u =u\odot \xi, \ \ \ \trace(\mathbb{E}[\xi] u)=(u\cdot \xi)
\]
and
\begin{align*}
   2 (\mathbb{E}[\xi] u) \xi= (u\cdot \xi)\xi +u|\xi|^2, \ \ \ \Rightarrow \ \ \     u= \frac{2}{|\xi|^2} (\mathbb{E}[\xi] u) \xi-\frac{(u\cdot \xi)\xi}{|\xi|^2} 
\end{align*}
yielding
\[
(\mathbb{E}[\xi] u) = \frac{2}{|\xi|^2} ((\mathbb{E}[\xi] u) \xi)\odot \xi-\frac{(u\cdot \xi)}{|\xi|^2} (\xi \odot \xi)
\]
and
\[
0 =  2 ((\mathbb{E}[\xi] u) \xi)\odot \xi- \trace(\mathbb{E}[\xi]u)  (\xi \odot \xi)-\mathbb{E}[\xi] u|\xi|^2=\mathbb{S V}[\xi](\mathbb{E}[\xi] u).
\]
So the second order is the minimum required to find a linear function $\mathbb{SV}[\xi]$ for which $\mathbb{E} [\xi] u\in \mathrm{Ker}(\mathbb{SV}[\xi])$ In the deviatoric operator the control on $(u\cdot \xi)$ requires an additional $|\xi|^2$, differently from the symmetric case.
\end{remark}
We can now compute the the wave cone of $\mathcal{A}$.
\begin{proposition}[Wave cone of $\mathcal{A}$]\label{prop:WaveCone}
    If $M \in \mathbb{M}_{sym_0}^{n\times n}$ and $|\xi|=1$ then
    \[
    \mathbb{A}[\xi]M=0 \ \ \Leftrightarrow \ \ M=v\odot \xi -\frac{(v\cdot\xi)}{n}\II \ \text{for some $v\in \R^n$}.
    \]
In particular
    \[
    \Lambda_{\mathcal{A}}=\left\{\left. v\odot \xi -\frac{(v\cdot\xi )}{n} \II \ \right| \ v\in \R^n, \ \xi\neq 0 \right\}.
    \]
\end{proposition}
\begin{proof}
If $M=v\odot \xi - \frac{(v\cdot \xi)}{n}\II$ for some $v$ then it is immediate that $\mathbb{A}[\xi]M=0$. So we prove the other implication. Up to a rotation we can assume without loss of generality that $\xi=e_1$. Then $\mathbb{A}[e_1]M=0$ implies, from \eqref{eqn:SymbolA}
\[
M= (Me_1 \otimes e_1 + e_1 \otimes Me_1) -\frac{(e_1^t M e_1 )}{n-1}\left[(n-2) e_1 \otimes e_1 +\II \right].
\]
This gives us 
\[
M_{11}=(e_1^t M e_1 )=:\varrho, \ \ M_{ij}=0 \ \ \text{for $i\neq j$, $i,j>1$}, \ \ M_{ii}=-\frac{\varrho}{(n-1)}, \ i> 1.
\]
Thus, setting $w_j:=M_{1j}=M_{j1}$ for $j>1$ we have that
\begin{equation*}
M = \left(   \begin{array}{ccccc}
         \varrho & w_2 & \ldots &   \ldots &w_n   \\
          w_2 & -\frac{\varrho}{n-1}& \ldots  & \ldots & 0\\
          \vdots & \vdots &-\frac{\varrho}{n-1} &\ldots &0\\
            \vdots  &  \vdots  & \vdots & \ddots & 0\\
          w_n & 0 &\ldots& 0 & -\frac{\varrho}{n-1}
    \end{array}
    \right)
\end{equation*}
By now choosing $v_j:=2w_j$ for $j>1$ and $v_1= \frac{n \varrho}{n-1}$ we get also $\frac{\varrho}{n-1}=\frac{v_1}{n}$ and thus
\begin{equation*}
M = \left(   \begin{array}{ccccc}
         v_1\left(1-\frac{1}{n}\right) & \frac{v_2}{2} & \ldots &   \ldots &\frac{v_n}{2}   \\
          \frac{v_2}{2} & -\frac{v_1}{n}& \ldots  & \ldots & 0\\
          \vdots & \vdots &-\frac{v_1}{n} &\ldots &0\\
            \vdots  &  \vdots  & \vdots & \ddots & 0\\
         \frac{v_n}{2}  & 0 &\ldots& 0 & -\frac{v_1}{n}
    \end{array}
    \right)= v\odot e_1 - \frac{(v\cdot e_1)}{n}\II
\end{equation*}
and the claim follows.
\end{proof}
\begin{corollary}[Polar vector of $\E_d u$]\label{cor:PolarSingualrPart}
Let $u\in \BDD(\Omega)$. Then there exists two Borel vector fields $a,b:\Omega\rightarrow \R^n$ such that, for $|\E_d^s u|$-a.e. $x\in \Omega$
\[
\frac{\d \E_d^s u}{\d |\E_d^s u|}(x)=a(x)\odot b(x) - \frac{(a(x)\cdot b(x))}{n}\II=a(x)\otimes_{\E_d} b(x)
\]
where $\E_d^s u$ is the singular part in the Radon-Nikodým derivative of $\E_du$, with respect to $\L^n$.
\end{corollary}
\begin{proof}
 Due to \eqref{eqn:waveconeStr} we must have
 \[
 \frac{\d \E_d u }{\d |\E_d u|}(x)\in \Lambda_{\mathcal A} \ \ \ \text{for $|\E_d^s u|$-a.e. $x\in \Omega$}
 \]
 Thanks to Proposition \ref{prop:WaveCone} we conclude.
\end{proof}
The above Corollary gives a precise structure to blow-ups around singular points and motivates Theorem \ref{MainTheoremINTRO}, Section \ref{sct:ProofofRigidity}.

\section{Fine properties from the annihilator}\label{sct:Fine}
\subsection{Structure of the gradient}\label{sbsct:StructureGrad}
For any $u\in L^1(\Omega;\R^n)$ the Lebesgue point Theorem ensures that for $\L^n$-a.e. $x\in$ there exists a precise representative $u(x)$ such that
\[
\lim_{r\rightarrow 0}\fint_{B_{r}(x)}|u(y)-u(x)|\d x=0.
\]
The set of points where this property fails is denoted as $S_u$ and is the \textit{discontinuity set of $u$}. For $x\in  \Omega\setminus S_u$ the value $u(x)$ is also called the \textit{approximate limit} of $u$. \\

We recall that $u\in L^1(\Omega;\R^n)$ is said to be \textit{approximately differentiable at $x\in \Omega\setminus S_u $} if there exists a matrix $L\in \mathbb{M}^{n\times n}$ such that
\[
\lim_{r\rightarrow 0} \fint_{B_r(x)} \frac{|u(x)-u(y) - L (x-y)|}{r}\d y =0.
\]
In this case $M$ is also called the approximate gradient and the notation $\nabla u(x)=M$ is adopted.
Thanks to \cite[Theorem 1.1]{raita2019critical} (see also \cite[Theorem 3.4]{alberti2014p}, that requires a non-local representation of the involved quantity in the spirit of Proposition \ref{prop:NonLocalRepr} in the Appendix) we can deduce that $u\in \BDD(\Omega)$ is approximately differentiable at $\L^n$-a.e. $x\in \Omega$. Moreover the same results ensures that
\[
\frac{\d \E_d u}{\d \L^n}(x)=  e_d (u)(x), \ \ \ e_d (u)(x):=  e (u)(x) - \frac{\trace(e(u)(x))}{n}\Id
\]
where we recall that $e(u)(x)=\frac{\nabla u(x)+\nabla u(x)^t}{2}$.\\

Thus, by Corollary \ref{cor:PolarSingualrPart} we have
\[
\E_d u= e_d (u)(x)\L^n + a(x)\otimes_{\E_d} b(x)|\E_d^s u|
\]
for two measurable vector fields $a,b:\Omega \rightarrow \R^n$.\\

We recall that for $u\in L^1_{loc}$ the set $J_u$ is defined as the set of points $x$ for which there exists a triplets $(u^+(x),u^-(x),\nu_u(x))$ such that $u^+\neq u^-$, $\nu_u \in \mathbb{S}^{n-1}$ and
    \[
  0 =\lim_{r\rightarrow 0+} \fint_{B^\pm_r(x)}  |u(y)-u^{\pm}(x)|\d y
    \]
where    
    \begin{align*}
 B_r^{-}(x):= \{y\in B_r(x) \ | \ y\cdot \nu_u(x)\leq 0\}, \ \ \ \ \ \   B_r^{+}(x):= \{y\in B_r(x) \ | \ y\cdot \nu_u(x)\geq 0\}. 
    \end{align*}
Clearly $J_u\subseteq S_u$. A recent result \cite{del2021rectifiability} shows that $J_u$ is always $n-1$ rectifiable and that $\nu_u$ is the unitary vector field orienting $J_u$ (namley $\nu_u(x)^{\perp}=\mathrm{Tan}(J_u,x)$ for $\H^{n-1}$-a.e. $x\in J_u$, cf. \cite{ambrosio2000functions}).
\subsection{Structure of the jump part}\label{sbsct:StructureJump}
Define
\begin{equation}\label{eqn:ThetaU}
\Theta_u:=\left\{x\in \Omega \  \left| \ \limsup_{r\rightarrow 0} \frac{|\E_d u| (B_r(x))}{r^{n-1}}>0\right.\right\}.
\end{equation}
We make use of the results in \cite{arroyo2019dimensional} and \cite{breit2017traces} to prove that $|\E_d u|\ll\H^{n-1}$ and a reasonable structure result for the gradient on the Jump set jump set. Set
\[
\Lambda_{\mathcal{A}}^{n-1}:=\bigcap_{v\in \R^n}\bigcup_{\xi\in v^{\perp}\setminus\{0\}} \ker(\mathbb{A}(\xi)).
\]
\begin{proposition}\label{prop:finePropGradient}
Let $n\geq 3$, $u\in \BDD(\Omega)$ and $\mathcal{A}$ be the annihilator given by Proposition \ref{prop:Annihilator}. Then
\begin{equation}\label{eqn:RectifiabilityWaveCone}
\Lambda_{\mathcal{A}}^{n-1}=\{0\}.
\end{equation}
As a consequence we have that:
\begin{itemize}
    \item[1)] $|\E_d u|\ll\H^{n-1}$ ;
    \item[2)] $|\E_d u|\left(\left\{x\in \Omega \ \left| \ \limsup_{r\rightarrow 0} \frac{|\E_d u| (B_r(x))}{r^{n-1}}=+\infty \right.\right\}\right)=0$;
    \item[3)] $\H^{n-1}(\Theta_u\Delta J_u)=0$ and
    \[
     \E_d u\restr_{\Theta_u}= \E_d u\restr_{J_u}= [u ]\otimes_{\E_d} \nu_u  \H^{n-1}\restr_{J_u}.
    \]
\end{itemize}
\end{proposition}
Before proceeding to the proof we first provide a simple Lemma from linear algebra that will simplify our argument in computing $ \Lambda_{\mathcal{A}}^{n-1}$.
\begin{lemma}\label{lem:LinAlgLem}
Let $a,\xi \in \R^n\setminus\{0\}$ with $|\xi|=1$.  Then
\begin{itemize}
\item[1)] If $a$ and $\xi$ are not parallel then $a\odot \xi$ has two distinct eigenvalues
\[
\mu_1=\frac{(a\cdot \xi)+|a|}{2}, \ \ \mu_2=\frac{(a\cdot \xi)-|a|}{2}
\]
\item[2)] If $a$ and $\xi$ are  parallel then $a\odot \xi$ has one eigenvalue 
\[
\mu_1=(a\cdot \xi)=|a|
\]
\end{itemize}
\end{lemma}
\begin{proof}
We treat the two cases separately.\\
\smallskip

\textbf{Case 1).} Without loss of generality we can suppose that $\xi= e_1$ and that $ \langle e_1 ,a \rangle=  \langle e_1,e_2 \rangle $.  Note that
\begin{equation}
a\odot e_1=(a\cdot e_1) e_1\odot e_1 + (a\cdot e_2) e_1\odot e_2 = \left(\begin{array}{cc}
(a\cdot e_1) & \frac{(a\cdot e_2)}{2}\\
\frac{(a\cdot e_2)}{2} & 0
\end{array}\right).
\end{equation}
To find the eigenvalues we need to solve
\[
0=\det( a\odot e_1 - \mu \Id)= -((a\cdot e_1) - \mu)\mu - \frac{(a\cdot e_2)^2}{4}= \mu^2 - \mu (a\cdot e_1)- \frac{(a\cdot e_2)^2}{4},
\]
whose solutions are precisely
\[
\mu_{1,2}= \frac{(a\cdot e_1)\pm \sqrt{(a\cdot e_1)^2+(a\cdot e_2)^2}}{2}=\frac{(a\cdot e_1)\pm |a|}{2},
\]
which are the claimed values.\\
\smallskip

\textbf{Case 2).} Without loss of generality we can suppose that $\xi= |\xi| e_1$,  $a=|a|e_1$.  Then,  the only eigenvector is $e_1$ itself with eigenvalue
\[
\mu_1=|a||\xi|=(a\cdot \xi).
\]
\end{proof}
\begin{proof}[Proof of Proposition \ref{prop:finePropGradient}]
It is enough to prove \eqref{eqn:RectifiabilityWaveCone}. From this relation indeed, by \cite[Corollary 1.4, Theorem 1.5]{arroyo2019dimensional} and a simple application of the theory in \cite{breit2017traces}, we will conclude properties 1)-3).\\
\smallskip

\textbf{Proof of \eqref{eqn:RectifiabilityWaveCone}}. Let $A\in \Lambda_{\mathcal{A}}^{n-1}$. Then, by Definition of $\Lambda_{\mathcal{A}}^{n-1}$ and by Proposition \ref{prop:WaveCone} in particular
\[
A\in \bigcap_{v\in \R^n } \bigcup_{\xi\in v^{\perp}\setminus \{0\}} \left\{\left. a\odot \xi - \frac{(a\cdot \xi)}{n}\Id \  \right| \ a\in \R^n\right\}.
\]
Fix any $v\in \R^n$ and let $\xi\in v^{\perp}\setminus \{0\}$, $a\in \R^n$ be such that
\[
A=a\odot \xi - \frac{(a\cdot \xi)}{n}\Id.
\] 
Without loss of generality, up to redefine $a$, we can assume that $|\xi|=1$. Denote by 
\[
\mathrm{eig}(A)=\left\{\lambda_1 ,\lambda_2 ,\lambda_3 ,\ldots, \lambda_n \right\}
\]
the family of eigenvalues of $A$. Note that the eigenvectors of $A$ are given by the eigenvectors of $M:=a\odot \xi$ and by a base of $\mathrm{Ker}(M)$. All the eigenvectors $v\in \mathrm{Ker}(M) $ has the same eigenvalue $-\frac{(a\cdot \xi)}{n}$. Note that $\mathrm{Ker}(M)$ has dimension either $n-2$ (for $a\not \parallel \xi$) or $n-1$ (for $a\parallel \xi$), (as shown in \cite{de2016structure}). In any case $A$ has at least $n-2$ coincident eigenvalues $\lambda_3=\ldots=\lambda_n=-\frac{(a\cdot \xi)}{n}$. Let $v_k$ be an eigenvectors relative to $\lambda_k$. For $i=1,2,3$ let now $\xi_i\in v_i^{\perp}\setminus \{0\}$ and $a_i\in \R^n$ be such that
\[
A=a_i \odot \xi_i - \frac{(a_i\cdot \xi_i)}{n}\Id.
\]
By fixing $z\in a^{\perp}\cap a_i^{\perp} $ we get
\[
- \frac{(a\cdot \xi)}{n}|z|^2=z^tAz= - \frac{(a_i\cdot \xi_i)}{n}|z|^2 \ \ \Rightarrow \ \  - \frac{(a_i\cdot \xi_i)}{n}=\lambda_3 \ \ \text{for all $i=1,2,3$}.
\]
Now note that (since $\xi_i\cdot v_i=0$)
\[
\lambda_2 v_2= A v_2 = \xi_2 (a_2 \cdot v_2) + \lambda_3 v_2 \ \ \Rightarrow \ \ (a_2 \cdot v_2)=0, \  \lambda_2 = \lambda_3,
\]
and
\[
\lambda_1 v_1= A v_1 = \xi_1 (a_1 \cdot v_1) + \lambda_3 v_1 \ \ \Rightarrow \ \ (a_1 \cdot v_1)=0, \  \lambda_1 = \lambda_3.
\]
In particular $\lambda_1=\lambda_2=\ldots=\lambda_n$. Now, by Lemma \ref{lem:LinAlgLem}, if $a\not \parallel \xi$ we just observe that
\[
\lambda_1=\frac{n-2}{2n}(a\cdot \xi) -\frac{|a|}{2},\  \ \lambda_2=\frac{n-2}{2n}(a\cdot \xi)+\frac{|a|}{2} 
\]
while for $ a\parallel \xi$
\[
\lambda_1=\frac{(n-1)}{n}|a|.
\]
 In the first case, from $\lambda_1=\lambda_2$, we immediately have $|a|=0$. In the second case, from $\lambda_1=\lambda_3$ we have $\frac{(n-1)}{n}|a|=-\frac{|a|}{n}$ which again implies $a=0$. In particular $A=0$ is the only possibility and \eqref{eqn:RectifiabilityWaveCone} holds true.
\smallskip
\end{proof}
In view of Proposition \ref{prop:finePropGradient} we conclude that, for $u\in \BDD(\Omega)$ we have the spltting in three mutually singular measures
\[
\E_d u =  e_d (u) \L^n + [u] \otimes_{\E_d} \nu_u  \H^{n-1}\restr_{J_u}+  a(x)\otimes_{\E_d} b(x)|\E_d^c u|
\]
where $|\E_d^c u|$ is the Cantor part. While for $\BD$ it is known the further important property $|\E u|(S_u\setminus J_u)=0$ (for $\BV$ it holds in the stronger form $\H^{n-1}(S_u\setminus J_u)=0$) this is actually not known in the $\BDD$ context. At the current state it seems technically difficult to be established and the available technology, such as \cite{ambrosio1997fine}, \cite{arroyo2020slicing} does not seem to apply to the deviatoric operator since it does not satisfy a one-dimensional slicing property. We refer the reader to \cite{arroyo2019fine} for a partial result in this sense. 

\section{Proof of Rigidity Theorem \ref{MainTheoremINTRO}}\label{sct:ProofofRigidity}

In this Section we prove the rigidity structure for maps with constant polar vector field, i.e., that satisfies
\[
\E_d u=(a\otimes_{\E_d} b)  \mu, \ \ \mu\in \mathcal{M}(\Omega;\R^+), \ a,b\in \R^n.
\]
In homogenization problem this scenario is the only one that occurs when dealing with Cantor points where the characterization of the blow-up is required. It is also the most challenging from the technical point of view.\\

A very important difference between rigidity in $\BD$ (cf. \cite{de2019fine}) and rigidity in $\BDD(\Omega)$ is that we cannot, in the proof, perform a change of variable that will make $a\perp b$. Indeed while for $\tilde{u}(x)=A u (A^t x)$ it holds
    \[
    e(\tilde{u})(x)=A e (u) (A^tx) A^t
    \]
we cannot express $\E_d\tilde{u}$ as a linear transformation of $\E_d(u)$.  So somehow the operator $\E_d$ does not behave well under change of variables. However, if the matrix $A$ is a rotation then we can infer
    \[
    \E_d (\tilde{u})(x)=A \E_d (u) (A^tx) A^t.
    \]
This property allows us, without loss of generality, to rotate the coordinates in order to have a more explicit relation between $a$ and $b$. In particular without loss of generality we can assume that $a=e_1$, $b=\alpha e_1+\beta e_2$. In this way, by selectively chosing $\a=0$, or $\beta=0$, we can deal with the case of perpendicular vectors, parallel vectors or general position vectors, respectively. \\

We find convenient to introduce the function $f=\frac{\dive(u)-\alpha g}{n}$, quantifying how much $u$ is far from satisfying a $\BD$ rigidity.  

\begin{lemma}\label{lem:tecnico}
Let  $n \geq 3$ and $u\in C^{\infty}(\R^n;\R^n)$ be such that
\[
\E_d u = \left[e_1 \odot(\alpha e_1+\beta e_2)-\frac{\a}{n}\Id\right]g
\]
for some $g\in C^{\infty}(\R^n)$.  Setting $f=\frac{\dive(u)-\alpha g}{n}$,  then the following set of equation hold 
\begin{align}
\beta \partial_{21}g - \a \partial_{22}g-\partial_{22}f-\partial_{11}f=&0 \label{ZEROone}\\
\partial_{jj} g \frac{\beta }{2}+\partial_{21} f=&0\ \ \ &\text{for all $j\geq 3$},\label{ONEone}\\
\partial_{11} f+\partial_{jj}f+\alpha \partial_{jj} g=&0 \ \ \ &\text{for all $j\geq 3$},\label{ONEtwo} \\
\frac{\beta}{2}\partial_{1j}g-\a\partial_{2j}g - \partial_{2j}f=&0\ \ \ &\text{for all $j\geq 3$}, \label{ONEthree}\\
\partial_{22}f +\partial_{jj}f=&0 \ \ \ &\text{for all $j\geq 3$}, \label{TWOone}\\
\partial_{2j} g\frac{\beta}{2}-\partial_{1j}f=&0 \ \ \ &\text{for all $j\geq 3$}, \label{TWOtwo}\\
\partial_{kj} g=&0 \ \ \ &\text{for all $k,j\geq 3$, $k\neq j$},\label{TWOthree}\\
\partial_{kj} f  =&0\ \ \ &\text{for all $k\neq j$ and $j\geq 3$}, \label{THREEone}\\
\partial_{ii}f +\partial_{jj}f=&0 \ \ \ &\text{for all $i,j\geq 3$, $i\neq j$}  \label{THREEtwo}.
\end{align} If $\beta\neq 0$  we further get  
\begin{align}
     \partial_{2j}g=&\partial_{1j}g=0 \ \ \ \ &\text{for all $j\geq 3$},\label{1jg}\\
\partial_{j}g =& -\frac{2 \partial_{12}f }{\beta} x_j +  v_j  \ \ \ \ &\text{for all $j\geq 3$}\label{constancy}
\end{align}
for some $v_j\in \R$ and $\mathrm{Hess}(f)$ is constant. 
\end{lemma}
\begin{proof}
Observe that
\begin{align*}
    (\E_d u)e_k= & e_1 \alpha\delta_{1,k} + \frac{\beta }{2}[e_1 \delta_{2,k}+e_2\delta_{1,k}] - \frac{\alpha}{n}e_k=e_1 \left[  \alpha\delta_{1,k} + \frac{\beta }{2}\delta_{2,k}\right]+e_2 \frac{\beta }{2}\delta_{1,k}- \frac{\alpha}{n}e_k.
\end{align*}
Then, by recalling that $Wu:=\frac{\nabla u - \nabla u^t}{2}$, combining the above with \eqref{eqn:Fond} we have
\begin{align*}
    \nabla (Wu)_{i,j}=&\left( e_1 \left[ \alpha  \delta_{1,j} + \frac{\beta }{2}\delta_{2,j}\right]+e_2 \frac{\beta }{2}\delta_{1,j}- \frac{\alpha }{n}e_j\right) \partial_{i}g \\
    &- \left( e_1 \left[\alpha \delta_{1,i} + \frac{\beta }{2}\delta_{2,i}\right]+e_2 \frac{\beta }{2}\delta_{1,i}- \frac{\alpha }{n}e_i\right) \partial_{j}g+\partial_i \left(\frac{\dive(u)}{n}\right) e_j - \partial_j \left(\frac{\dive(u)}{n}\right) e_i.
\end{align*}
We can rewrite the above as
\begin{align}
 \nabla (Wu)_{i,j}=&e_1 \left[ \left( \alpha \delta_{1,j} + \frac{\beta }{2}\delta_{2,j}\right)\partial_i g - \left(\alpha  \delta_{1,i} + \frac{\beta }{2}\delta_{2,i}\right)\partial_j g  \right] +e_2 \frac{\beta }{2}\left[\delta_{1,j}\partial_i g- \delta_{1,i}\partial_j g\right]\nonumber \\
 &+e_j\partial_i\left(\frac{\dive(u)- \alpha g}{n}\right) - e_i \partial_j\left(\frac{\dive(u)- \alpha g}{n}\right) \label{final}
\end{align}
this being valid for all $i,j=1,\ldots,n$.  Thus
\begin{align}
\nabla (Wu)_{12}=&e_1 \left( \frac{\beta }{2}  \partial_1 g -  \alpha   \partial_2 g  - \partial_2 f  \right) - e_2 \left(\frac{\beta }{2}  \partial_2 g -\partial_1 f\right),\label{ZERO} \\
 \nabla (Wu)_{1,j}=&  -  e_1\left( \alpha \partial_j g+\partial_j f\right) -e_2 \frac{\beta}{2} \partial_j g+e_j\partial_1f,\ \ &\text{for all $j\geq 3$}, \label{ONE} \\
  \nabla (Wu)_{2,j}=&  -  e_1\frac{\beta}{2}  \partial_j g+e_j\partial_2 f -e_2\partial_jf,\ \ &\text{for all $j\geq 3$},\label{TWO} \\
 \nabla (Wu)_{i,j}=&e_j\partial_if -e_i\partial_jf \ \  &\text{for all $i,j\geq 3$}.\label{THREE} 
\end{align}
By considering the curl of \eqref{ZERO} we get \eqref{ZEROone}. By taking the curl of \eqref{ONE} we get \eqref{ONEone},\eqref{ONEtwo} and \eqref{ONEthree}.  By considering the curl of \eqref{TWO} we get \eqref{TWOone},  \eqref{TWOtwo} and \eqref{TWOthree}.  By considering the curl of \eqref{THREE} we get \eqref{THREEone} and \eqref{THREEtwo}. \\

If thence $\beta\neq 0$,  by \eqref{THREEone} and \eqref{TWOtwo} we obtain
\[
\partial_{2j} g=0 \ \ \text{for all $j\geq 3$}
\]
and this combined with \eqref{ONEthree} (still for $\beta\neq 0$) yields \eqref{1jg}.  Finally \eqref{1jg},  \eqref{TWOthree} and \eqref{ONEone} implies
\[
\nabla(\partial_j g)=-\frac{2}{\beta}(\partial_{12}f) e_j \ \ \ \text{for all $j\geq 3$}.
\]
From this we immediately get $\partial_{k12}f=0$  for all $k$ (it is immediate if $n\geq 4$ while in dimension $n=3$ we obtain it by deriving in $\partial_2$ the relation $\partial_{13}f=0$ given by \eqref{THREEtwo}).  This yields \eqref{constancy}. By deriving in $\partial_1$ (or $\partial_{2}$) \eqref{ONEtwo} for $j=3$ we get $\partial_{111}f=0$ (or $\partial_{211}f=0)$ that yields $\nabla (f_{11})=0$. By doing the same on \eqref{TWOone} we get also $\nabla (f_{ii})=0$ for all $i\geq 2$. As a consequence we also have $\nabla (f_{ij})=0$ for all $i,j$ and thus $\mathrm{Hess}(f)$ must finally be a constant matrix.  
\end{proof}
We now treat two spearate cases, depending on $\beta$ being zero or different from zero. Before proceeding let us recall some well established facts in the next Remark. 

\begin{remark}\label{rmk:CrucialKNOWN}
    
Observe that it is immediate to verify that if $w$ solves
\[
\E w =(a\odot b) g
\]
then it will solves also $\dive(w)=\trace(\E w)=(a\cdot b)g$ and thence
\[
\E_d w =(a\otimes_{\E_d}b)g.
\]
Since (see for instance \cite[Theorem 2.10, Assertion (i)]{de2019fine}) the function 
\begin{equation}\label{eqn:SolToKnownPart}
w=a\psi_1(b\cdot x)+b\psi_2(a\cdot x) +(v\cdot x)[a(b\cdot x)+b(a\cdot x)] - v(a\cdot x)(b\cdot x) 
\end{equation}
- for $a,b$ not parallel - solves
\[
\E w= (a\odot b) (\psi_1'(b\cdot x) + \psi_2'(a\cdot x)+2(v\cdot x))
\]
and thus
\begin{equation}\label{eqn:KnownPart}
\E_d w =(a\otimes_{\E_d}b) (\psi_1'(b\cdot x) + \psi_2'(a\cdot x)+2(v\cdot x)).
\end{equation}
\end{remark}

With this established we can now state and proceed.

\subsection[Rigidity for non-parallel vectors]{Rigidity for non parallel vectors ($\beta \neq 0$)}\label{sbsct:RigidityNonParallel}
In this Section we provide the proof to the following
\begin{theorem}\label{thm:MainRigidityNonParallel}
Let $u\in \BDD(A)$ for a connected open set $A\subset \R^n$. Suppose that
\begin{equation}\label{eqn:RigidPolarMeasure}
\E_d u=(a\otimes_{\E_d}b)\nu
\end{equation}
for some $a, b\in \R^n$, $a\neq \lambda b$ and some positive Radon measure $\nu \in \mathcal{M}(A;\R^+)$. Then there exists two functions $\psi_1,\psi_2\in \BV_{loc}(\R)$ and $v\in \langle a, b\rangle^{\perp} $ such that
\begin{equation}\label{eqn:rigidityMain}
u(x)=\psi_1(x\cdot a)b+\psi_2(x\cdot b)a+Q(x)+L(x)
\end{equation}
for some $L\in \mathrm{Ker}(\E_d)$ and for some third order degree polynom $Q$ solving 
\begin{equation}\label{eqn:PolynomialRemainderMainThm}
\E_d Q=(a\otimes_{\E_d} b) \left((v\cdot x) +\eta(a\cdot x)(b\cdot x)-\vartheta \sum_{j=3}^n  (x\cdot w_j)^2 \right)
\end{equation}
where  $\eta,\vartheta\in \R$, $v\in \langle a,b\rangle^{\perp}$ and $\{w_3,\ldots,w_{n} \}$ is an orthonormal basis of $\langle a,b\rangle^{\perp}$. 
\end{theorem}
The proof of \ref{thm:MainRigidityNonParallel} is achieved by arguing first on regular functions and then by a density argument. In order to correctly pass to the limit we need to gather control on $\eta,\theta$. To do this some features on the general integral of the Polynomial equation \eqref{eqn:PolynomialRemainderMainThm} must be found.
\subsubsection{Rigidity for regular functions}\label{sbsbsct:RigidityRegFunctionsNONPAR}

\begin{proposition}\label{propo:RigForRegFunctions}
Let  $A\subset \R^n$ a connected set and $u\in C^{\infty}(A;\R^n)$ be such that
\begin{equation}\label{eqn:RigidPolar}
\E_d u= (a\otimes_{\E_d} b) g
\end{equation}
for some $a, b\in \R^n$, $a\neq \lambda b$, $g\in C^{\infty}(A;\R^+)$. Then there exists two functions $\psi_1,\psi_2\in C^{\infty}(\R)$ and $v\in \langle a, b\rangle^{\perp} $ such that
\begin{equation}\label{eqn:rigidityReg}
u(x)=\psi_1(x\cdot a)b+\psi_2(x\cdot b)a+Q(x)+L(x)
\end{equation}
for some $L\in \mathrm{Ker}(\E_d)$ and for some $Q$ solving 
\begin{equation}\label{eqn:PolynomialRemainder}
\E_d Q=(a\otimes_{\E_d} b) \left((v\cdot x) +\eta(a\cdot x)(b\cdot x)-\vartheta \sum_{j=3}^n  (x\cdot w_j)^2 \right)
\end{equation}
where  $\eta,\vartheta\in \R$, $v\in \langle a,b\rangle^{\perp}$ and $\{w_3,\ldots,w_{n} \}$ is an orthonormal basis of $\langle a,b\rangle^{\perp}$.
\end{proposition}

\begin{proof}[Proof of Proposition \ref{propo:RigForRegFunctions}]
 We place ourselves in the coordinate $a=e_1$, $b=\alpha e_1+\beta e_2$ for $\beta \neq 0$. Recall that, by Lemma \ref{lem:tecnico} we have $\mathrm{Hess}(f)$ is a constant. By \eqref{ZEROone}:
\begin{equation}\label{key}
\zeta = \beta\partial_{1 2 }g  - \alpha \partial_{22}g.
\end{equation}
Also \eqref{1jg}, \eqref{TWOthree} and \eqref{constancy} implies that
 \[
 \partial_{ki}g=0 \ \ \text{for all $k\neq i$, $i\geq 3$},  \ \ \partial_{j}g = -2\tau x_j+v_j \ \ j\geq 3
 \]
for some $\tau= \partial_{12}f\in \R$. We need just to identify $\partial_1g$ and $\partial_2 g$. Let us now consider separately the case $\alpha=0$ and $\alpha \neq 0$.\\

\smallskip
\textbf{The case $\alpha=0$ ($a\perp b$).} From \eqref{key} we get $\partial_{12}g=\eta$.  From this and \eqref{1jg} we derive that 
\[
\partial_{1}g=h_1(x_1)+\eta x_2 , \ \ \ \partial_{2}g=h_2(x_2)+\eta x_1
\]
for some function $h_1,h_2\in C^{\infty}(\R)$. Thus
\begin{equation}
    \nabla g=e_1 (h(x_1)+\eta x_2) + e_2 (h(x_2)+\eta x_1)  + \sum_{j=3}^d (v_j- 2\tau x_j)e_j
\end{equation}
and thus
\[
g= H_1(x_1)+H_2(x_2) + \eta x_1x_2 + \sum_{j=3}^d v_j x_j - \tau x_j^2  
\]
with $H_1,H_2$ such that $H_1'=h_1$, $H_2'=h_2$. In particular  according to Remark \ref{rmk:CrucialKNOWN} we have that, setting $\psi_1'=H_1$, $\psi_2'=H_2$, $u$ as in \eqref{eqn:rigidityReg}, with $Q$ solving 
\[
\E_d Q = (e_1\otimes_{\E_d}e_2)\left( (v\cdot x) +\eta x_1x_2- \sum_{j=3}^n  \tau x_j^2\right)
\]
must solve 
\[
\E_d u=(e_1\otimes_{\E_d}e_2) g.
\]
\smallskip

\textbf{The case $\a\neq 0$.} By \eqref{key} we now show that the following wave equations are in force.
\begin{align}
\beta^2\partial_{11}(\partial_{2 }g)  -  \alpha^2\partial_{22}(\partial_{2 }g)=&0\label{onde1}\\
\beta^2\partial_{11}(\partial_{1 }g)  -  \alpha^2 \partial_{22}(\partial_{1 }g)=&w(x_1)\label{onde2}.
\end{align}
for some $w:\R\rightarrow\R$. Indeed, \eqref{onde1} comes from
\begin{align*}
    \beta^2\partial_{11}(\partial_{2 }g) =\beta \partial_{1}(\beta\partial_{12 }g) =\alpha \beta \partial_{1}(\partial_{22 }g) =\alpha \partial_2(\beta\partial_{12}g)= \alpha^2 \partial_{22}(\partial_{2 }g).
\end{align*}
Equation \eqref{onde2} just come from the fact that (by \eqref{1jg})  
\[
\partial_{j}(\beta^2\partial_{11}(\partial_{1 }g)  - \alpha^2 \partial_{22}(\partial_{1}g))=0 \ \ \ \text{for all $j\geq 3$}
\]
and 
\begin{align}
\partial_{2}\left(\beta^2\partial_{11}(\partial_{1 }g)  - \alpha^2\partial_{22}(\partial_{1 }g)\right)=\partial_{1}\left(\beta^2\partial_{11}(\partial_{2 }g)  - \alpha^2\partial_{22}(\partial_{2}g)\right)=_{\eqref{onde1}}0.
\end{align}
Notice that $\partial_1 g,\partial_2 g$ are functions depending only on $x_1,x_2$ (by \eqref{1jg}). By the well-known D'Alambert formula for the general solutions of the planar wave equation we thus conclude
    \begin{align}
    \partial_{1}g =& f_0(x_1)+f_1^1(\alpha x_1-\beta x_2) + f_1^2 (\alpha x_1+\beta x_2) \label{eqn:BeforeInt1}\\
    \partial_2 g=& f_2^1(\alpha x_1-\beta x_2) + f_2^2 (\alpha x_1+\beta x_2).\label{eqn:BeforeInt2}
    \end{align}
We now integrate $\nabla g$ from \eqref{eqn:BeforeInt1}, \eqref{eqn:BeforeInt2}. We first observe the relation $\partial_{21}g=\partial_{12}g $ implying
\[
 -\beta (f_1^1)'(\alpha x_1-\beta x_2) + \beta (f_1^2)' (\alpha x_1+\beta x_2)=\alpha (f_2^1)'(\alpha x_1-\beta x_2) + \alpha (f_2^2)' (\alpha x_1+\beta x_2)
\]
which yields by computing on $\alpha x_1=\beta x_2$ and $\alpha x_1=-\beta x_2$:
\[
-\beta f_1^1(t)=\alpha f_2^1(t)+c_1 \ \ \ \beta f_1^2(t)= \alpha f_2^2(t) + c_2, \ \ 
\]
Then, accounting for \eqref{constancy} and the above we have
\begin{align*}
    g(x)=&\int_0^1 \nabla g(t x)\cdot x \d t= -\frac{\tau}{\beta} \sum_{j=3}^d x_j^2 + v_j x_j +\gamma_j \\
    &+\int_0^1 [f_0^1(t x_1)x_1+ f_1^1(t(\alpha x_1-\beta x_2))x_1+ f_1^2(t(\alpha x_1+\beta x_2)) x_1]\d t\\
    &+\int_0^1 [f_2^1(t(\alpha x_1-\beta x_2))x_2+ f_2^2(t(\alpha x_1+\beta x_2)) x_2 ] \d t\\
    =&\int_0^1 \nabla g(t x)\cdot x \d t= -\frac{\tau}{\beta} \sum_{j=3}^d x_j^2 + v_j x_j +\gamma_j \\
    &+\int_0^1 \left(f_0^1(t x_1)x_1 - \left(\frac{c_1}{\beta} -\frac{c_2}{\beta} \right)x_1 \right)\d t\\
    &-\frac{1}{\beta}\int_0^1 f_2^1(t(\alpha  x_1-\beta x_2))\left(\alpha x_1-\beta x_2\right)\d t\\
    &+\frac{1}{\beta}\int_0^1  f_2^2(t(\alpha  x_1+\beta x_2)) \left(\alpha x_1 +\beta x_2\right)  \d t  
    \end{align*}
and thus
\[
g(x)=h_1(x_1)+h_2(\alpha x_1+\beta x_2) + h_3(\alpha x_1-\beta x_2) - \sum_{j=3}^d\frac{\tau}{\beta}|x_j|^2+v\cdot x +\gamma
\]
for $v\in \langle e_1,e_2\rangle^{\perp}$.  By exploiting \eqref{key} again we also derive that
\[
-2\alpha \beta^2 h_3''(\alpha x_1-\beta x_2)= \zeta.
\]
Since $\alpha\neq 0$: 
\[
h_3(t)=-\frac{\zeta }{4\alpha \beta^2}t^2+\sigma t+\gamma
\]
By observing that
\begin{align*}
(\alpha x_1-\beta x_2)^2=&(\alpha x_1+\beta x_2)^2 - 4\alpha \beta x_1 x_2\\
\alpha x_1-\beta x_2 =& 2\alpha x_1 - (\alpha x_1+\beta x_2)
\end{align*}
and up to redefine $h_1$ and $h_2$ we can rewrite $g$ as
\begin{align}
    g(x)=&h_1(x_1)+h_2(\alpha x_1+\beta x_2)  +\frac{\zeta }{\beta}x_1x_2 - \frac{\tau}{\beta}\sum_{j=3}^n|x_j|^2+v\cdot x \nonumber\\
   =& h_1(x_1)-\frac{\zeta  \alpha }{\beta^2}x_1^2+h_2(\alpha x_1+\beta x_2)  +\frac{\zeta }{\beta^2}x_1(\alpha x_1+\beta x_2) - \frac{\tau}{\beta}\sum_{j=3}^n|x_j|^2+v\cdot x\nonumber\\
      =& H_1(x_1)+H_2(\alpha x_1+\beta x_2)  +\frac{\zeta }{\beta^2}x_1(\alpha x_1+\beta x_2) - \frac{\tau}{\beta}\sum_{j=3}^n|x_j|^2+v\cdot x
\end{align}
Setting 
\[
\bar{u}:= \psi_1(x\cdot a)b+\psi_2(x\cdot b)a 
\]
with $\psi_1'(t)= H_1(t) $, $\psi_2'(t)= H_2(t) $ we have
\[
\E_d \bar{u}=\left[e_1\odot (\alpha e_1+\beta e_2) - \frac{\alpha}{n}\Id \right]\left(H_1(x_1)+H_2(\alpha x_1+\beta x_2)\right).
\]
Thus $Q:=u-\bar{u}$ satisfies
\begin{equation*}
\E_d Q= \left[e_1\odot (\alpha e_1+\beta e_2) - \frac{\alpha}{n}\Id \right]\left((x\cdot v) +\eta x_1 (\alpha x_1+\beta x_2) - \vartheta\sum_{j=3}^n  x_j^2 \right).
\end{equation*}
and hence $u=\bar{u}+Q+L$ as claimed with $\eta= \frac{\zeta }{\beta^2}$, $\vartheta=\frac{\tau}{\beta} $.  
\end{proof}
\subsubsection{Features of the polynomial solutions}\label{sbsbsct:FeaturesPolSolNONPAR}
In this section we derive some specific features of the polynom $Q$ satisfying \eqref{eqn:PolynomialRemainder}. This is required in order to pass to the limit in our density argument. To do this we observe that, since
\[
w=(v\cdot x) [a(b\cdot x)+b(a\cdot x)]-v(a\cdot x)(b\cdot x),
\]
due to Remark \ref{rmk:CrucialKNOWN} solves
\[
\E_d w=(a\otimes_{\E_d}b) (2v\cdot x),
\]
we are just left to compute the solution to
\begin{equation}\label{polyinter}
\E_d P=(a\otimes_{\E_d}b) \left[  \eta (a\cdot x)(b\cdot x) -  \vartheta  \sum_{j=3}^n (w_j\cdot x)^2\right]
\end{equation}
for $\{w_j\}_{j=3}^n$ orthonormal basis of $\langle a,b\rangle^{\perp}$.  We will not need the explicit solutions but just the features required to pass to the limit from $C^{\infty}$ solutions of \eqref{eqn:RigidPolar} to $\BDD$ solutions of \eqref{eqn:RigidPolarMeasure}.

\begin{proposition}\label{prop:PolynomialPartSolutions}
Set $a=e_1, b=\alpha e_1+\beta e_2$, $\beta \neq 0$ and consider the equation

\begin{equation}\label{eqn:PolynomialRemainderInCoordinate}
\E_d P= e_1 \otimes_{\E_d} (\alpha e_1+\beta e_2) \left[\eta x_1(\alpha x_1+\beta x_2)-\vartheta  \sum_{j=3}^n x_j^2\right].
\end{equation}
Then any particular solution $P\in C^2$ of \eqref{eqn:PolynomialRemainderInCoordinate} is a third order degree polynom and satisfies
    \begin{equation}\label{eqn:DerivativesSpecialSolution}
    \partial_{123} P_3 (x)=\vartheta \beta , \ \ \partial_{223}P_3(x)=-\frac{2\alpha \vartheta  -\eta\beta^2 }{2}.
    \end{equation}

\end{proposition}
\begin{remark}\label{rmk:ngeq4}
We underline that, for $n\geq 4 $ a slightly stronger result - not needed in the purpose of computing the limit - hold: for $n\geq 4$ \eqref{eqn:PolynomialRemainderInCoordinate} has a solution if and only if $\eta=\frac{2\alpha\vartheta }{\beta^2}$. Indeed for $n\geq 4$ from \eqref{TWOone}, \eqref{THREEtwo} we get $\partial_{ii}f=0$ for all $i\geq 2$. Suppose now that that we have a solution $P\in C^2$ for $n\geq 4$ and set as in previous computation
\[
g:=\eta x_1(\alpha x_1+\beta x_2)-\vartheta\sum_{j=3}^n x_j^2, \ \ \ f:=\frac{\dive(P)-\alpha g}{n}
\]
Then \eqref{ONEtwo}  gives $\partial_{11}f =  2\alpha \vartheta$. This, plugged in \eqref{ZEROone} and combined with the fact that $\partial_{22}f=0$, gives
\[
2\alpha \vartheta=\beta \partial_{12} g -\alpha \partial_{22}g=\beta^2 \eta  \ \ \ \Rightarrow \ \ \eta=\frac{2\alpha \vartheta}{\beta^2}.
\] 
This is not the case in $n=3$ where a particular solution to \eqref{eqn:PolynomialRemainderInCoordinate} can be provided even for independent $\eta,\vartheta$.
\end{remark}

\begin{proof}
Let $P$ be a particular solution of \eqref{eqn:PolynomialRemainderInCoordinate}. We start by observing that, having set $M:=(a\otimes_{\E_d} b)$
\begin{align}
    \E_d (\partial_1 P )=&  M\eta [\alpha x_1+(\alpha x_1+\beta x_2)]\label{Sold1p}\\
    \E_d (\partial_2 P )=& M \beta \eta x_1\label{Sold2p}\\
    \E_d (\partial_j P )=& - M 2\vartheta    x_j \ \ j\geq 3\label{Soldjp}
\end{align}
All this equation have the structure of \eqref{eqn:KnownPart} and thus the solutions are all given, up to some element of $\mathrm{Ker}(\E_d)$, by Formula \eqref{eqn:SolToKnownPart}. In particular this tells that $\partial_{k}P$ is a second order degree polynom. Once integrated we get that $P$ is a polynom with degree less or equal to $3$. Note that, since the right hand side of \eqref{eqn:PolynomialRemainderInCoordinate} is a second order degree polynom and $\E_d$ is a first order differential operator, $P$ needs to have at least one term of third degree: so $P$ is a third degree polynom. We now focus on the proof of \eqref{eqn:DerivativesSpecialSolution}.\\
\smallskip

Formula \eqref{eqn:SolToKnownPart} gives the exact structure:
\begin{align*}
    \partial_1 P  =&   \eta  \left[\frac{(\alpha x_1+\beta x_2)^2}{2}e_1 + \frac{\alpha x_1^2}{2}(\alpha e_1+\beta e_2)\right]+L_1(x)\\
 \partial_2 P  =&  \frac{\beta \eta}{2}    x_1^2 (\alpha  e_1+\beta e_2) + L_2(x)\\
   \partial_j P  =& - \vartheta  [ e_1(\alpha  x_1+\beta x_2)x_j+(\alpha  e_1+\beta e_2) x_1x_j - e_j x_1(\alpha x_1+\beta x_2)]+ L_j(x)\ \ j\geq 3
\end{align*}
for $L_i\in \mathrm{Ker}(\E_d)$. Since $L_i(x)= A_ix + (s_i \cdot x)x - s_i \frac{|x|^2}{2}+b_i$ for some $s_i,b_i\in \R^n$, $A=R_i+\gamma_i\II$ with $R_i\in \mathbb{M}_{sym_0}^{n\times n}$, $\gamma_i\in \R$ then
\[
\partial_{k} L_i(x)=A_ie_k + (s_i \cdot e_k)x +(s_i \cdot x)e_k - s_i x_k.
\]
Then
\begin{align*}
    \partial_{123} P= \vartheta  \beta e_3  +\partial_{12} L_3(x), \ \ \partial_{123} P_3=\vartheta \beta +\partial_{12} L_3(x)\cdot e_3.
\end{align*}
Since
\[
\partial_{12} L_3(x)=  (s_3 \cdot e_2)e_1 +(s_3 \cdot e_1)e_2  
\]
Then we have immediately
\[
\ \partial_{123} P_3=\vartheta \beta .
\]
Also
\begin{align}\label{almostP3}
    \partial_{223} P_3= \partial_{22} (L_3(x)\cdot e_3)=  - (s_3\cdot e_3).
\end{align}
To compute this value we now take advantage of Schwarz Theorem to derive information on $s_1,s_2,s_3$. In particular by $\partial_{13}P=\partial_{31}P$ we obtain
\begin{align*}
    - \vartheta [ e_1\alpha  x_3+(\alpha e_1+\beta e_2)x_3 - e_3  (2\alpha  x_1+\beta x_2)]+ \partial_1 L_3(x)=\partial_3 L_1(x).
\end{align*}
By computing in $x=0$ we get rid of the affine part $A_3e_1=A_1e_3$ and
\begin{align*}
    - \vartheta  [ e_1\alpha  x_3 +(\alpha e_1+\beta e_2)x_3 &- e_3  (2\alpha  x_1+\beta x_2)]+ (s_3 \cdot e_1)x +(s_3 \cdot x)e_1 - s_3 x_1\\
    &=(s_1 \cdot e_3)x +(s_1 \cdot x)e_3- s_1 x_3.
\end{align*}
and computing in $x=e_1$, $x=e_2$ and $x=e_3$ gives
\begin{align*}
  2\vartheta  \alpha   e_3  + 2(s_3 \cdot e_1)e_1 -s_3   
    &=(s_1 \cdot e_3)e_1 +(s_1 \cdot e_1)e_3 \\
    \vartheta  \beta e_3+ (s_3 \cdot e_1)e_2 +(s_3 \cdot e_2)e_1   
    &=(s_1 \cdot e_3)e_2 +(s_1 \cdot e_2)e_3 \\
      - 2\alpha \vartheta    e_1 -\vartheta \beta e_2 + (s_3 \cdot e_1)e_3 +(s_3 \cdot e_3)e_1  
    &=2(s_1 \cdot e_3)e_3 - s_1.
\end{align*}
Giving
\begin{equation}\label{p13p31ONE}
s_3\cdot e_1=s_1\cdot e_3, \ \ s_3\cdot e_2=0, \ \ (s_1 \cdot e_2)=\vartheta  \beta 
\end{equation}
and
\begin{align}  
s_1&=(2\alpha \vartheta  - (s_3\cdot e_3)  )  e_1 +\vartheta \beta e_2 + (s_1\cdot e_3) e_3\label{p13p31TWO}\\
  s_3&=(s_1\cdot e_3)e_1 + (2\alpha \vartheta      -  (s_1 \cdot e_1))e_3 \label{p13p31THREE}  
\end{align}
From $\partial_{23}P=\partial_{32}P$ we get
\begin{align*}
    -\vartheta [\beta x_3 e_1-\beta e_3 x_1] +\partial_2 L_3(x)=\partial_3 L_2(x).
\end{align*}
Again computing at $x=0$ allows to ignore the part $A_3e_2=A_2e_3$ and obtain
\begin{align*}
    -\vartheta  [\beta x_3 e_1&-\beta e_3 x_1] + (s_3 \cdot e_2)x +(s_3 \cdot x)e_2 - s_3 x_2\\
    &=(s_2 \cdot e_3)x +(s_2 \cdot x)e_3 - s_2 x_3.
\end{align*}
Computing at $x=e_1$, $x=e_2$ and $x=e_3$ yields  
\begin{align*}
   \vartheta  \beta e_3  + (s_3 \cdot e_2)e_1 +(s_3 \cdot e_1)e_2  
    &=(s_2 \cdot e_3)e_1 +(s_2 \cdot e_1)e_3 \\
       (s_3 \cdot e_2)e_2 +(s_3 \cdot e_2)e_2 - s_3  
    &=(s_2 \cdot e_3)e_2 +(s_2 \cdot e_2)e_3  \\
    -\vartheta   \beta  e_1  + (s_3 \cdot e_2)e_3 +(s_3 \cdot e_3)e_2  
    &=(s_2 \cdot e_3)e_3 +(s_2 \cdot e_3)e_3 - s_2  
\end{align*}
that result in
\[
(s_3 \cdot e_1) =0, \ \ (s_3\cdot e_2)=(s_2\cdot e_3)=0 
\]
which combined with \eqref{p13p31ONE}, \eqref{p13p31TWO} and \eqref{p13p31THREE} gives
\begin{align}
s_2&=\vartheta  \beta  e_1 -(s_3 \cdot e_3)e_2 \label{p23p32TWO} \\
      s_3&=-(s_2\cdot e_2)e_3 \nonumber.
\end{align}
Finally by $\partial_{21}P=\partial_{12}P$ we get, still after neglecting the affine part $A_1e_2=A_2e_1$
\begin{align*}
    \eta \beta(\alpha x_1+\beta x_2)e_1 &+(s_1 \cdot e_2)x +(s_1 \cdot x)e_2 - s_1x_2\\
    &=\eta\beta x_1(\alpha e_1+\beta e_2)+(s_2 \cdot e_1)x +(s_2 \cdot x)e_1 - s_2 x_1.
\end{align*}
computing in $x=e_1$ yields finally
\begin{align*}
    \eta \beta \alpha e_1  +(s_1 \cdot e_2)e_1 +(s_1 \cdot e_1 )e_2  
     =&\eta\beta  (\alpha e_1+\beta e_2)+2(s_2 \cdot e_1)e_1  - s_2   
\end{align*}
which means by  \eqref{p13p31ONE} 
\begin{align*}
     s_2 =&  \vartheta \beta e_1  + (\eta\beta^2   - (s_1 \cdot e_1 ))e_2  
\end{align*}
The above combined with \eqref{p23p32TWO} and \eqref{p13p31THREE} gives
\[
(s_2 \cdot e_2)=-(s_3\cdot e_3)\ \Rightarrow \ (\eta\beta^2- (s_1 \cdot e_1 ))=-(s_3\cdot e_3)=-(2\alpha \vartheta -(s_1\cdot e_1))
\]
that gives
\[
  (s_1\cdot e_1)= \frac{2\alpha \vartheta +\eta\beta^2}{2},  \ \ (s_3\cdot e_3)=    \frac{ 2\alpha \vartheta -\eta\beta^2}{2}.
\]
By plugging $(s_3\cdot e_3)$ into \eqref{almostP3} we conclude.
\end{proof}
\begin{remark}
Observe that, with (a discrete amount of) patience the approach proposed in the proof of Proposition \ref{prop:PolynomialPartSolutions} allows one to build a complete particular polynomial solution of the PDE \eqref{eqn:PolynomialRemainderInCoordinate}. Indeed the choice of $s_1,s_2,s_j$ are force by the gradient structure of $\nabla P$ and thus, once identified it is possible to choose $\{A_i\}_{i=}^{n} \subset \mathbb{M}^{n\times n}_{skew}$, $\{b_i\}_{i=1}^{n}\subset \R^d$ so that we can integrate $\nabla P$ to get $P$. By doing so it is possible to observe that there is a choice ($A_i=0$, $b_i=0$) so that a particular solution $P$ is also a \textit{homogeneous} polynom of degree $3$.
\end{remark}

\subsubsection{Approximation argument}\label{sbsbsct:ApproximationNONPAR}
The following is a standard approximation argument.
 \begin{lemma}[Approximation Lemma]\label{lem:Approx}
     Let $u\in \BDD(\R^n)$. Let $\varrho_\e\in C^{\infty}_c(B_{\e}(0))$ be a mollifying kernel. Then it holds $u_\e:=u\star \varrho_\e \in C^{\infty}$ and
  \begin{itemize}
      \item[(1)] $\E_d (u_\e)(x)= (\E_d u\star \varrho_\e)(x)$ ;
      \item[(2)] $u_\e\rightarrow u$ in $L_{\loc}^1(\R^n)$.
  \end{itemize}
 \end{lemma}
\begin{lemma}\label{lem:unid}
    Let $s_\e(x)=\psi_\e(x\cdot \ell)$ for some $\ell \in \R^n$ and for some sequence of measurable functions $\psi_\e\in \BV_{\loc}(\R)$. Suppose that $s_\e\rightarrow s$ in $L^1_{\loc}(\R^n)$. Then $s(x)=\psi(x\cdot \ell)$ for some measurable function $\psi\in L^1_{\loc}(\R)$.
\end{lemma}
\begin{proof}
Suppose without loss of generality that $\ell=e_1$. Since $\psi_\e\in \BV_{\loc}(\R)$ then $s_\e\in \BV_{\loc}(\R^n)$ with $Ds_\e=e_1 D\psi_\e(\d x_1)\otimes \L^{n-1}(\d x_2\ldots\d x_n)$. Here we recall the notation 
\[
\mu(\d x_1)\otimes \nu(\d x_2\ldots\d x_n)(\varphi):=\int_{\R^n} \varphi \d \mu(x_1)\d\nu(x_2)\ldots \d\nu(x_n).
\]
For $\varphi\in C^{\infty}_c(\R^n)$ we also have for $k\neq 1$:
\[
\int_{\R^n} s \partial_k \varphi\d x=\lim_\e \int_{\R^n} s_\e \partial_k \varphi\d x=- \lim_\e \int_{\R^n} (e_1\cdot e_k) \varphi\d D\psi_\e (\d x_1)\d x_2\ldots\d x_d=0.
\]
Thence, distributionally $\partial_k s=0$ for all $k\neq 1$. This implies that $s(x)=\psi(x\cdot e_1)$ for a measurable function $\psi:\R\rightarrow \R$.
\end{proof}
Finally, to handle the polynomial part we need to employ the following Lemma
\begin{lemma}\label{lem:PolynomPart}
    Let $Q_\e:\R^n\rightarrow \R$ be a polynoms in $x_1,\ldots,x_n$ of degree $k$. Suppose that $Q_\e\rightarrow Q$ in $L_{loc}^1 $ for some $Q\in L^1_{loc}$. Then $Q$ is a polynom of degree at most $k$  and the coefficients of $Q_\e$ converges to the coefficient of $Q$.
\end{lemma}
\begin{proof}
Fix an open bounded set $A\subset \R^n$ and note that the space 
    \[
    S:=\{Q: A\rightarrow \R \ | \ \text{$Q$ is a polynom of degree at most $k$}\}
    \]
is a closed finite dimensional vector space of finite dimension. Let $\underline{i}\in\{0,1,\ldots,k\}^n$ be a multindex and, for $Q\in S$, $c^Q_{\underline{i}}$ be denoting the coefficient of $Q$ in front of a monom $x^{\underline{i}}:=x_1^{i_1}\ldots x_n^{i_n}$, $|\underline{i}|=i_1+\ldots+i_n \leq k$. Then the application
\[
T:S\rightarrow \R^N, \ \ T(Q):=(c^Q_{\underline{i}})_{\underline{i}\in\{0,1,\ldots,k\}^n}
\]
is a linear application between finite dimensional vector spaces and hence is continuous. Thus, since $Q_\e$ is converging in $L_{loc}^1$, we have $Q_\e\rightarrow Q$ in $L^1(A)$ for some $Q\in S$. Thence
\[
|T(Q_\e)-T(Q)| \leq C_T\|Q_\e - Q\|_{L^1}\rightarrow 0.
\]
By the very definition of $T$ we have that the coefficient of $Q_\e$ converges to those of $Q$.
\end{proof}
We are now in the position for proving Theorem \ref{thm:MainRigidityNonParallel}.
 \begin{proof}[Proof of Theorem \ref{thm:MainRigidityNonParallel}]
Let $u\in \BDD(\R^n)$ satisfying 
\[
\E_d u=(a\otimes_{\E_d} b)\nu
\]
Since $a,b$ are not parallel, without loss of generality, we can assume that $a=e_1, b=\alpha e_1+\beta e_2$ for $\beta  \neq 0$ and $\alpha\in \R$. Consider $u_\e$ the approximation as in Lemma \ref{lem:Approx} and note that
\[
\E_d u_\e =  \E_d u \star \varrho_\e= (a\otimes_{\E_d} b) (\nu \star \varrho_\e)(x).
\]
We invoke Proposition \ref{propo:RigForRegFunctions} with $g_\e= (\nu \star \varrho_\e)\in C^{\infty}$ and we conclude that 
\[
u_\e(x)=a\psi_1^{\e}(x\cdot b)+b \psi_2^\e(x\cdot a) + Q_\e(x)+ L_\e(x).
\]
with $Q_\e$ solving \eqref{eqn:PolynomialRemainder}. We just need to show that the claimed structure is stable under the limit as $\e\to 0$, yielding the thesis also on $u$. Notice that $u_\e \rightarrow u$ and since 
\[
\sup_{\e}\{|\E_d(u_\e-L_\e)|(A)\}<+\infty
\]
we have $u_\e-L_\e\rightarrow \bar{u}$. Since 
\[
\E_d u =\wlim_{\e\rightarrow 0}\E_d u_\e=\wlim_{\e\rightarrow 0 }\E_d( u_\e-L_\e)=\E_d \bar{u},
\]
where $\wlim $ denote the \textit{weak star limit}, we have also that $L:=u-\bar{u}\in \mathrm{Ker}(\E_d)$ and that 
\[
L_\e=u_\e-(u_\e-L_\e)\rightarrow u-\bar{u}=L.
\] 
We now complete the proof by separately analyzing the polynomial part and the one dimensional part.\\
\smallskip

\textbf{Step one:} \textit{limit of the polynomial part}. Suppose that $a=e_1$, $ b=\alpha e_1 +\beta e_2$.  We know that (up to an element of $\mathrm{Ker}(\E_d)$) it must hold (cf. Remark \ref{rmk:CrucialKNOWN})
\[
Q_\e(x)=(v_\e \cdot x)[e_1 (\alpha x_1+ \beta x_2)+(\alpha e_1 + \beta e_2)x_1 ] - v_\e x_1(\alpha x_1+ \beta x_2)+P_\e(x) 
\]
for some $v_\e\in \langle e_1,e_2\rangle^\perp$ and for some $P_\e$ solving \eqref{eqn:PolynomialRemainderInCoordinate}
\[
\E_d P_\e =e_1\otimes_{\E_d}(\a e_1 + \b e_2)  \left[\eta_\e x_1(\alpha x_1+\beta x_2) -\vartheta _\e \sum_{j=3}^n x_j^2\right].
\]
Since $(u_\e-L_\e)\rightarrow \bar{u}$ in $L^1$ for any $j\geq 3$ we have
\[
Q_\e(x)\cdot e_j=(u_\e-L_\e)\cdot e_j\rightarrow (\bar{u}\cdot e_j) \ \ \text{in $L^1$}.
\]
But
\[
Q_\e(x)\cdot e_j=-(v_\e\cdot e_j)x_1(\alpha x_1+\beta x_2)+P_\e \cdot e_j
\]
We use Proposition \ref{prop:PolynomialPartSolutions} to infer that a solution to \eqref{eqn:PolynomialRemainderInCoordinate} must be of the form
\[
(P_\e(x)\cdot e_3):=\vartheta _\e \beta  x_1x_2x_3  -(2\alpha \vartheta _\e -\eta_\e\beta^2 )\frac{x_2^2x_3}{4}+\sum_{\underline{i}\in \{0,\ldots,3\}^n\setminus\{(1,1,1,0\ldots ,0), (0,2,1,0,\ldots,0)\}} (c_{\underline{i}}(\eta_\e,\vartheta _\e) \cdot e_3)\, x^{\underline{i}} 
\]
and hence 
\begin{align*}
Q_\e\cdot e_3=& -( v_\e\cdot e_3 )x_1 (\alpha x_1+\beta x_2) +\vartheta _\e \beta  x_1x_2x_3  -(2\alpha \vartheta _\e -\eta_\e\beta^2 )\frac{x_2^2x_3}{4}\\
&+\sum_{\underline{i}\in \{0,\ldots,3\}^n\setminus\{(1,1,1,0\ldots ,0), (0,2,1,0,\ldots,0)\}} (c_{\underline{i}}(\eta_\e,\vartheta _\e )\cdot e_3)\,  x^{\underline{i}} .
\end{align*}
By means of Lemma \ref{lem:PolynomPart} we have that the coefficients of $Q_\e\cdot e_3$ must converge to something and thus $\eta_\e,\vartheta _\e $ and $( v_\e\cdot e_3 )$ converges to something. In particular we have that $v_\e\rightarrow v$,  $\vartheta _\e \rightarrow \vartheta$ and $\eta_\e\rightarrow \eta$. Since the coefficient of $Q_\e$ are all expressed as polynomial functions of $\eta_\e,\vartheta _\e $ and $v_\e$ we have that $Q_\e\rightarrow Q$ for $Q$ solving \eqref{eqn:PolynomialRemainder}. \\
\smallskip

\textbf{Step two:} \textit{limit of the one-dimensional part}. Set $w_\e:=u_\e-Q_\e-L_\e$. Thanks to Step one, two and a standard compacntess argument we know that $w_\e\rightarrow w:=\bar{u}-Q$ in $L^1$ for some $Q$ solving \eqref{eqn:PolynomialRemainderMainThm}. Since
\[
w_\e(x)=a\psi_1^{\e}(x\cdot b)+b \psi_2^\e(x\cdot a),
\]
pick now $z\in b^{\perp}$ for which $a\cdot z \neq 0$ and note that
\[
s^z_\e(x):=\frac{z}{a\cdot z} \cdot w_\e=\psi_1^\e(x\cdot b)
\]
where $\psi_1$ is actually smooth since $w_\e$ is smooth. Since $w_\e \rightarrow w$ in  $L_{\loc}^1(\R^n)$ hence $s^z_\e \rightarrow \frac{z}{a\cdot z}\cdot w$. From the otherside, by applying Lemma \ref{lem:unid} we conclude $s^z_\e \rightarrow s(x)=\psi_1(x\cdot b)$ for a measurable function $\psi_1\in L^1_{\loc}(\R)$ independent of $z$ (since $s^z_\e(x)=\psi_1^\e(x\cdot b)$). Thence
 \begin{equation}\label{dude}
 z\cdot w=(a\cdot z) \psi_1(x\cdot b)  \ \ \text{for all $z\in b^{\perp}$}.
 \end{equation}
 Analogously, starting from $h\in a^{\perp}$ we conclude
 \begin{equation}\label{dude2}
 h\cdot w= (b\cdot h) \psi_2(x\cdot a)  \ \ \text{for all $h\in a^{\perp}$},
 \end{equation}
 for some measurable function $\psi_2\in L^1_{\loc}(\R)$.
 We now choose $z_1,\ldots ,z_{n-1}$ orthonormal basis of $b^{\perp}$ and $h_1,\ldots,h_{n-1}$ orthonormal basis of $a^{\perp}$. Observe - by \eqref{dude} -  that 
 \begin{align*}
w(x)=&\sum_{i=1}^{n-1}z_i (z_i\cdot w(x)) + \left(\frac{b}{|b|}\cdot w(x)\right)\frac{b}{|b|}=\sum_{i=1}^{n-1}z_i \left(a\cdot z_i\right) \psi_1(x\cdot b) +\left(\frac{b}{|b|}\cdot w(x)\right)\frac{b}{|b|}.
 \end{align*}
 By also expressing 
 \[
 \frac{b}{|b|}=\sum_{i=1}^{n-1} \left( \frac{b}{|b|}\cdot h_i \right)h_i + \left(\frac{b}{|b|}\cdot \frac{a}{|a|} \right)\frac{a}{|a|}
 \]
and using \eqref{dude2}, we can further write
  \begin{align}
 w(x) =&\sum_{i=1}^{n-1}z_i \left(a\cdot z_i\right) \psi_1(x\cdot b) + \sum_{i=1}^{n-1}\left(\frac{b}{|b|}\cdot h_i\right)(h_i\cdot w(x))\frac{b}{|b|}  + \left(\frac{b}{|b|}\cdot \frac{a}{|a|}\right)\left(\frac{a}{|a|}\cdot w(x)\right)\frac{b}{|b|}\nonumber\\
=& a\psi_1(x\cdot b) + \psi_2(x\cdot a) \frac{b}{|b|^2}\sum_{i=1}^{d-1}(b\cdot h_i)^2+\psi_2(x\cdot a) \frac{b}{|b|^2}\left(b\cdot \frac{a}{|a|}\right)^2 \nonumber   \\
&+\left(\frac{b}{|b|}\cdot \frac{a}{|a|}\right)\left(\frac{a}{|a|}\cdot w(x)\right)\frac{b}{|b|}-\psi_1(x\cdot b) \left(a\cdot \frac{b}{|b|}\right) \frac{b}{|b|}- \psi_2(x\cdot a) \frac{b}{|b|^2}\left(b\cdot \frac{a}{|a|}\right)^2 \nonumber\\
=& a\psi_1(x\cdot b) + b\psi_2(x\cdot a)  \nonumber \\
&+ \left(\frac{b}{|b|}\cdot \frac{a}{|a|}\right)\frac{b}{|b|}\left[\left(\frac{a}{|a|}\cdot w(x)\right) -|a|\psi_1(x\cdot b)  -\left(b\cdot \frac{a}{|a|}\right) \psi_2(x\cdot a) \right]\label{finalfinal}
 \end{align}
 We immediately conclude if $b\cdot a=0$. Otherwise we multiply \eqref{finalfinal} by  $h\in a^{\perp}$ and, using \eqref{dude2}, we have
\begin{align*}
(b\cdot h) \psi_2(x\cdot a) =&   (b\cdot h)\psi_2(x\cdot a) \\
&+ \left(\frac{b}{|b|}\cdot \frac{a}{|a|}\right)\frac{b\cdot h }{|b|}\left[\left(\frac{a}{|a|}\cdot w(x)\right) -|a|\psi_1(x\cdot b)  -\left(b\cdot \frac{a}{|a|}\right) \psi_2(x\cdot a) \right]
 \end{align*}
 yielding
 \[
\left(\frac{b}{|b|}\cdot \frac{a}{|a|}\right)\frac{b\cdot h}{|b|}\left[\left(\frac{a}{|a|}\cdot w(x)\right) -|a|\psi_1(x\cdot b)  -\left(b\cdot \frac{a}{|a|}\right) \psi_2(x\cdot a) \right]=0 \ \ \text{for all $h\in a^{\perp}$}.
 \]
 By selecting an $h$ for which $(b\cdot h)\neq 0$ (which exists since $a\neq \lambda b$) we conclude 
 \[
 \left[\left(\frac{a}{|a|}\cdot w(x)\right) -|a|\psi_1(x\cdot b)  -\left(b\cdot \frac{a}{|a|}\right) \psi_2(x\cdot a) \right]=0
 \]
 which implies, by \eqref{finalfinal}
 \[
w(x)=a\psi_1(x\cdot b) + b\psi_2(x\cdot a).
 \]
 \smallskip
 
\textbf{Step three:} \textit{bounded variation of the one-dimensional part.} We are just left to show that the $\psi'$s are $\BV_{\loc}$. We thus express (for the sake of shortness) $a=  e_1$, $h=e_2$ and $b=\alpha e_1+\beta e_2$. Then
\[
u(x)=  e_1 [\psi_1(\alpha x_1+\beta  x_2) +\alpha \psi_2( x_1)]+\beta  e_2\psi_2(x_1)  
\]
We choose $\varphi \in C^{\infty}_c([-1,1])$ with $\|\varphi \|_{\infty}\leq 1$ and $\eta,\eta_3,\ldots,\eta_{n}\in C^{\infty}_c(\R)$ with $\int  \eta_i \d t=1$. We consider thus
\[
\Phi(x):= \varphi(x_1)\eta(x_2)\eta_3(x_3)\ldots \eta_n(x_n) (e_1\odot e_2).
\]
Note that 
\[
2\dive(\Phi)(x)=\varphi(x_1)\eta'(x_2)\eta_3(x_3)\ldots \eta_n(x_n)) e_1 + \varphi' (x_1)\eta_2(x_2)\eta_3(x_3)\ldots \eta_n(x_n)) e_2.
\]
We thus apply \eqref{eqn:GGformula} and the fact that $\int \eta_i=\int \eta=1$ to infer
\begin{align*}
2&\int_{\R^n} \Phi \cdot \d \E_d u(x)=2\int_{\R^n} (\dive(\Phi) \cdot u) \d x\\
&=\int_{\R^n} \varphi(x_1)\eta'(x_2)\eta_3(x_3)\ldots \eta_n(x_n)) u_1 \d x+  \int_{\R^n}\varphi' (x_1)\eta(x_2)\eta_3(x_3)\ldots \eta_n(x_n)) u_2\d x \\
&=\int_{\R^2} \varphi(x_1)\eta'(x_2) [ \psi_1(\alpha x_1+\beta   x_2) +\alpha \psi_2( x_1)] \d x_1 \d x_2+ \beta  \int_{\R}\varphi' (x_1) \psi_2(x_1) \d x_1.
\end{align*}
By rearranging the above relation, setting $K=\spt(\Phi)$, $K_2:=\spt(\varphi \eta )\subset \R^2$, we obtain
\begin{align*}
\left|\int_{\R} \varphi'(x_1) \psi_2 (x_1) \d x_1\right| \leq C \left[|\E_d u|(K) +  \int_{K_2}[|\psi_1(\alpha x_1+\beta x_2)|+|\psi_2(x_1)|] \d x_1 \d x_2\right]
\end{align*}
with the constant $C$ possibly depending on $\eta, \eta'_i$ but not on $\varphi$.
For any fixed $T>0$, the right hand side is independent of  $\varphi \in C^{\infty}_c([-T,T])$  with $\|\varphi\|_{\infty}\leq 1$ - being $K\subset [-T,T]\times \spt(\eta)\times \spt(\eta_3)\times \ldots\times \spt(\eta_d)$ - yielding that $\psi_2\in \BV_{\loc}(\R)$. \\

Writing  $a=(a\cdot \frac{b}{|b|}) \frac{b}{|b|}+(a\cdot z)z$ for some $z\in b^{\perp}$) (setting again without loss of generality $b= |b| e_1$, $z=e_2$ and $a=\a_1 e_1+\a_2 e_2$ for some $\a_1,\a_2$) and replicating the above argument (since we have - formally - inverted the role of $x_1$, $x_2$ we can test with the same $\Phi(x)$) we obtain also $\psi_1\in \BV_{\loc}(\R)$.
\end{proof}

\subsection{Rigidity for parallel vectors ($\beta =0$)} \label{sbsct:RigidityParallel}
In this Section we instead provide the proof to the following
\begin{theorem}\label{thm:MainRigidityParallel}
Let $u\in \BDD(A)$ for a connected open set $A\subset \R^n$. Suppose that
\[
\E_d u=(a\otimes_{\E_d}a)\nu
\]
for some $a\in \R^n$, and some Radon measure $\nu \in \mathcal{M}(A;\R)$. Then there exists functions $F \in \BV_{loc}(\R))$, $\{P_j\}_{j=2}^n\subset W^{1,1}_{loc}(\R)$ with $P_j'\in \BV_{loc}$, $G\in W^{1,1}_{loc}(\R)$ with $G'\in \BV_{loc}(\R)$ such that
\begin{equation}\label{eqn:ShapeUForBetaEqualZero}
\begin{split}    
u=& F(a \cdot x) a + \left(\sum_{j=2}^n P_j'(a\cdot x)(w_j\cdot x) + \frac{(w_j\cdot x)^2}{2}G'(a\cdot x)\right)a\\
&-\sum_{j=2}^n ((w_j\cdot x) G(a\cdot x)+P_j(a\cdot x ))w_j+\varrho Q(x)+L(x)
\end{split}
\end{equation}
for $\{w_2,\ldots,w_n\}$ orthonormal basis of $a^{\perp}$, for some $L\in \mathrm{Ker}(\E_d)$, $\varrho \in \R$ and for some third order degree polynom $Q$ solving 
\begin{equation}\label{eqn:PolynomialRemainderMainThmBETAZERO}
\E_d Q=(a\otimes_{\E_d} a) \left( (w_2\cdot x)^2-(w_3\cdot x)^2\right).
\end{equation}
Moreover if $n\geq 4$ then $\varrho=0$.
\end{theorem}
\begin{remark}\label{rmk:cfParallelAndNotParallel}
Note that since having $\beta=0$ result in fewer controls on the relevant quantities (cf.  Lemma \ref{lem:tecnico}) there is no surprise that we have a more complex structure on the $u$, compared to the case $\beta \neq0$. However still the lowest regularity info's is on the one directional part along $a$ which is only $\BV$ and no more.    
\end{remark}
As before we place ourself in coordinate so that $a=e_1$ and we focus first on the $C^{\infty}$ case.  \begin{lemma}\label{lem:rigForRegUnidir}
    Let $u\in C^{\infty}(\R^n;\R^n)$, $g\in C^{\infty}(\R^n)$ be such that
    \begin{equation}\label{eqn:PDEforBEqualZero}
        \E_d u= \left[ e_1\odot e_1 - \frac{1}{n} \Id \right] g.
    \end{equation}
Then necessarily, for some $\psi, h, p_j:\R \rightarrow \R$ it holds
\begin{equation}\label{eqn:ShapeOfGBetaEqualZero}
g=h(x_1)+\sum_{j=2}^n p_j(x_1) x_j 
+\psi(x_1)\sum_{j=2}^n \frac{x_j^2}{2}+\varrho(x_2^2-x_3^2)
\end{equation}
and consequently
\begin{equation}\label{eqn:ShapeRegUForBetaEqualZero}
\begin{split}    
u=&(H(x_1)-\Psi(x_1))e_1 + \left(\sum_{j=2}^n P_j'(x_1)x_j + \frac{x_j^2}{2}\Psi''(x_1)\right)e_1\\
&-\sum_{j=2}^n (x_j \Psi'(x_1)+P_j(x_1))e_j+\varrho Q(x)+L(x)
\end{split}
\end{equation}
for $L\in \mathrm{Ker}(\E_d)$, $\Psi, H, P_j:\R \rightarrow \R$ such that $\Psi'''=\psi$, $H'=h$, $P''_j=p_j$ and for a polynom $Q$ solving
\begin{equation} 
\E_d Q=(e_1\otimes_{\E_d} e_1) \left( x_2^2-x_3^2\right).
\end{equation}
Moreover, if $n\geq 4$, $\varrho=0$.
\end{lemma}
\begin{proof}
We now split the proof in three steps: the first two to compute the $g$ dependently on weather $n=3$ or $n\geq4$, and the last one to conclude. Just restricted to this proof will be useful the notation: for $w:\R^n\rightarrow \R^n$ write $w(\hat{x}_{i_1},\ldots, \hat{x}_{i_k})$ to emphasize that the function $w$ does not depend on the variable $x_{i_1},\ldots x_{i_k}$ (but possibly depend on all the others).\\
\smallskip

\textbf{Step one:} \textit{computing the $g$ for $n\geq 4$}. 
In this case, as already observed in \ref{rmk:ngeq4}, we have (due to \eqref{TWOone}, \eqref{THREEtwo}$) \partial_{jj}f=0$ for all $j\geq 2$. Also, by invoking Lemma \ref{lem:tecnico} (for $\a=1, \beta=0$) we obtain the relevant set of equations
\begin{align}
\partial_{11} f+\partial_{22}f=&-\partial_{22}g \label{1122}\\
 \partial_{21} f=&0\ \ \ &\text{for all $j\geq 3$},  \label{12}\\
 -\partial_{2j}g - \partial_{2j}f=&0\ \ \ &\text{for all $j\geq 3$},  \label{2j2j}\\
\partial_{22}f +\partial_{jj}f=&0 \ \ \ &\text{for all $j\geq 3$},   \label{f22jj}\\
 -\partial_{1j}f=&0 \ \ \ &\text{for all $j\geq 3$},   \label{f1j}\\
\partial_{kj} g=&0 \ \ \ &\text{for all $k,j\geq 3$, $k\neq j$},  \label{gkj}\\
\partial_{kj} f  =&0\ \ \ &\text{for all $k\neq j$ and $j\geq 3$},   \label{fki}
\end{align}
By \eqref{fki} and \eqref{2j2j}:
\begin{equation}\label{set1}
\partial_{2j}f=0  \ \ j\geq 3,  \ \ \partial_{2j}g=0 \ \ j\geq 3.
\end{equation}
The above and \eqref{12} implies
\begin{equation}\label{set2}
\nabla(\partial_2 f)=0, \ \ \nabla(\partial_1 f)=(-\partial_{jj}g)e_1 \ \ j\geq 3, \ \ \nabla(\partial_j f)=0\ \ j\geq 3.
\end{equation}
and \eqref{1122} gives additionally
\[
-\partial_{22}g=\partial_{11}f.
\]
Since \eqref{12}, \eqref{f1j} yields $\partial_k (\partial_{11}f)=0$ for $k\neq 1$ it follows that
\begin{equation}\label{22jj}
\partial_{22}g= \partial_{jj}g=\psi(x_1) \ \ j \geq 3.
\end{equation}
for some $\psi:\R \rightarrow\R$. Thence, for some $p_2$, it must hold
\[
\partial_2 g=p_2(\hat{x}_2)+\psi(x_1) x_2 .
\]
By \eqref{set1} we now must have $p_2(\hat{x}_2)=p_2(x_1)$ and (for some $G$)
\[
g=G(\hat{x}_2)+p_2(x_1)x_2+\frac{x_2^2}{2} \psi(x_1) .
\]
By $\eqref{gkj}$ and  $\eqref{22jj}$ we have $\partial_{jj} G(\hat{x}_2)=\psi(x_1)$ for $j\geq 2$ and $\partial_{kj} G(\hat{x}_2)=0$. Thus, for some $p_j$, 
\[
\partial_j G(\hat{x}_2)=p_j(x_1)+\psi(x_1)x_j, \ j\geq 3, \ \partial_2 G=0.
\]
 In particular
 \[
 G(\hat{x}_2)= h_j(\hat{x_j},\hat{x}_2)+p_j(x_1)x_j+\psi(x_1) \frac{x_j^2}{2}\ \ j\geq 3
 \]
for some $h_j$, which means
 \[
 G(\hat{x}_2)= h(\hat{x}_2)+\frac{1}{n-2}\sum_{j=3}^n p_j(x_1)x_j +\frac{\psi(x_1)}{n-2}\sum_{j=3}^n \frac{x_j^2}{2}
 \]
where $h(\hat{x}_2)=\frac{1}{n-2}\sum_{j=3}^n h_j(\hat{x}_2,\hat{x}_j)$. Since $\partial_j G(\hat{x}_2)=\psi(x_1)x_j+p_j(x_1)$ we immediately have $\partial_j h(\hat{x}_2)=0 $ for all $j\geq 3$ and thus $h(\hat{x}_2)=h(x_1)$. Hence, up to rename $\psi, p_j$ to include the factor $\frac{1}{n-2}$):
 \[
 g=   h(x_1)+\sum_{j=2}^n p_j(x_1)x_j+\psi(x_1)\sum_{j=2}^n \frac{x_j^2}{2}.
 \]
 \smallskip

\textbf{Step two:} \textit{computing the $g$ for $n=3$.} In this case the relevant set of equation from Lemma \ref{lem:tecnico} reduces to
\begin{align}
\partial_{11} f+\partial_{22}f=&-\partial_{22}g \label{11223d}\\
 \partial_{21} f=&0, \label{213d}\\
\partial_{11} f+\partial_{33}f+\partial_{33} g=&0,  \label{11333d}\\
 -\partial_{23}g - \partial_{23}f=&0, \label{23233d}\\
\partial_{22}f +\partial_{33}f=&0, \label{22333d}\\
 -\partial_{13}f=&0,\label{133d}\\  
\partial_{23} f  =&0.\label{233d}
\end{align}
From \eqref{213d} and \eqref{133d} we get
\[
\partial_{211}f=\partial_{311}f=0 \ \ \Rightarrow \ \partial_{11}f=S_1(x_1).
\]
From \eqref{213d} and \eqref{233d} we get 
\[
\partial_{122}f=\partial_{322}f=0 \ \ \Rightarrow \ \partial_{22}f=S_2(x_2).
\]
From \eqref{133d} and \eqref{233d} we get 
\[
\partial_{133}f=\partial_{233}f=0\ \ \Rightarrow \ \partial_{33}f=S_3(x_3).
\]
But because of \eqref{22333d} clearly $S_2(x_2)=\varrho$, $S_3(x_3)=-\varrho$ for some constant $\varrho \in \R$. Thus setting $\psi=S_1$ we have from \eqref{11223d} and \eqref{11333d}:
\begin{align}
-\partial_{22} g=&\psi(x_1)+\varrho,\label{22rho}\\
-\partial_{33}g=&\psi(x_1)-\varrho\label{33rho}.
\end{align}
From \eqref{22rho} and $\partial_{23}g=0$ (obtained by combining \eqref{233d} and \eqref{23233d}) we get
\begin{equation}\label{d2gfinal}
-\partial_2 g= p_2(x_1) +(\psi(x_1)+\varrho) x_2
\end{equation}
and thus
\[
- g=h_2(x_1,x_3)+p_2(x_1)x_2+\left(\psi(x_1)+\varrho\right)\frac{x_2^2}{2}.
\]
From \eqref{33rho} and again $\partial_{23}g=0$ we instead get
\begin{equation}\label{d3gfinal}
-\partial_3 g= p_3(x_1) +(\psi(x_1)-\varrho) x_3
\end{equation}
and
\[
-g= h_3(x_1,x_2)+ p_3(x_1)x_3 + \left(\psi(x_1)-\varrho\right)\frac{x_3^2}{2}.
\]
Thus, we have (setting $h(x_1,x_2,x_3)=h_2(x_1,x_3)+h_3(x_1,x_2)$)
\begin{align*}
    -g=\left(\psi(x_1)+\varrho\right)\frac{x_2^2}{2}+\left(\psi(x_1)-\varrho\right)\frac{x_3^2}{2}+p_2(x_1)x_2+ p_3(x_1)x_3 + h(x_1,x_2,x_3)
\end{align*}
but the relations \eqref{d2gfinal}, \eqref{d3gfinal} also implies $\partial_2 h=\partial_3 h=0$ and thence (up to a renaming of $\psi,p_j,\varrho$ ) 
\begin{align*}
    g=\psi(x_1)\sum_{j=2}^3 \frac{x_j^2}{2}+\varrho(x_2^2-x_3^2) +h(x_1)+\sum_{j=2}^3 p_j(x_1)x_j.
\end{align*}
\textbf{Step three:} \textit{computing the $u$}. It is now a simple computation to check that a special solution to \eqref{eqn:PDEforBEqualZero} with a $g$ as in \eqref{eqn:ShapeOfGBetaEqualZero} is given by
\begin{align*}
    w= (H(x_1)-\Psi(x_1))e_1+e_1\left(\sum_{j=2}^n  P'_j(x_1) x_j+ \frac{x_j^2}{2}\Psi''(x_1)\right) - \sum_{j=2}^n e_j\left(P_j(x_1)+x_j \Psi'(x_1) \right)  + \varrho Q
\end{align*}
with $H'=h, \Psi'''=\psi$ e $P''_j=p_j$, and where $Q$ solves
\[
\E_d Q=(e_1\otimes_{\E_d} e_1) (x_2^2-x_3^2). 
\]
Indeed
\begin{align*}
  \nabla w= &(H'(x_1)-\Psi'(x_1))e_1\otimes e_1+e_1\otimes e_1 \left(\sum_{j=2}^d  P''_j(x_1) x_j+ \frac{x_j^2}{2}\Psi'''(x_1)\right)  \\
  & +e_1\otimes e_j \left(\sum_{j=2}^d  P'_j(x_1)  + x_j \Psi''(x_1)\right)-e_j\otimes e_1 \left(P'_j(x_1)+\sum_{j=2}^n x_j \Psi''(x_1))\right) \\
    &- \sum_{j=2}^n e_j\otimes e_j \Psi'(x_1) + \varrho\nabla Q
\end{align*}
and thus
\begin{align*}
  \E w= &\left((H'(x_1)-\Psi'(x_1))+ \sum_{j=2}^d  P''_j(x_1) x_j+ \frac{x_j^2}{2}\Psi'''(x_1)\right)e_1\otimes e_1  -\sum_{j=2}^n e_j\otimes e_j \Psi'(x_1)   + \varrho\E Q\\
  = &( g(x)-\varrho(x_2^2 - x_3^2)-\Psi'(x_1)) e_1\otimes e_1  -\sum_{j=2}^n e_j\otimes e_j \Psi'(x_1)   + \varrho\E Q\\
    = &( g(x)-\varrho(x_2^2 - x_3^2) ) e_1\otimes e_1  -  \Psi'(x_1)\Id   + \varrho\E Q\\
\end{align*}
Since
\begin{align*}
  \trace(\E w)= &  g(x)- \varrho(x_2^2 - x_3^2)- n\Psi'(x_1) + \varrho\trace(\E Q)
\end{align*}
we conclude
\begin{align*}
    \E_d w=&( g(x)-\varrho(x_2^2 - x_3^2)) e_1\otimes e_1  -\Psi'(x_1)\Id    - \frac{g(x) - \varrho(x_2^2 - x_3^2) - n\Psi'(x_1)}{n} \Id + \varrho\E_d Q\\
    =& (g(x)- \varrho(x_2^2 - x_3^2))e_1\otimes_{\E_d}e_1 + \varrho\E_d Q\\
    =&g(x)e_1\otimes_{\E_d}e_1 + \varrho(x_2^2 - x_3^2)e_1\otimes_{\E_d}e_1=  g(x) e_1\otimes_{\E_d}e_1.
\end{align*}
Thus $u$ must be equal to $w$ up to a Kernel element $L\in \mathrm{Ker}(\E_d)$ achieving the structure as in \eqref{eqn:ShapeRegUForBetaEqualZero} and concluding the proof.
\end{proof}
We can now prove the rigidity structure for parallel vectors.
\begin{proof}[Proof of Theorem \ref{thm:MainRigidityParallel}]
Without loss of generality we can assume that $a=e_1$. Thus by mollifying $u_\e:=u\star \varrho_\e$ and applying \ref{lem:rigForRegUnidir} we get that
\[
g_{\e}=h_{\e}(x_1)+\sum_{j=2}^n p_{\e,j}(x_1) x_j 
+\psi_{\e}(x_1)\sum_{j=2}^n \frac{x_j^2}{2}+\varrho_{\e}(x_2^2-x_3^2)
\]
and that 
\begin{equation}\label{eqn:generalShapePAr}
\begin{split}    
u_\e=&(H_\e(x_1)-\Psi_\e(x_1))e_1 + \left(\sum_{j=2}^n P_{\e,j}'(x_1)x_j + \frac{x_j^2}{2}\Psi_\e''(x_1)\right)e_1\\
&-\sum_{j=2}^n (x_j \Psi_\e'(x_1)+P_{\e,j}(x_1))e_j+\varrho_\e Q(x)+L_\e(x)
\end{split}
\end{equation}
with $P''_{\e,j}=p_{\e,j}$, $\Psi'''_{\e}=\psi_{\e}$, $H'_{\e}=h_{\e}$.\\

All is left to do is to compute the limit. Note that as before we conclude immediately that $L_\e \rightarrow L\in \mathrm{Ker}(\E_d)$ so without loss of generality we can neglect the term in $L_\e$ in computing the limit and assume rightaway that $u_\e\rightarrow u$ in $L^1$. Recall  also that $g_{\e}\L^d \wt \nu$ and thus
\[
\sup_{\e>0}\{\|g_{\e}\|_{L^1}\}<+\infty. \ \ 
\]

Up to redefine the functions $H_\e,\Psi_\e,P_{\e,j}$ and operating a translation we can restrict ourselves to consider just to check the convergence on the compact set $[-1,1]^n$. \\

We now split the proof in two steps, depending on $n$ being $3$ or bigger than $3$.\\

\textbf{Step one:} \textit{the case $n\geq 4$}. In this case we do not have to handle the polynomial term $Q$ since $\varrho_\e =0$. Let $\eta_i\in C^{\infty}_c([-1,1])$ be such that 
\begin{align*}
 \eta_i\geq 0 \ \ \text{for all $i=1,\ldots,n$}\\
 \int_{[-1,1]} \eta_i(t)\d t=1 \ \ \text{for all $i=2,\ldots,n$}.
 \end{align*}
 Set also
 \begin{align*}
r_i:=\int_{[-1,1]}  t\eta_i (t)\d t , \ \ \ s_i:=\int_{[-1,1]}  \frac{t^2}{2}\eta_i (t)\d t \ \ \ \text{for $i=2,\ldots,n$}.
\end{align*}
Let us denote by $\eta_{\hat j}(\hat{x}_j):=\eta_1(x_1)\ldots \eta_{j-1}(x_{j-1})\eta_{j+1}(x_{j+1})\ldots \eta_n(x_n)$. Since for a fixed $j\geq 2$ we have $(\E_d u_\e)_{1j}=0$ we get $-\partial_{j}(u_\e \cdot e_1) =   \partial_1 (u_\e \cdot e_j)$:
\begin{align*}
-\int_{[-1,1]^n} (u_\e \cdot e_1) \eta_j'(x_j) \eta_{\hat j}  (\hat{x}_j)\d x=&\int_{[-1,1]^n} (u_\e \cdot e_j)\eta_1'(x_1)  {\eta}_{\hat 1}  (\hat{x}_1)\d x\\
=&  -\int_{[-1,1]^2} (x_j \Psi_\e'(x_1) + P_{\e,j}(x_1))\eta_1'(x_1) \eta_j(x_j)\d x_1\d x_j\\
=& -r_j\int_{[-1,1]}  \Psi_\e'(x_1) \eta_1'(x_1)\d x_1- \int_{[-1,1]}P_{\e,j}(x_1))\eta_1'(x_1)  \d x_1.
\end{align*}
By selecting an even function $\eta_j$ we get $r_j=0$ and therefore
\begin{align*}
\left|\int_{[-1,1]}P_{\e,j}(x_1))\eta_1'(x_1)  \d x_1\right|&\leq \|\eta_j'\|_{\infty}\int_{[-1,1]^n} |(u_\e \cdot e_1) | \d x  .
\end{align*}
Since $u_\e\cdot e_1\rightarrow u\cdot e_1$ it follows that  
\begin{equation}\label{bndPprime}
\sup_{\e}\left\{\|P'_{\e,j}\|_{L^1}\right\}\leq C .
\end{equation}
By integrating \eqref{eqn:generalShapePAr} against $e_j \varphi \eta_2\ldots \eta_n$ for $j\geq 3$, $\varphi\in L^{\infty}([-1,1])$ and even $\eta_i\in C^1$, $i\geq 2$ we get
\begin{align*}
    \int_{Q} (u_{\e}(x)\cdot e_j ) \varphi(x_1)\eta_{\hat 1}(\hat{x}_1)\d x&=-\int_{[-1,1]} P_{\e,j}(x_1)\varphi(x_1)\d x,
\end{align*}
and thus, for all $\e>0$, $j=2,\ldots,n$, by duality
\[
\|P_{\e,j}\|_{L^1}=\sup_{\|\varphi\|_{L^\infty} \leq 1}\left\{\left| \int_{[-1,1]} P_{\e,j}(x_1) \varphi(x_1)\d x_1\right|\right\}\leq C \|u_{\e}\|_{L^1}\leq C.
\]
Thence, the above combined with \eqref{bndPprime}, we have $P_{\e,j}\rightarrow P_j$ (up to a subsequence) in $L^1$ and with $P_j\in \BV([-1,1])$. By observing that $u_\e\cdot e_j\rightarrow u\cdot e_j$ in $L^1$ we deduce also that $x_j \Psi'_\e(x_1) \rightarrow (u\cdot e_j)-P_j$ in $L^1$. By integrating over a rectangle for which $x_j\neq 0$ we get $\Psi_\e'\rightarrow G$ for some $G\in L^1$ and thence
\[ 
u_{\e}\cdot e_j \rightarrow x_j G(x_1) + P_j(x_1)=u\cdot e_j
\] 
and 
\begin{equation}\label{bndpsiprime}
\sup_{\e}\left\{\|\Psi'_\e\|_{L^1}\right\}<+\infty.
\end{equation}

Fix $0<s\leq 1$. We now choose $\eta_j=\eta_2$ for all $j\geq 2$ with $\eta_2$ even so that $s_j=s_2=s$ and $r_j=0$ for all $j\geq 2$ and this yields
\begin{align*}
\int_{Q} g_{\e} \eta_1(x_1)\ldots \eta_n(x_n)\d x=&-  \left( s(n-1)    \int_{[-1,1]} \Psi_\e''(x_1)\eta'_1(x_1) \d x_1\right.\\
&\left.+ \int_{[-1,1]} H_\e(x_1)\eta'_1(x_1)\d x_1 \right)
\end{align*}
 
Note that this means that, for all $0<s\leq 1$, there is $C=C(s)$ such that
\[
\sup_\e \left\{\|s(n-1)\Psi'''_\e+H'_\e\|_{L^1}\right\}<C(s).
\]
By selecting two different values for $s$ we can decouple the uniform bound as 
\begin{equation}\label{bndHPsiThird}
\sup_\e \left\{\|H'_\e\|_{L^1}\right\}<+\infty, \ \  
\sup_\e \left\{\| \Psi'''_\e \|_{L^1} \right\}<+\infty.
\end{equation}
Fix now any $j\geq 2$. By now selecting $\eta_i'$s for $i\geq 2$ to be all equal and even except for  $i=j$ (so that $s_i=s$, $r_i=0$ for $i\neq j$, $i\geq 2$). Then
\begin{align*} 
\int_{Q} g_{\e} \eta_1(x_1)\ldots \eta_n(x_n)\d x=&  s(n-2)    \int_{[-1,1]} \Psi_\e'''(x_1)\eta_1(x_1) \d x_1+ s_j  \int_{[-1,1]} \Psi_\e'''(x_1)\eta_1(x_1) \d x_1\\
& +r_j \int_{[-1,1]} P_{\e,j}''(x_1)\eta_1(x_1) \d x_1 + \int_{[-1,1]} H'_\e(x_1)\eta_1(x_1)\d x_1 
\end{align*} 
and thus
\[ 
\| P''_{\e,j}\|_{L^1}\leq C(\|\Psi'''_{\e}\|_{L^1}+\|H'_{\e}\|_{L^1}+\|g_{\e}\|_{L^1})
\]
yielding
\begin{equation}\label{bndPsecond}
    \sup_{\e}\{\| P''_{\e,j}\|_{L^1}\}<+\infty  \  \ \ \text{for all $j\geq 2$}.
\end{equation}
 Analogously considering $\varphi\in L^{\infty}$ and even $\eta_j=\eta_2$ for all $j\geq 2$ we get
 \begin{align*}
     \int_{Q} (u_{\e} \cdot e_1)\varphi(x_1)\eta_2(x_2)\ldots\eta_n(x_n)=&\int_{[-1,1]} (H_{\e}(x_1)-\Psi_{\e}(x_1))\varphi(x_1)\d x_1 \\
    & + s(n-1) \int_{[-1,1]}\Psi''_{\e}(x_1) \varphi(x_1)\d x_1
 \end{align*}
meaning
\[
\sup_{\e>0}\{ \|H_{\e}-\Psi_{\e}+s(n-1)\Psi''_{\e}\|_{L^1}\}\leq C(s).
\]
By selecting two different values for $s$ we can decouple the uniform bound as  
\begin{equation}\label{bndHminusPsiPsiSecondo}
\sup_{\e>0}\{ \|H_{\e}-\Psi_{\e}\|_{L^1}\}<+\infty, \ \ \  \sup_{\e>0}\{\|\Psi''_{\e}\|_{L^1}\}<+\infty.
\end{equation}
By combining \eqref{bndHminusPsiPsiSecondo}, \eqref{bndHPsiThird} and \eqref{bndpsiprime} we get that
\[
H_\e- \Psi_{\e} \rightarrow F, \ \ \Psi''_\e \rightarrow G' \ \ \text{in $L^1$}
\]  
for some $F \in \BV([-1,1])$, $G\in W^{1,1}([-1,1])$, $G'\in \BV([-1,1])$. By combining \eqref{bndPprime} and \eqref{bndPsecond} we get also that
\[
P'_{\e,j}\rightarrow P_j' \ \ \text{in $L^1$}
\]
with $P_j'\in \BV([-1,1])$ and the final limit $u$ has the claimed structure.\\

 \smallskip
 \textbf{Step two:} \textit{the case $n=3$}. The main issue here is in deducing the convergence of $\varrho_{\e}$ and then everything follows exactly as in Step one. By choosing $\eta_2=\eta_3$ even, integrating $g_{\e} \eta_{1}\eta_2\eta_3$ over $Q$ we can obtain again \eqref{bndHPsiThird} since the polynomial part cancels out. By now considering $\varphi\in L^{\infty}$ and only $\eta_2$ to be even (and still integrating $g_{\e} \eta_{1}\eta_2\eta_3$ over $Q$) we conclude that
 \[
\left|\int_{[-1,1]} (P''_{\e,3}(x_1)+ \varrho_{\e} (s_2-s_3))\varphi(x_1)\d x_1 \right|\leq C(s_2,s_3)( \|H'_{\e}\|_{L^1}+ \|\Psi'''_{\e}\|_{L^1}+\|g\|_{L^1}).
 \]
By selecting specific $s_2,s_3$ we can again decouple the bound as
\[
\sup_{\e}\{\|P''_{\e,3}\|_{L^1}\}\leq C, \ \ \sup_{\e}\{\varrho_{\e}\}\leq C.
\]
In particular This implies that $\varrho_{\e}\rightarrow \varrho$ and thus $\varrho_{\e}Q \rightarrow \varrho Q$. Now, by considering $w_{\e}:= u_{\e}-\varrho_{\e} Q$ and by applying the exact arguments of Step one to $w_{\e}$ we conclude that $u - \varrho Q$ has the claimed structure. The proof is complete.
 \end{proof} 

\begin{proof}[Proof of Theorem \ref{MainTheoremINTRO}]
    It comes as a consequence of Proposition \ref{prop:Annihilator}, Theorem \ref{thm:MainRigidityNonParallel} and Theorem \ref{thm:MainRigidityParallel}.
\end{proof}
\section{Proof of Kernel projection Theorem \ref{MainTheoremINTRO2}}\label{sct:KernelProjection}
We now provide a proof of Theorem \ref{MainTheoremINTRO2}, yielding an explicit map $\mathcal{R}_K$ which is particularly useful in homogenization and integral representation problems. \\

For any given convex set $K$ we recall the intrinsic quantity defined in Theorem \ref{MainTheoremINTRO2}:
	\[
	\tau_K:=\frac{1}{P(K)}\int_{\partial K} \left(\frac{\left|y-\mathrm{bar}(K)\right|^2}{2}\Id- (y-\mathrm{bar}(K))\otimes (y-\mathrm{bar}(K))\right)\d\H^{n-1}(y).
	\]
	where $P(K)$ is the distributional perimeter of $K$ (see \cite{Maggi}), 
	\[
	\mathrm{bar}(K):=\frac{1}{|K|}\int_K y\d y.
	\]
and we denote by $\nu_K$ the outern unit normal (defined $\H^{n-1}$-a.e. on $\partial K$). For such sets the distributional tangential derivative of $\nu_K$ is a matrix-valued Radon measure \cite[Theorem 2, Section 6.3]{evans2018measure}. We will denote it as $D^{\mathrm{Tan}}\nu_K\in \mathcal{M}(\R^n;\mathbb{M}^{n\times n})$, we recall that it is supported on $\partial K$ and
\begin{equation} \label{eqn:MeanCurv2}
	\int_{\partial K} F(y)^t \d D^{\mathrm{Tan}}\nu_K (y):=-\int_{\partial K}\nu_K(y)^t (\nabla^{\mathrm{Tan}} F)(y)\\d\H^{n-1}(y)\ \ \ \text{for all $F\in C^{\infty}(\R^n;\R^n)$}
	\end{equation}
with $z^t$ denoting the transpose of $of z\in \R^n$ and where we have defined
	\[
	\nabla^{\mathrm{Tan}} F(y):= \nabla F(y) - \nabla F(y) \nu_K(y) \otimes \nu_K(y).
	\]
	\begin{remark}\label{rmk:CurvBalls}
If $K$ has $C^2$ boundary the $D^{\mathrm{tan}}\nu_K$ can be explicitize in terms of $K$ as follows: let $H_K(y)$ be the scalar mean curvature of $\partial K$ at $y$ and define the vector $\mathbf{H}_K(y):=(n-1)H_K(y)\nu_K(y)$. Then we recall that (see \cite[Theorem 11.8]{Maggi}) for such regular sets we have the Gauss-Green formula on surfaces  
\[
\int_{\partial K}\nabla^{\mathrm{Tan}}\left( v\cdot \nu_K\right)d\H^{n-1}=\int_{\partial K}\HH_K v\cdot \nu_Kd\H^{n-1}
\]
in force for all $v\in C^{\infty}(\R^n;\R^n)$.
In particular since
\[
\nabla^{\mathrm{Tan}}\left( v\cdot \nu_K\right)= v^t \nabla^{\mathrm{Tan}}\nu_K+\nu_K^t\nabla^{\mathrm{Tan}}v
\]
we have
	\begin{equation}\label{eqn:MeanCurv}
	\int_{\partial K} \nu_K^t\left(\nabla^{\mathrm{Tan}}v\right)
	\d\H^{n-1}=\int_{\partial K} \HH_K v\cdot\nu_K-  v^t(\nabla^{\mathrm{Tan}}\nu_K)\d\H^{n-1}.
	\end{equation} 
Thus
\[
D^{\mathrm{Tan}} \nu_K=\left(\nabla^{\mathrm{Tan}}\nu_K-\mathbf{H}_K \otimes  \nu_K\right)\H^{n-1}_{\restr \partial K}
\]
 According to our convention, the mean curvature $H_K$ is always positive for convex sets $K$.  As a title of example we recall that, for $K=B_{r}$ we know that $H_{B_r}=\frac{1}{r}$. 
	\end{remark} 
\begin{remark}
 If $K$ has only piece-wise $C^2$ boundary then $D^{\mathrm{Tan}}\nu_K$ might contains parts orthogonal to  $\H^{n-1}_{\restr \partial K}$. Consider, as a title of example, $K=[-1,1]^2$ in $\R^2$. Then, still as a consequence of the Gauss-Green formula on surfaces (with the boundary terms accounted with the conormal as in \cite[Theorem 11.8]{Maggi}):
 \[
 D^{\mathrm{Tan}}\nu_K=\left[-(e_1+e_2)\H^0_{\restr(1,1)}-(e_1-e_2)\H^0_{\restr (1,-1)} +(e_1+e_2)\H^0_{\restr (-1,-1)}-(-e_1+e_2)\H^0_{\restr (-1,1)}\right]\otimes \nu_K
 \]
 which is nothing but the mean curvature measure (supported only on the vertexes of $Q$) tensorized with $\nu_K$.
\end{remark}
Given these intrinsic quantities, now for $u\in \BDD(\Omega)$, $x\in \R^n$ and for $K$ a convex set we recall the Definition of the relevant quantities
    \begin{align}
     \mathrm{s}_{K}[u]&:=-\frac{1}{(n-1)|K|}\int_{\partial K} u^t \d  D ^{\mathrm{Tan}} \nu_K  (y)\label{eqn:s}\\
    \mathrm{A}_{K}[u]&:=\frac{1}{2|K|}\int_{\partial K}(u\otimes \nu_K-\nu_K\otimes u)\d\H^{n-1}(x)\label{eqn:A}\\
    \gamma_{K}[u]&:=\frac{1}{n|K|}\int_{\partial K}(u\cdot \nu_K)\d\H^{n-1}(y)   \label{eqn:T}\\
      \mathrm{b}_{K}[u]&:=\frac{1}{P(K)}\int_{\partial K}u\d\H^{n-1}(y)+\tau_K \mathrm{s}_K[u],\label{eqn:d}
    \end{align}
and we set
    \begin{equation}\label{eqn:linopeker}
    \begin{split}
    \mathcal{R}_{K}[u](y):=&\left(\mathrm{A}_{K}[u]+ \gamma_{K}[u] \Id\right) (y-\mathrm{bar}(K))\\ &+(\mathrm{s}_{K}[u] \cdot (y-\mathrm{bar}(K)))(y-\mathrm{bar}(K)) -\mathrm{s}_{K}[u]\frac{|y-\mathrm{bar}(K)|^2}{2}+ \mathrm{b}_{K}[u]
    \end{split}
   \end{equation}
This correction with $\mathrm{bar}(K)$ makes the quantity $\mathcal{R}_K$ depending only on the shape and the size of $K$ and not on the position. We now prove the Theorem \ref{MainTheoremINTRO2}.
\begin{proof}[Proof of Theorem \ref{MainTheoremINTRO2}]
The Poincarè inequality \eqref{eqn:PoincaréInequality} comes from \ref{prop:poincare} as soon as we show that $\mathcal{R}_K$ is linear, bounded and fixes $\mathrm{Ker}(\E_d)$ (see \cite{CVG2025,breit2017traces}). Thus we focus on the first part of the statement.\\

Clearly $\mathcal{R}_{K}$ is linear. Since the convergence in $\BDD$ implies the convergence of the traces, and since the elements identifying $\mathcal{R}_{K}$ are defined on the traces, the continuity (and thus the boundedness) follows at once. We focus on showing that $\mathcal{R}_{K}$  fixes the elements of $\mathrm{\E_d}$. For $L\in \mathrm{\E_d}$ we have
    \[
    L(y)=(A+\gamma \Id)y+(s\cdot y)y-s\frac{|y|^2}{2}+  b
    \]
for some $R\in  \mathbb{M}^{n\times n}_{skew},\gamma\in \R$,$s$, $b\in \R^n$. Let us split
    \[
    L(y):=M(y)+B(y), \ \ \  M(y):=(A+\gamma\Id)y+ b, \  \ B(y):= (s \cdot y)y-s\frac{|y|^2}{2}.  
    \]
It is straightforward to see that, for $u\in BV(\Omega)$ it holds
    \begin{align}
    \Aa_{K}[u]&=\frac{1}{2|K|}(Du(K)-Du^t(K)),\label{eqn:Aa}\\
    \gamma_{K}[u]&=\frac{1}{n|K|} \trace(Du(K))\label{eqn:gamma}
    \end{align}
   implying
   	\[
   	\Aa_K[M]=A, \ \gamma_K[M]= \gamma.
   	\]
Due to the very definition of $\ss_K$ we also have, for regular maps $v\in C^1(\R^n;\R^n)$ that
	\[
	\ss_K[v]=-\frac{1}{(n-1)|K|}\int_{\partial K} \nu_K(y)^t \nabla^{\mathrm{Tan}}v(y)\d \H^{n-1}(y).
	\]
In particular we have 	
\[
\nu_K^t\nabla^{\mathrm{Tan}}M(y)=(\nu_K^tA+\gamma\nu_K)-(\nu_K^tA\nu_K)   \nu_K(y)-\gamma\nu_K(y)=\nu_K^tA-(\nu_K^tA\nu_K)   \nu_K(y).
\] 
Since $K$ is center-symmetric, this facts immediately implies
	\[
	\ss_K[M]=0, \ \ \ \bb_K[M]=b.
	\]
Moreover, since 
	\[
	DB=(s\cdot y)\Id +y\otimes s-s \otimes  y,
	\]
then, again by center-symmetry, $DB(K)=0$, and thence we have (recalling \eqref{eqn:Aa}, \eqref{eqn:gamma}):
	\begin{align*}
	\Aa_K[B]=0, \ \ \ \gamma_K[B]=0.
	\end{align*}
Furthermore (setting to alleviate the notation $\nu=\nu_K$)
	\begin{align*}
 \nabla^{\mathrm{Tan}}B=(s\cdot y)\Id +y\otimes s  -s \otimes  y- \left[(s\cdot y)(\nu\otimes \nu) + (s\cdot \nu) y\otimes \nu -  (y\cdot \nu)s\otimes \nu \right]
	\end{align*}
and thus
	\begin{align*}
 \nu^t\nabla^{\mathrm{Tan}}B&=(s\cdot y)\nu +(\nu\cdot y)  s -(\nu\cdot s)  y- \left[(s\cdot y)   \nu+ (s\cdot \nu) (\nu\cdot y) \nu -  (y\cdot \nu)(\nu\cdot s)  \nu \right]\\
 &= (\nu\cdot y)  s -(\nu\cdot s)  y.
	\end{align*}
Therefore
	\begin{align*}
	\ss_K[B]= \frac{1}{(n-1)|K|} \int_{\partial K} \nu^t\nabla^{\mathrm{Tan}}B \d\H^{n-1}(y)=\frac{1}{(n-1)|K|}\int_{\partial K}  ((\nu\cdot y) s -(\nu\cdot s)  y)\d \H^{n-1}(y).
	\end{align*}
Since, clearly
		\begin{align*}
	 \int_{\partial K}s(y\cdot \nu)\d y&= \sum_{i=1}^n e_i \int_{K} \dive(s_i y)\d y=n|K|s\\
	 	 \int_{\partial K}y(s\cdot \nu)\d y&= \sum_{i=1}^n e_i \int_{K} \dive(y_i s)\d y= s |K| 
		\end{align*}
	then 
    \[
    \ss_K[B]=\frac{1}{(n-1)|K|}\left(n|K|s - s|K|\right)s =s.
    \]
Finally, this also shows that
		\begin{align*}
		\bb_K[B]=-\tau_K  s+\tau_K  \ss_K[B]=0.
\end{align*}		
Summarizing
	\begin{align*}
\ss_K[L]&=\ss_K[B]+\ss_K[M]=s,\\
	\Aa_K[L]&=\Aa_K[B]+\Aa_K[M]=A,\\
	\gamma_K[L]&=\gamma_K[B]+\gamma_K[M]=\gamma,\\
	\bb_K[L]&=\bb_K[B]+\bb_K[M]=b.
	\end{align*}
yielding $\mathcal{R}_K[L]=L$.  
\end{proof}

Before concluding the Section we state a Lemma which turns out to be quite useful in the blow-up theory. We recall the notion of \textit{pushforward measure} (see also \cite{bogachev2007measure} for additional details). Let $\mu\in \mathcal{M}(A_1;\mathbb{M}^{n\times n})$ and $f:A_1\rightarrow A_2$ be a $\mu$-measurable function. The pushforward measure $f_{\#}\mu\in  \mathcal{M}(A_2;\mathbb{M}^{n\times n}) $  is defined as
\[
f_{\#}\mu(B):=\mu(f^{-1}(B)) \ \ \text{for all $\mu$-measurable $B\subset A_2$}.
\]
We recall that the following integral formula holds
\begin{equation}\label{eqn:pushforwardIntegralFormula}
\int_{B} g(x) \d (f_{\#} \mu)(x)=\int_{f^{-1}(B)} g(f(x)) \d \mu(x).
\end{equation}

\begin{lemma}\label{lem:ChangeOfScale}
Let $K$ be a center symmetric convex body. If $u\in \BDD(K_{\varrho}(x))$ then  $v_{\varrho}(y):= \frac{u(x+\varrho y)}{\varrho}\in \BDD(K)$ and
	\begin{align*}
	|\E_d v_{\varrho}|(K)&=\varrho^{-n}|\E_d u|(K_{\varrho}(x)),\\
	\mathcal{R}_{K} [v_{\varrho}](y)&=\frac{\mathcal{R}_{K_{\varrho}(x)} [u](x+\varrho y)}{\varrho}
	\end{align*}
\end{lemma}
\begin{proof}
We can suppose without loss of generality that $\mathrm{bar}(K)=0$. If $u\in C^{\infty}(K_{\varrho}(x))$ we have
	\[
	\E_d v_{\varrho}(y)=\E_d \left( \frac{u(x+\varrho y)}{\varrho}\right)= (\E_d u)(x+\varrho y)
	\]
and thence
	\begin{align*}
	|\E_d v_{\varrho}|(K)&= \int_{K}| (\E_d u)(x+\varrho y)|\d y\\
	&=\varrho^{-n} \int_{K_{\varrho}(x)}|\E_d u(z)|\d z=\varrho^{-n}| \E_d u|(K_{\varrho}(x)).
	\end{align*}
By approximation in the strict convergence (see \cite[Theorem 2.8]{breit2017traces}) we now pass to the whole $\BDD$ . We now check the scaling properties of $\mathcal{R}_K$. We consider
	\begin{align}
	\mathrm{A}_K [v_{\varrho}]&= \frac{1}{2\varrho |K|}\int_{\partial K} [u(x+\varrho y) \otimes \nu_K - \nu_K \otimes u(x+\varrho y)]\d \H^{n-1}(y)\nonumber\\
	&= \frac{1}{2 |K_{\varrho}(x)|}\int_{\partial K_{\varrho}(x)} \left[u(z) \otimes \nu_K\left(\frac{z-x}{\varrho}\right) - \nu_K\left(\frac{z-x}{\varrho}\right)  \otimes u(z)\right]\d \H^{n-1}(y)\nonumber\\
	&= \frac{1}{2 |K_{\varrho}(x)|}\int_{\partial K_{\varrho}(x)} \left[u(z) \otimes \nu_{K_{\varrho}(x)}(z) - \nu_{K_{\varrho}(x)}(z)  \otimes u(z)\right]\d \H^{n-1}(y)\nonumber\\
	&= \mathrm{A}_{K_{\varrho}(x)}[u]\label{a}
	\end{align}
since
	\[
	\nu_{K_{\varrho}(x)}(z)=\nu_K\left(\frac{z-x}{\varrho}\right) \ \ \ \text{for all $z\in \partial K_{\varrho}(x)$}.
	\]
The exact same computation also yields
	\begin{equation}\label{t}
	\gamma_{K} [v_{\varrho}]= \gamma_{K_{\varrho}(x)}[u].
	\end{equation}
Moreover, setting $g^{x,\varrho}(y):=x+\varrho y$, $g^{x,\varrho}(\partial K)= \partial K_{\varrho}(x)$ we have from \eqref{eqn:pushforwardIntegralFormula}
	\begin{equation}\label{eqn:pushonF}
	\int_{\partial K_{\varrho}(x)} f(z) \d (g_{\#} (D^{\mathrm{Tan}} \nu_{K})(z)=\int_{\partial K } f(g(y)) \d (D^{\mathrm{Tan}} \nu_{K})(y)
	\end{equation}

Therefore
\begin{align}
	{\rm s}_{K}[v_{\varrho}]=&-\frac{1}{(n-1)\varrho |K| }\int_{\partial K} v(x+\varrho y)^t \d (D^{\mathrm{Tan}}\nu_{K})(y)\nonumber\\
	=&-\frac{1}{(n-1)\varrho |K| }\int_{\partial K_{\varrho}(x)} v( z)^t \d (g^{x,\varrho}_{\#} D^{\mathrm{Tan}}\nu_{K})(z)\label{first}
	\end{align}
Since
	\[
\nabla^{\mathrm{Tan}}\left(F(x+\varrho y) \right))=\varrho \left(\nabla^{\mathrm{Tan}}F\right)(x+\varrho y)
	\]
we have
\begin{align*}
	\int_{\partial K} \nu_K^t(y)\left(\nabla^{\mathrm{Tan}}F\right) & \left(x+\varrho y\right)\d \H^{n-1}(y)=	\frac{1}{\varrho} \int_{\partial K} \nu_K^t(y) \nabla^{\mathrm{Tan}}\left(F \left(x+\varrho y\right)\right) \d \H^{n-1}(y)\\
	&=-\frac{1}{\varrho} \int_{\partial K} F\left(x+\varrho y\right)^t \d (D^{\mathrm{Tan}} \nu_{K})(y)=-\frac{1}{\varrho} \int_{\partial K_{\varrho}(x)} F\left(y\right)^t \d (g^{x,\varrho}_{\#} D^{\mathrm{Tan}}\nu_{K})(y).
	\end{align*}
where the last equality follows from \eqref{eqn:pushonF}. Also
	\begin{align*}
	\int_{\partial K} \nu_K(y)^t \left(\nabla^{\mathrm{Tan}}F\right) \left(x+\varrho y\right)\d \H^{n-1}(y)&=\frac{1}{\varrho^{n-1}}\int_{\partial K_{\varrho}(x)} \nu_K\left(\frac{z-x}{\varrho}\right)^t\left(\nabla^{\mathrm{Tan}}F\right) \left(z\right) \d \H^{n-1}(z)	\\
	&=-\frac{1}{\varrho^{n-1}}\int_{\partial K_{\varrho}(x)} F \left(z\right)^t\d( D^{\mathrm{Tan}} \nu_{K_{\varrho}(x)}) (z).
	\end{align*}
in particular we deduce that
\[
 \int_{\partial K_{\varrho}(x)} F \left(z\right)^t\d( D^{\mathrm{Tan}} \nu_{K_{\varrho}(x)}) (z) =  \varrho^{n-2} \int_{\partial K_{\varrho}(x)} F\left(y\right)^t \d (g^{x,\varrho}_{\#} D^{\mathrm{Tan}}\nu_{K})(y) \ \ \ \text{for all $F\in C^{\infty}(\R^n;\R^n)$}.
\]
Thence
$g^{x,\varrho}_{\#} D^{\mathrm{Tan}}\nu_{K}=\varrho^{2-n}D ^{\mathrm{Tan}} \nu_{K_{\varrho}(x)}$. Thus, from \eqref{first}:
	\begin{align}
	\ss_K[v_{\varrho}]&= -\frac{1}{(n-1)\varrho |K| }\int_{\partial K_{\varrho}(x)} v( z)^t \d (g^{x,\varrho}_{\#} D^{\mathrm{Tan}}\nu_{K})(z)\nonumber\\
    &=-\frac{\varrho}{(n-1)|K_{\varrho}(x)|}\int_{\partial K_{\varrho}(x)} v(z)^t\d(D^{\mathrm{Tan}} \nu_{K_{\varrho}(x)})(z)=\varrho \ss_{K_{\varrho}(x)}[v] \label{s}.
	\end{align}
Finally we note that
	\begin{align*}
	\tau_{K_{\varrho}(x)}&=\frac{\varrho^2}{P(K)}\int_{\partial K} \left(\frac{|y|^2}{2}\Id- y\otimes y\right)\d\H^{n-1}(y)=\varrho^2\tau_K,
	\end{align*}
and
\[
	\frac{1}{P(K)}\int_{\partial K} \frac{ u(x+\varrho y) }{\varrho}\d\H^{n-1}(y)=\frac{1}{\varrho P(K_{\varrho}(x))}\int_{\partial K_{\varrho}(x)}  u(z) \d\H^{n-1}(z).
	\]
Thus
	\begin{align}
	\bb_K[v_{\varrho}]&=\frac{1}{\varrho P(K_{\varrho}(x))}\int_{\partial K_{\varrho}(x)}  u(z) \d\H^{n-1}(z)+\frac{\tau_{K_{\varrho}(x)}}{\varrho^2} \ss_K[v_{\varrho}]\nonumber\\
	&=\frac{1}{\varrho P(K_{\varrho}(x))}\int_{\partial K_{\varrho}(x)}  u(z) \d\H^{n-1}(z)+\frac{\tau_{K_{\varrho}(x)}}{\varrho} \ss_{K_{\varrho}(x)}[u]\nonumber\\
	&=\frac{1}{\varrho} \bb_{K_{\varrho}(x)}[u].\label{d}
	\end{align}
In particular, by collecting \eqref{a},\eqref{t},\eqref{s}, \eqref{d} and the very definition of $\mathcal{R}_K,\mathcal{R}_{K_{\varrho}(x)}$ we conclude
	\begin{align*}
	\mathcal{R}_K[v_{\varrho}](y)= \frac{\mathcal{R}_{K_{\varrho}(x)}[u](x+\varrho y)}{\varrho}.
	\end{align*}
\end{proof}

\section{Vanishing properties of the projection}\label{sct:Vanishing}
The final Proposition of this Section states a key property, usually required in the application, for the map $\mathcal{R}_K$. \begin{proposition}\label{propo:HigerOrderDies}
Let $K$ be a center-symmetric convex body and let $u\in \BDD(\Omega)$. Then, for any $x\notin \Theta_u$ (as in Definition \eqref{eqn:ThetaU}) it holds
	\begin{align}\label{eqn:decaysOnK}
	\lim_{\varrho\rightarrow 0} \varrho^2 | \ss_{K_{\varrho}(x)}[u]|&=\lim_{\varrho\rightarrow 0} \varrho |\Aa_{K_{\varrho}(x)}[u]|=\lim_{\varrho\rightarrow 0} \varrho |\gamma_{K_{\varrho}(x)}[u]|=0.
	\end{align}
and
	\[
	\lim_{\varrho\rightarrow 0} \fint_{K_{\varrho}(x)}|u(y)-\bb_{K_{\varrho}(x)}[u]|\d y =0.
	\]
\end{proposition}
In particular the above slightly extend a result for quasi-continuous points for  $\BDD$ maps, proved firstly in \cite{arroyo2019fine} in the general context of elliptic operators. We recall that a map $u$ is said to be  \textit{approximately quasi-continuous} at $x_0\in \Omega$ if
	\[
	\lim_{\e\rightarrow 0} \min_{b\in \R^n}\left\{ \fint_{K_{\varrho}(x_0)} |u(y)-b|\d y\right\}=0.
	\]
With the following Proposition we can ensure that $\BDD$ maps are approximately quasi-continuous at $\H^{n-1}$-a.e. $x\in \Omega$.\\

Let us underline that this notion of continuity is weaker than the notion of approximate continuity, for which fewer things are known. To obtain this fact we will exploit the specific case $K=B$, in particular the following Proposition.
\begin{proposition}\label{propo:HigerOrderDiesOnBall}
Let $\varrho>0$, $u\in \BDD(\Omega)$. Then for all $x\notin \Theta_u$  (as in Definition \eqref{eqn:ThetaU}) it holds
	\begin{align}\label{eqn:decaysOnBall}
	\lim_{\varrho\rightarrow 0} \varrho^2 |\ss_{B_{\varrho}(x)}[u]|&=\lim_{\varrho\rightarrow 0} \varrho |\Aa_{B_{\varrho}(x)}[u]|=\lim_{\varrho\rightarrow 0} \varrho |\gamma_{B_{\varrho}(x)}[u]|=0
	\end{align}
	and
	\[
	\lim_{\varrho\rightarrow 0} \fint_{B_{\varrho}(x)}|u(y)-\bb_{B_{\varrho}(x)}[u]|\d y =0.
	\]
In particular $u$ is approximately quasi-continuous at $\H^{n-1}$-a.e. point $x\in \Omega$.
\end{proposition}
The proof of the above Proposition is based on an useful characterization of $\Aa_{B_{\varrho}(x)}$, $\gamma_{B_{\varrho}(x)}$,  $\ss_{B_{\varrho}(x)}$ developed in the spirit of \cite{ambrosio1997fine}, \cite{kohn1980new}. 
The proof is quite technical and the techniques are well-known therefore we postpone it to the Appendix. We instead report here the proof of Proposition \ref{propo:HigerOrderDies}.
\begin{proof}
Let $x_0\notin \Theta_u$. The strategy is to start from $\mathcal{R}_{B_{\varrho}(x_0)}$ and replace each quantity with the one under analysis. For example, to deduce the correct decay rate on $\ss_{K_{\varrho}(x_0)}$ we consider the linear application
	\begin{align*}
	\bar{\mathcal{R}}_{\varrho,x_0}[u]:=&(\ss_{K_{\varrho}(x_0)}[u]\cdot (y-x_0))(y-x_0)-\frac{|y-x_0|^2}{2}\ss_{K_{\varrho}(x_0)}[u]\\
	&+(\Aa_{B_{\varrho}(x_0)}[u] +\gamma_{B_{\varrho}(x_0)}[u]\Id)(y-x_0)+\bb_{B_{\varrho}(x_0)}[u]
	\end{align*}
which is basically $\mathcal{R}_{B_{\varrho}(x_0)}$ where we replaced $s_{B_{\varrho}(x_0)}[u]$ with $s_{K_{\varrho}(x_0)}[u]$. The same computation as exploited in the proof of Lemma \ref{lem:ChangeOfScale} tells that
	\[
	\bar{\mathcal{R}}_{1,0}[u(x_0+\varrho \cdot)](y)=\bar{\mathcal{R}}_{\varrho,x_0}[u](x_0+\varrho y), 
	\]
while the same arguments in the Proof of Theorem \ref{MainTheoremINTRO2} in Section \ref{sct:KernelProjection} yields $\bar{\mathcal{R}}_{\varrho,x_0}[L]=L$ for all $L\in \mathfrak{K}$. In particular Proposition \ref{prop:poincare} combined with these two facts, as in Proposition \ref{prop:poincare}, leads us to say that
	\[
	\| u - \bar{\mathcal{R}}_{\varrho}[u]\|_{L^1(B_{\varrho}(x_0))} \leq c \varrho |\E_d|  (B_{\varrho}(x_0)) \ \ \ \text{for all $u\in \BDD(\Omega)$}
	\]
for a constant $c>0$ depending on  $\bar{\mathcal{R}}_{1,0}$ only. Henceforth
	\begin{align*}
	\| u - \bar{\mathcal{R}}_{\varrho,x_0}[u]\|_{L^1(B_{\varrho}(x_0))}\geq & \int_{B_{\varrho}(x_0)} \left|(s_{K_{\varrho}(x_0)}[u]\cdot (y-x_0)) (y-x_0) - \frac{|y-x_0|^2}{2}s_{K_{\varrho}(x_0)}[u]\right|\d y\\
	&- \int_{B_{\varrho}(x_0)}\varrho \left|\Aa_{B_{\varrho}(x_0)}[u] +\gamma_{B_{\varrho}(x_0)}[u]\Id   \right|\d y\\
		&- \int_{B_{\varrho}(x_0)} \left|u(y)- \bb_{B_{\varrho}(x_0)}[u] \right|\d y.
	\end{align*}
Moreover, for $v\in \R^n$ it holds
	\begin{align*}
	\int_{B_{\varrho}(x_0)}& \left|(v\cdot (y-x_0)) (y-x_0) - \frac{|y-x_0|^2}{2}v\right|\d y=\varrho^{n}\int_{B} \varrho^2\left|(v\cdot z) z - \frac{|z|^2}{2}v\right|\d z\\
	&\geq \varrho^{n+2}\left| \int_{B}  \left[(v\cdot z) z - \frac{|z|^2}{2}v\right]\d z\right| = \varrho^{n+2}\left| \int_{0}^1 t^{n+1} \int_{\partial B}  \left[(v\cdot z) z - \frac{v}{2}\right]\d \H^{n-1}(z)\right|\\
			&=  \varrho^{n+2}C(n) |v|.
	\end{align*}
Henceforth
	\begin{align*}
	 \varrho^2  |s_{K_{\varrho}(x_0)}[u]|\leq& c \frac{|\E_d u|(B_{\varrho}(x_0)) }{\varrho^{n-1}}\\
	 &+ \kappa \left(\fint_{B_{\varrho}(x_0)} \left|u(y)- \bb_{B_{\varrho}(x_0)}[u] \right|\d y+\varrho \left|\Aa_{B_{\varrho}(x_0)}[u] +\gamma_{B_{\varrho}(x_0)}[u]\Id \right|\right).
	\end{align*}
Thanks to Proposition \ref{propo:HigerOrderDiesOnBall} and the fact that $x_0\notin \Theta_u$ we get
	\[
	\lim_{\varrho\rightarrow 0}  \varrho^2  |s_{K_{\varrho}(x_0)}[u]|=0.
	\]
The other quantities can be treated similarly by applying the previous argument to the linear application $\mathcal{R}_{B_{\varrho}(x_0)} $ where we have replaced $\Aa_{B_{\varrho}(x_0)}$ with $\Aa_{K_{\varrho}(x_0)}$
	\begin{align*}
	\bar{\mathcal{R}}_{\varrho,x_0}[u]:=&(\ss_{B_{\varrho}(x_0)}[u]\cdot (y-x_0))(y-x_0)-\frac{|y-x_0|^2}{2}\ss_{B_{\varrho}(x_0)}[u]\\
	&+(\Aa_{K_{\varrho}(x_0)}[u] +\gamma_{B_{\varrho}(x_0)}[u]\Id)(y-x_0)+\bb_{B_{\varrho}(x_0)}[u],
	\end{align*}
	and to the linear application $\mathcal{R}_{B_{\varrho}(x_0)} $ where we have replaced $\gamma_{B_{\varrho}(x_0)}$ with $\gamma_{K_{\varrho}(x_0)}$
		\begin{align*}
	\bar{\mathcal{R}}_{\varrho,x_0}[u]:=&(\ss_{B_{\varrho}(x_0)}[u]\cdot (y-x_0))(y-x_0)-\frac{|y-x_0|^2}{2}s_{B_{\varrho}(x_0)}[u]\\
	&+(\Aa_{B_{\varrho}(x_0)}[u] +\gamma_{K_{\varrho}(x_0)}[u]\Id)(y-x_0)+\bb_{B_{\varrho}(x_0)}[u].
	\end{align*}
\end{proof}

\section{Appendix}\label{sct:appendix}
 
We here provide a proof for the Proposition \ref{propo:HigerOrderDies}. In order to do this we set up the lighter notation 	
\begin{align*}
\Aa_{\varrho,x}:=\Aa_{B_{\varrho}(x)}, \ \ \ \ &\gamma_{\varrho,x}:=\gamma_{B_{\varrho}(x)}\\
\ss_{\varrho,x}:=\ss_{B_{\varrho}(x)}, \ \ \ \ &
\bb_{\varrho,x}:=\bb_{B_{\varrho}(x)}.
\end{align*}
We can Moreover, in this case we can explicitely compute the above quantities.
 
	\begin{lemma}\label{lemAp:Compute}
	It holds
		\begin{align*}
		\ss_{\varrho,x}[u]&= \frac{1}{(n-1) \omega_n \varrho^{n+1}}\int_{\partial B_{\varrho}(x)} \left[u(y)-n\left(\frac{  y-x }{| y-x |} \cdot u(y) \right)\frac{ y-x }{|y-x|}\right]\d\H^{n-1}(y)\\
		\Aa_{\varrho,x}[u]&=\frac{1}{2\omega_n \varrho^n}\int_{\partial B_{\varrho}(x)} \left[u(y)\otimes\frac{ y-x }{| y-x |}-\frac{ y-x }{| y-x |}\otimes u(y) \right]\d\H^{n-1}(y)\\
		\gamma_{\varrho,x}[u]&=\frac{1}{n\omega_n \varrho^n} \int_{\partial B_{\varrho}(x)} \left( u\cdot \frac{ y-x }{| y-x |}\right) \d\H^{n-1}(y)\\
		\bb_{\varrho,x}[u]&= \frac{1}{n\omega_n \varrho^{n-1}}\int_{\partial B_{\varrho}(x)} u(y)\d\H^{n-1}(y)+\frac{n-2}{2n}\ss_{\varrho,x}[u]\varrho^2.
		\end{align*}
	\end{lemma}
\begin{proof}
For $\Aa_{\varrho,x}$, $\gamma_{\varrho,x}$ it is enough to translate the definition given in \eqref{eqn:A}, \eqref{eqn:T}. Also $\bb_{\varrho,x}$ can be obtained by confronting it with \eqref{eqn:d} and by observing that
	\begin{align*}
	\tau_{B_{\varrho}(x)}=&\frac{1}{n\omega_n \varrho^{n-1}} \int_{\partial B_{\varrho}(x)} \left[\frac{|y-x|^2}{2}\Id -  (y-x) \otimes (y-x) \right]\d\H^{n-1}(y)\\
	=&\frac{\varrho^2}{n\omega_n \varrho^{n-1}} \int_{\partial B_{\varrho}(x)} \left[\frac{1}{2}\Id -  \frac{(y-x)}{|y-x|} \otimes \frac{(y-x)}{|y-x|} \right]\d\H^{n-1}(y)\\
	=&\frac{\varrho^2}{2}\Id- \frac{\varrho^2}{n\omega_n} \int_{\partial B }  z\otimes z \d\H^{n-1}(z)=\frac{\varrho^2}{2}\Id- \frac{\varrho^2}{n}\Id=\varrho^2\frac{(n-2)}{2n}\Id
	\end{align*}
since
		\begin{align*}
		\int_{\partial B}  z\otimes z\d\H^{n-1}(z)=|B|\Id.
		\end{align*}
We also compute $\ss_{\varrho,x}$ starting from \eqref{eqn:s} and by recalling also Remark \ref{rmk:CurvBalls} :
	\begin{align*}
	\ss_{\varrho,x}=	&\frac{1}{(n-1)\omega_n \varrho^n} \int_{\partial B_{\varrho}(x)}u(y)^t \nabla^{\mathrm{Tan} }\left(\frac{y-x}{|y-x|}\right)  \d\H^{n-1}(y)\\
    &-\frac{1}{(n-1)\omega_n \varrho^n} \int_{\partial B_{\varrho}(x)} \frac{(n-1)}{\varrho} \left(\frac{y-x}{|y-x|} \cdot u(y)\right)\frac{y-x}{|y-x|} \d\H^{n-1}(y)\\
	=&\frac{1}{(n-1)\omega_n \varrho^{n+1}} \int_{\partial B_{\varrho}(x)}u(y)^t\left[\Id - \frac{y-x}{|y-x|}\otimes  \frac{y-x}{|y-x|}\right]  \d\H^{n-1}(y)\\
	&-  \frac{1}{(n-1)\omega_n \varrho^{n+1}} \int_{\partial B_{\varrho}(x)} (n-1) \left(\frac{y-x}{|y-x|} \cdot u(y)\right)\frac{y-x}{|y-x|} \d\H^{n-1}(y)\\
	=&\frac{1}{(n-1)\omega_n \varrho^{n+1}} \int_{\partial B_{\varrho}(x)}\left[u(y)- \left(u(y)\cdot \frac{y-x}{|y-x|}\right) \frac{y-x}{|y-x|}\right]  \d\H^{n-1}(y)\\
	&-  \frac{1}{(n-1)\omega_n \varrho^{n+1}} \int_{\partial B_{\varrho}(x)} (n-1) \left(\frac{y-x}{|y-x|} \cdot u(y)\right)\frac{y-x}{|y-x|} \d\H^{n-1}(y)\\
		=&\frac{1}{(n-1)\omega_n \varrho^{n+1}} \int_{\partial B_{\varrho}(x)} \left[u(y) - n \left(\frac{y-x}{|y-x|} \cdot u(y)\right)\frac{y-x}{|y-x|}\right] \d\H^{n-1}(y).
	\end{align*}
\end{proof}
The next Proposition allows us to re-write $\mathrm{A}_{\varrho,x}$, $\gamma_{\varrho,x}$, $\mathrm{s}_{\varrho,x}$ in terms of a non-local interaction. These computations have been done by mimicking the approach in \cite{ambrosio1997fine}, \cite{kohn1980new}.
\begin{proposition}\label{prop:NonLocalRepr}
Let $u\in \BDD(\R^n)$. Then, for any $\tau>0$, and $0<\varrho<\tau$ it holds that
    \begin{align}
        \Aa_{\varrho,x}[u]&=-\frac{1}{\omega_n}\int_{B_{\tau}(x)\setminus B_{\varrho}(x)} \frac{\Gamma(y-x)}{|y-x|^{n}}\d \E_d u(y)+\Aa_{\tau,x}[u]\label{eqn:antisym}\\
        \gamma_{\varrho,x}[u]\Id &=-\frac{1}{(n-1)\omega_n}\int_{B_{\tau}(x)\setminus B_{\varrho}(x)}\frac{\Xi(y-x)}{|y-x|^{n}}\d \E_d u(y)+\gamma_{\tau,x}[u]\Id\label{eqn:diver}\\
        \ss_{\varrho,x}[u]&= -\frac{1}{(n-1)\omega_n} \int_{B_{\tau}(x) \setminus B_{\varrho}(x)} \frac{\Upsilon(y-x)}{|y-x|^{n+1}} \d \E_du(y)+\ss_{\tau,x}[u]\label{eqn:scdorder}
    \end{align}
where $\Gamma(z),\Xi(z):\mathbb{M}^{n\times n}_{\mathrm{sym}_0} \rightarrow \mathbb{M}^{n\times n} $ are the zero homogeneus tensors acting as
    \begin{align*}
    \Gamma(z)M&:=\frac{ Mz  \otimes z}{|z|^2}   - \frac{  z\otimes M z}{|z|^2}  \\
    \Xi(z)M&:=\left( \frac{(z^t M z)}{|z|^2}-\frac{\trace(M)}{n}\right) \Id  
    \end{align*}
and  
$\Upsilon(z):\mathbb{M}^{n\times n}_{\mathrm{sym}_0}\rightarrow \R^n$ is the  zero homogeneus tensor acting as
    \begin{align*}
        \Upsilon(z) M&:=2\frac{M z}{|z|}-(n+2) \frac{(z^tMz)}{|z|^2} \frac{z}{|z|}+\trace(M) \frac{z}{|z|}.
    \end{align*}
\end{proposition}
\begin{proof}
Set, without loss of generality, $x=0$. Recall that for $f:\R^n \rightarrow\R$, $M:\R^n\rightarrow\mathbb{M}^{n\times n}$ the following chain rule formula holds:
\[
\dive(f(y) M(y))= f(y) \dive(M(y))+M(y)\nabla f(y).
\]
We now split the proof in three steps.\\
\smallskip

\textbf{Step one:} \textit{Proof of \eqref{eqn:antisym}}. Define the trace free symmetric matrix-valued map
    \[
    \psi_{i j}(y):=\frac{y_j (e_i\odot y)- y_i (e_j\odot y) }{\omega_n |y|^{n+2}}\in C^{\infty}(\R^n\setminus\{0\};\mathbb{M}^{n\times n}_{sym_0})
    \]
Then we claim that
    \begin{equation}\label{solves}
        \left\{
        \begin{array}{rl}
            (\mathbb{E}_d\left[\frac{y}{|y|}\right] u)\cdot \psi_{i j}(y) &=\frac{(u\otimes \frac{y}{|y|}-\frac{y}{|y|}\otimes u)_{i j}}{2\omega_n|y|^n}\\
            \E_d^* \psi_{i j} &=0 \ \ \ \text{on $\R^n\setminus \{0\}$}
        \end{array}
        \right.
    \end{equation}
Indeed, since $\psi_{i,j}$ is symmetric and trace free we can compute the adjoint operator
    \begin{align*}
    \E_d^* \psi (y) &=\dive(\psi_{ij}(y))
    \end{align*}
where recall that $\dive(M)$ is meant to be the row-divergence. Henceforth, since
    \begin{align*}
        \dive(y_h (e_r\otimes y)) &=  y_h(n+1) e_r \\
        \dive(y_h (y \otimes e_r)) &=y \delta_{hr}+y_h e_r\\
        \dive(y_h (e_r\odot y))=&y_h(n+2) e_r +y \delta_{hr}
    \end{align*}
we have 
\begin{align*}
    \dive\left( \frac{y_j}{|y|^{n+2}}  (e_i\odot y) \right)=&\frac{1}{|y|^{n+2}} [y_j(n+2) e_i +y \delta_{ij}]-\frac{(n+2)}{|y|^{n+4}}( y_j y_i y + e_i |y|^2).
\end{align*}
The above expression is symmetric in $i,j$ and thus
    \begin{align*}
         \dive(\psi_{ij}(y) & \dive\left( \frac{y_j}{|y|^{n+2}}  (e_i\odot y) \right)- \dive\left( \frac{y_i}{|y|^{n+2}}  (e_j\odot y) \right)=0.
    \end{align*}
Moreover
    \begin{align*}
       \left(\mathbb{E}_d\left[\frac{y}{|y|}\right]u\right)\cdot \psi_{i,j}(y)&=\frac{1}{\omega_n |y|^{n+2}} \left[u\odot \frac{y}{|y|}- \frac1n \left(u\cdot \frac{y}{|y|} \right)\Id\right]\cdot \left(y_j (e_i\odot y)- y_i (e_j\odot y)\right)\\
        &=\frac{1}{\omega_n |y|^{n+2}} \left(u\odot \frac{y}{|y|}\right)\cdot \left(y_j (e_i\odot y)- y_i (e_j\odot y)\right)
    \end{align*}
and
    \begin{align*}
        \left(u\odot \frac{y}{|y|}\right)\cdot (y_j (e_i\odot y)- y_i (e_j\odot y))&=  y_j  \left(u\odot \frac{y}{|y|}\right)\cdot  \left(e_i\odot y \right) - y_i \left(u\odot  y \right)\cdot  \left(e_j\odot\frac{y}{|y|}\right) \\
        &=\frac{y_j}{2|y|}(u_i|y|^2 +(u\cdot y)y_i) -\frac{y_i}{2|y|}(u_j|y|^2 +(u\cdot y)y_j) \\
        &= \frac{|y|^2}{2}\left(u\otimes \frac{y}{|y|}-\frac{y}{|y|}\otimes u\right)_{ij}  
    \end{align*}
since
\begin{align*}
     y_j  \left(u\odot \frac{y}{|y|}\right)\cdot  \left(e_i\odot y \right)&= \frac{y_j}{4|y|}\left(u\otimes y+y\otimes u\right)\cdot \left(e_i\otimes y+y\otimes e_i\right) \\
    =&\frac{y_j}{2|y|}(u_i|y|^2 +(u\cdot y)y_i)
\end{align*}
which finally prov\eqref{solves}. Finally note that 
\begin{equation}\label{eqn:adjointPsi}
    \psi_{ij}(y) \cdot M= (e_i\otimes e_j) \cdot \frac{\Gamma(y) M}{\omega_n |y|^n}, \ \ \text{for all $M\in \mathbb{M}^{n\times n}_{\mathrm{sym}_0}$}.
\end{equation} 
With such a $\psi_{ij}$ at hand we observe that, for $0<\varrho<\tau$ we have thanks to \eqref{eqn:adjointPsi},\eqref{eqn:GGformula} and \eqref{solves} that
    \begin{align*}
    (e_i\otimes e_j)\cdot \int_{B_{\tau}(0)\setminus B_{\varrho}(0)} \frac{\Gamma(y)}{\omega_n|y|^{n+2}}\d \E_d u(y)  = &\int_{B_{\tau}(0)\setminus B_{\varrho}(0)} \psi_{ij}(y)\cdot \d \E_d u(y)\\
    =&\int_{\partial (B_{\tau}(0)\setminus B_{\varrho}(0)) } (\mathbb{E}_d [\nu] u)\cdot \psi_{ij}(y)\d\H^{n-1}(y)\\
    =&\int_{\partial B_{\tau}(0)}\left(\mathbb{E}_d \left[\frac{y}{|y|}\right] u\right)\cdot \psi_{ij}(y)\d\H^{n-1}(y)\\
    &-\int_{\partial B_{\varrho}(0)) } \left(\mathbb{E}_d \left[\frac{y}{|y|}\right] u\right)\cdot \psi_{ij}(y)\d\H^{n-1}(y)\\
    =&\frac{1}{2\omega_n \tau^n}\int_{\partial B_{\tau}(0)}(u\otimes \nu(y)-\nu(y)\otimes u)_{ij}\d\H^{n-1}(y)\\
    &-\frac{1}{2\omega_n \varrho^n}\int_{\partial B_{\varrho}(0)}(u\otimes \nu(y)-\nu(y)\otimes u)_{ij}\d\H^{n-1}(y)\\
    =&(e_i\otimes e_j)\cdot (\mathrm{A}_{\tau,x}[u]-\mathrm{A}_{\varrho,x}[u])
   \end{align*}
\smallskip

\textbf{Step two:} \textit{Proof of \eqref{eqn:diver}}. Define the trace free symmetric-matrix valued map
    \[
    \psi_{ij}(y):=  \frac{\delta_{ij}}{(n-1)\omega_n |y|^{n }}\left[\frac{y}{|y|}\otimes \frac{y}{|y|}-\frac{\Id}{n}\right]
    \]
and again notice that
    \begin{equation}\label{solves2}
        \left\{
        \begin{array}{rl}
        \left(\mathbb{E}_d\left[\frac{y}{|y|}\right]u\right)\cdot  \psi_{ij}(y)&=\delta_{ij}\frac{\left(u\cdot\frac{y}{|y|}\right)}{n\omega_n|y|^n}   \\
            \E_d^*  \psi_{ij} &=0 \ \ \text{on $\R^n\setminus \{0\}$ }
        \end{array}\right.
    \end{equation}
Indeed
    \begin{align*}
     \dive\left(\frac{y\otimes y}{|y|^{n+2}}\right)=&-(n+2)\frac{y}{|y|^{n+2}} +(n+1)\frac{y}{|y|^{n+2}}=\frac{y}{|y|^{n+2}} \\
\dive\left(\frac{1}{n|y|^n}\Id \right)=&-\frac{y}{|y|^{n+2}}
    \end{align*}
giving for $y\neq 0$
    \[
    \E_d^*  \psi_{ij}(y)=0.
    \]
Moreover
    \begin{align*}
    \left(\mathbb{E}_d\left[\frac{y}{|y|}\right]u\right)\cdot  \psi_{ij}(y)=&\frac{\delta_{ij}}{(n-1)\omega_n |y|^n}\left(u\odot \frac{y}{|y|}\right) \cdot \left(\frac{y}{|y|}\otimes \frac{y}{|y|}-\frac{1}{n}\Id\right)\\
    =&\frac{\delta_{ij}}{(n-1)\omega_n |y|^n}\left[\left(u\odot \frac{y}{|y|}\right)\cdot \frac{y}{|y|}\otimes \frac{y}{|y|} - \frac{1}{n}\left(u\cdot \frac{y}{|y|}\right)\right]\\
      =&\frac{\delta_{ij}}{n\omega_n |y|^n} \left(u\cdot \frac{y}{|y|}\right).
    \end{align*}
which finally proves \eqref{solves2}. Notice that
    \begin{equation}\label{adjoinpsitrace}
    ( \psi_{ij}(y)\cdot M)=(e_i\otimes e_j)\cdot \frac{\Xi(y)M}{(n-1)\omega_n |y|^n}, \ \ \ \text{for any $M\in \mathbb{M}_{\mathrm{sym}_0}^{n\times n}$}
    \end{equation}
 and henceforth again due to \eqref{eqn:GGformula}, \eqref{solves2} and \eqref{adjoinpsitrace} we get
    \begin{align*}
        (e_i\otimes e_j)\cdot \frac{1}{(n-1)\omega_n }\int_{B_{\tau}(0)\setminus B_{\varrho}(0)} \frac{\Xi(y)}{|y|^{n}}\d \E_d u(y)=&\int_{B_{\tau}(0)\setminus B_{\varrho}(0)}  \psi_{ij}(y) \cdot \d \E_d u(y) \\
        =&\int_{\partial B_{\tau}(0) } \left(\mathbb{E}_d\left[\frac{y}{|y|}\right] u \right)\cdot  \psi_{ij}(y) \d \H^{n-1} \\
        &- \int_{\partial B_{\varrho}(0) } \left(\mathbb{E}_d\left[\frac{y}{|y|}\right] u \right)\cdot  \psi_{ij}(y) \d \H^{n-1}\\
         =&(\gamma_{\tau,x }[u]-\gamma_{\varrho,x }[u] ) \delta_{ij}.
    \end{align*}
\smallskip

\textbf{Step Three:} \textit{Proof of \eqref{eqn:scdorder}}. Given the trace free symmetric matrix defined as
    \[
    \psi_k(y):=\frac{1}{(n-1)\omega_n|y|^{n+2}}\left( 2 e_k\odot y  -  (n+2) y_k  \frac{y}{|y|}\otimes \frac{y}{|y|}+y_k \Id\right)  
    \]
we have
\begin{align*}
2    \dive \left(\frac{ e_k\odot y}{|y|^{n+2}}\right)=& (n+1) \frac{e_k}{|y|^{n+2}}-\frac{(n+2)}{ |y|^{n+4}} (e_k |y|^2 + y_ky)\\
    -(n+2) \dive\left(\frac{y_k}{|y|^{n+2}} \frac{y}{|y|}\otimes \frac{y}{|y|}\right)=&-n (n+2)y_k\frac{y}{|y|^{n+4}}+\frac{(n+2)^2}{|y|^{n+4}} y_k   y\\
    \dive\left(\frac{y_k}{|y|^{n+2}} \Id\right)=&\frac{e_k}{|y|^{n+2}}-\frac{(n+2)}{|y|^{n+4}}y_k y.
\end{align*}
Hence we immediately see that
\begin{align*}
     \E_d^*  \psi_k(y) =&\dive(\psi_k)=0.
\end{align*}
Also
\begin{align*}
 \psi_k(y)\cdot  \mathbb{E}_du\left[\frac{y}{|y|}\right]  =&\frac{1}{(n-1)\omega_n|y|^{n+2}}\left( 2 e_k\odot y  -  (n+2) y_k  \frac{y}{|y|}\otimes \frac{y}{|y|}+y_k \Id\right)   \cdot \left(u\odot \frac{y}{|y|}\right)\\
 &=\frac{1}{(n-1)\omega_n|y|^{n+2}}\left(   u_k |y| + \frac{y_k}{|y|} (u\cdot y) - (n+2)\frac{y_k}{|y|} (u\cdot y)  + \frac{y_k}{|y|} (u\cdot y)\right)\\
 &=\frac{1}{(n-1)\omega_n|y|^{n+1}}\left(   u_k     -  n \frac{y_k}{|y|} \left(u\cdot \frac{y}{|y|}\right)   \right)\\
 &=\frac{1}{(n-1)\omega_n|y|^{n+1}} \left(\left[\Id - n \frac{y}{|y|}\otimes \frac{y}{|y|}\right]u \right)_k.
\end{align*}
Since we see that
   \[
   \frac{e_k \cdot \Upsilon(y)M}{(n-1)\omega_n |y|^{n+1}}=  \psi_k(y)\cdot M  \ \ \ \text{for any $M\in \mathbb{M}_{\mathrm{sym}_0}^{n\times n}$}
   \]
we conclude again as in Step one and two. 
\end{proof}
\begin{proposition}\label{propo:keydecayment}
If $x\notin \Theta_u$ then
    \[
   \lim_{\varrho\rightarrow 0} \varrho^2|\mathrm{s}_{\varrho,x}[u]|=\lim_{\varrho\rightarrow 0}\varrho|\mathrm{A}_{\varrho,x}[u]|=\lim_{\varrho\rightarrow 0}\varrho|\gamma_{\varrho,x}[u]|=0.
    \]
\end{proposition}
\begin{proof}
We assume $x=0$ without loss of generality and we also write $\mathrm{A}_{\varrho}$, $\gamma_{\varrho}$, $\mathrm{s}_{\varrho}$ in place of $\mathrm{A}_{\varrho,0}$, $\gamma_{\varrho,0}$, $\mathrm{s}_{\varrho,0}$. We notice that, the condition $0\notin \Theta_u$ can be translated into
    \[
   \lim_{\tau\rightarrow 0^+} \sup_{t\in (0,\tau)} \frac{|\E_du|(B_{t}(0))}{t^{n-1}}=0.
    \]
We start by showing the key fact
    \begin{align}\label{zerokill}
       \lim_{\tau\rightarrow 0^+} \limsup_{\varrho \rightarrow 0^+} \varrho \int_{B_{\tau}(0)\setminus B_{\varrho}(0)} \frac{1}{|y|^n} \d|\E_d u|(y)=0
    \end{align}
This comes as a consequence of the layer-cake representation for the measure
    \[
    \mu:= |\E_d u|\llcorner_{B_{\tau}(0)\setminus B_{\varrho}(0)}
    \]
since
\begin{align*}
    \int_{B_{\tau}(0)\setminus B_{\varrho}(0)} \frac{1}{|y|^n} \d|\E_d u|(y)&=\int_{\R^n} \frac{1}{|y|^n}\d \mu(y)=\int_0^{+\infty}\mu\left(\left\{ \frac{1}{|y|^n}>t\right\}\right)\d t\\
    &=n\int_0^{+\infty} \frac{\mu(\{|y|\leq s\}}{s^{n+1}} \d s\\
    &=n\int_{\varrho}^{\tau} \frac{|\E_d u|(B_s(0)\setminus B_{\varrho}(0))}{s^{n+1}} \d s + \frac{|\E_d u|(B_\tau (0)\setminus B_{\varrho}(0))}{\tau^{n}} \\
\end{align*}
Moreover
    \begin{align*}
      n\int_{\varrho}^{\tau} \frac{|\E_d u|(B_s(0)\setminus B_{\varrho}(0))}{s^{n+1}} \d s &\leq n \sup_{t\in (0,\tau)}   \frac{|\E_d u|(B_t(0))}{t^{n-1}}\int_{\varrho}^{\tau} s^{-2}\d s\\
      &=n\sup_{t\in (0,\tau)}   \frac{|\E_d u|(B_t(0))}{t^{n-1}} \left(\frac{1}{\varrho}-\frac{1}{\tau}\right).
    \end{align*}
Therefore
    \begin{align*}
        \limsup_{\varrho\rightarrow 0^+} \varrho \int_{B_{\tau}(0)\setminus B_{\varrho}(0)} \frac{1}{|y|^n} \d|\E_d u|(y)\leq n\sup_{t\in (0,\tau)}   \frac{|\E_d u|(B_t(0))}{t^{n-1}}
    \end{align*}
yielding \eqref{zerokill}, since $0\notin\Theta_u$. By now using the representation \eqref{eqn:antisym}, \eqref{eqn:diver}, \eqref{eqn:scdorder} we see that  (by definition)
    \[
    |\Gamma(y)|   + |\Xi(y)| + |\Upsilon(y)|\leq C
    \]
giving that
    \begin{align*}
        \varrho^2|\mathrm{s}_{\varrho}[u]|+\varrho|\mathrm{A}_{\varrho}[u]|+\varrho|\gamma_{\varrho}[u]|\leq C\varrho\int_{B_{\tau}(0)\setminus B_{\varrho}(0)} \frac{1}{|y|^n}\d|\E_d u|(y)+\varrho \kappa(\tau)
    \end{align*}
where $\kappa(\tau)$ is a constant depending on $\tau$ only. We now first take the limit in $\varrho$ and then in $\tau$, which by exploiting \eqref{zerokill}, achieves the proof.
\end{proof}
We are now ready to prove Proposition \ref{propo:HigerOrderDiesOnBall}.
\begin{proof}[Proof of Proposition \ref{propo:HigerOrderDiesOnBall}]
Without loss of generality set $x=0$. Relation \eqref{eqn:decaysOnBall} comes from Proposition \ref{propo:keydecayment}. We now set $\mathcal{R}_{\varrho}=\mathcal{R}_{B_{\varrho}(0)}$ and we first invoke Poincaré ineqality \ref{prop:poincare} to see that
    \[
    \frac{1}{\omega_n\varrho^n}\int_{B_{\varrho}(0)}|u(y)-\mathcal{R}_\varrho [u](y)|\d y\leq c  \frac{|\E_d u|(B_{\varrho}(0))}{\varrho^{n-1}} 
    \]
for a universal constant $c$ independent of $\varrho$. Moreover
  \begin{align*}  \fint_{B_{\varrho}(0)}|u(y)-\mathcal{R}_\varrho [u](y)|\d y\geq& \fint_{B_{\varrho}(0)}|u(y)-d_{\varrho}[u]|\d y\\
  &-\kappa \left[ \fint_{B_{\varrho}(0)}|R_{\varrho}[u] y|\d y + \fint_{B_{\varrho}(0)}|\gamma_{\varrho}[u] y|\d y\right]\\
  &- \kappa \left|\fint_{B_{\varrho}(0)} [(s_{\varrho}[u]\cdot y)y - s_{\varrho}[u]|y|^2]\d y\right|  \\
  \geq & \fint_{B_{\varrho}(0)}|u(y)-d_{\varrho}[u]|\d y-\kappa \left[\varrho|R_{\varrho}[u]|+\varrho|\gamma_{\varrho}[u]|\right].
  \end{align*}
 for a universal constant $\kappa>0$. Henceforth
    \begin{align*}
       \fint_{B_{\varrho}(0)}|u(y)-d_{\varrho}[u]|\d y&\leq \fint_{B_{\varrho}(0)}|u(y)-\mathcal{R}_\varrho [u](y)|\d y+\kappa \left[\varrho|R_{\varrho}[u]|+\varrho|\gamma_{\varrho}[u]|+\varrho^2 |s_{\varrho}[u]|\right]\\
           &\leq \kappa'\left[\frac{|\E_d u|(B_{\varrho}(0))}{\varrho^{n-1}}+ \varrho|R_{\varrho}[u]|+\varrho|\gamma_{\varrho}[u]|+\varrho^2 |s_{\varrho}[u]|\right].
    \end{align*}
By taking the limit as $\varrho\rightarrow 0$ and by exploiting $x\notin \Theta_u$, together with Proposition \ref{propo:keydecayment} we achieve the proof.
\end{proof}

\bibliography{references}
\bibliographystyle{plain}

\end{document}